\documentclass[11pt]{amsart}
\usepackage{fullpage}
\usepackage{color}
\usepackage{graphicx,psfrag} 
\usepackage{color} 
\usepackage{tikz}
\usepackage{pgffor}
\usepackage{hyperref}
\usepackage{todonotes}

\usepackage{wasysym,amssymb} 
	\usepackage{subfigure}

\usepackage{perpage}	 

\MakePerPage{footnote}	 

\usepackage[normalem]{ulem} 

\newtheorem*{claim*}{Claim}

\newtheorem{theorem}{Theorem}[section]
\newtheorem{lemma}[theorem]{Lemma}
\newtheorem{corollary}[theorem]{Corollary}
\newtheorem{proposition}[theorem]{Proposition}

\newtheorem{assumption}[theorem]{Assumption}
\newtheorem{algorithm}[theorem]{Algorithm}

\theoremstyle{definition}
\newtheorem{definition}[theorem]{Definition}
\newtheorem{example}[theorem]{Example}
\newtheorem{property}[theorem]{Property}
\newtheorem{question}[theorem]{Question}

\theoremstyle{remark}
\newtheorem{remark}[theorem]{Remark}
\newtheorem{fact}[theorem]{Fact}

\numberwithin{equation}{section}

\definecolor{gray}{rgb}{.5,.5,.5}

\definecolor{black}{rgb}{0,0,0}

\definecolor{blue}{rgb}{0,0,1}
\def\blue{\color{blue}}
\definecolor{red}{rgb}{1,0,0}
\def\red{\color{red}}
\definecolor{green}{rgb}{0,1,0}

\definecolor{yellow}{rgb}{1,1,.4}

\newcommand{\etalchar}[1]{$^{#1}$}

\makeatletter
\def\@tocline#1#2#3#4#5#6#7{\relax
  \ifnum #1>\c@tocdepth 
  \else
    \par \addpenalty\@secpenalty\addvspace{#2}%
    \begingroup \hyphenpenalty\@M
    \@ifempty{#4}{%
      \@tempdima\csname r@tocindent\number#1\endcsname\relax
    }{%
      \@tempdima#4\relax
    }%
    \parindent\z@ \leftskip#3\relax \advance\leftskip\@tempdima\relax
    \rightskip\@pnumwidth plus4em \parfillskip-\@pnumwidth
    #5\leavevmode\hskip-\@tempdima
      \ifcase #1
       \or\or \hskip 1em \or \hskip 2em \else \hskip 3em \fi%
      #6\nobreak\relax
    \dotfill\hbox to\@pnumwidth{\@tocpagenum{#7}}\par
    \nobreak
    \endgroup
  \fi}
\makeatother

\DeclareMathAlphabet{\mathpzc}{OT1}{pzc}{m}{it}

\begin{document}

\title{Moduli spaces of ten-line arrangements with double and triple points}

\author{Meirav Amram}
\address{Department of Mathematics and Computer Science, Bar-Ilan University, Ramat-Gan, 52900, Israel\newline and Shamoon College of Engineering,
Bialik/Basel Sts., Beer-Sheva 84100, Israel}
\email{meirav@macs.biu.ac.il, meiravt@sce.ac.il}

\author{Moshe Cohen}
\address{Department of Mathematics and Computer Science, Bar-Ilan University, Ramat Gan 52900, Israel}
\email{cohenm10@macs.biu.ac.il}

\author{Mina Teicher}
\address{Department of Mathematics and Computer Science, Bar-Ilan University, Ramat Gan 52900, Israel}
\email{teicher@macs.biu.ac.il}

\author{Fei Ye}
\address{Department of Mathematics, The University of Hong Kong, Hong Kong}
\email{fye@maths.hku.hk}

\date{\today}

\begin{abstract}
Two arrangements with the same combinatorial intersection lattice but whose complements have different fundamental groups are called a Zariski pair.  This work finds that there are at most nine such pairs amongst all ten line arrangements whose intersection points are doubles or triples. 
 This result is obtained by considering the moduli space of a given configuration table which describes the intersection lattice.  A complete combinatorial classification is given of all arrangements of this type under a suitable assumption, producing a list of  seventy-one described in a table, 
 most of which do not explicitly appear in the literature.  This list also includes other important counterexamples:  nine combinatorial arrangements that are not geometrically realizable.
\end{abstract}



\maketitle

\tableofcontents

 

\textbf{Acknowledgements.} 
  This research was partially supported by the Minerva Foundation of Germany through the Emmy Noether Institute.  The second author was partially supported by the Oswald Veblen Fund of the Institute of Advanced Study in Princeton.

\section{Introduction}
\label{sec:Intro}



  A {\it line arrangement} $\mathcal{A}$ 
in $\mathbb{CP}^2$ is a finite collection of projective lines. We define the complement of $\mathcal{A}$ as $\mathbb{CP}^2\setminus \bigcup_{L\in \mathcal{A}}L$ and denote it as $M(\mathcal{A})$. The set $L(\mathcal{A})=\{\bigcap_{i\in S}L_i | S\subseteq\{1, 2, \dots, k\}\}$ partially ordered by reverse inclusion is called the {\em intersection lattice} of $\mathcal{A}$. A set with such an intersection lattice structure is often called a \emph{configuration}. 
 Two line arrangements $\mathcal{A}_1$ and $\mathcal{A}_2$ are lattice isomorphic, denoted as $\mathcal{A}_1\sim \mathcal{A}_2$,  if up to a permutation on the labels of the lines their lattices are the same. See the textbook \cite{OrlTer} by Orlik and Terao for a brief introduction to the history and general theory of arrangements of hyperplanes.


In 1980, Orlik and Solomon \cite{OrlSol} proved that the cohomology
algebra of $M(\mathcal{A})$ is determined by the combinatorics of $L(\mathcal{A})$. This led to a ``conjecture'' that homotopy invariants of $M(\mathcal{A})$ are combinatorial invariants.

Conversely,  Falk studied  whether $L(\mathcal{A})$ is a homotopy invariant. In \cite{Falk}, he presented a pair of central arrangements
in $\mathbb{C}^3$ with different underlying lattices but homotopy equivalent complements. However, for line arrangements, Jiang and Yau \cite{JiaYau93} showed that homeomorphic equivalence implies lattice isomorphism.  
Towards the ``conjecture,'' it was  proved for line arrangements by Jiang and Yau \cite{JiaYau93} and for hyperplane arrangement by 
Randell \cite{Ran} independently that if two arrangements are lattice isotopic, i.e. they are connected by a one-parameter family with constant intersection lattice, then their complements are diffeomorphic.  
  As applications,  Jiang, Wang, and Yau \cite{JiaYau94}, \cite{Wang} showed that the intersection lattices of {\em nice arrangements} and {\em simple arrangements} determine the topology of the complements. Nazir and Yoshinaga \cite{NazYos} define {\em simple $C_3$} line arrangements and show that their intersection lattices determine the topology of their complements. 

However, the ``conjecture'' is not true in general. In 1998, Rybnikov \cite{Ryb} discovered a pair of two complex arrangements of thirteen lines with fifteen triple points. The arrangements are lattice isomorphic but the fundamental groups of complements are different. We call a pair of lattice isomorphic line arrangements a {\em Zariski pair} if the fundamental groups of complements are different.  Our definition is stronger than the original definition, introduced by Artal Bartolo in \cite{AB}, which is a pair of lattice isomorphic line arrangements with different embedding type.  Rybnikov's example is the first and smallest Zariski pair so far. It is not known if there is a Zariski pair of arrangements of less than thirteen lines. It is not known if there is a Zariski pair of real line arrangements.

By studying fundamental groups, Garber, Teicher and Vishne \cite{GTV:8} proved that there is no Zariski pair of arrangements of up to eight real lines which covered a result of Fan \cite{Fan} on arrangements of six lines.

By studying moduli spaces, Nazir and Yoshinaga \cite{NazYos} proved that there is no Zariski pair of arrangements of up to eight complex lines and listed a classification of arrangements of nine lines without proof (later proved to be complete by Ye in \cite{Fei9}). The classification implies that there is also no Zariski pair of arrangements of nine lines. 

\subsection{The moduli space of an arrangement with fixed intersection lattice}
\label{subsec:geom}

The earlier program classified the moduli spaces of arrangements of nine lines \cite{Fei9}, followed by arrangements of ten lines with some multiple points of order greater than three \cite{Fei:10}.  The goal of this work is to classify the moduli spaces of arrangements of ten lines with only triple points and double points. One purpose of the classification is to search for Zariski pairs  from arrangements of ten lines.  
 
Comparing fundamental groups of complements of line arrangements is very hard. Rybnikov \cite{Ryb} distinguished the two fundamental groups of his pair of arrangements by delicate analysis of the lower central series quotients of the groups.  Following Rybnikov's idea, Artal Bartolo, Carmona Ruber, Cogolludo Agust{\'{\i}}n, and Marco Buzun{\'a}riz present in \cite{ABC2}  an alternative proof in detail by using the Alexander module and its truncations as well as certain combinatorial invariants developed only for line arrangements with only double and triple points.  We hope  that their method can be applied to the arrangements of ten lines with only double and triple points that we produce in this present paper.

 Let  $\mathcal{A}$ be  a complex line arrangement.  We define the {\em moduli space} of line arrangements with the fixed lattice $L(\mathcal{A})$ (or simply, the moduli space of $\mathcal{A}$) as
\[\mathcal{M}_\mathcal{A}=\{\mathcal{B}\in ((\mathbb{CP}^2)^*)^n | \mathcal{B}\sim \mathcal{A})\}/ PGL(3, \mathbb{C}).\]

By Randell's Lattice-Isotopy Theorem in \cite{Ran} and Cohen and Suciu's Theorem 3.9 in  \cite{CoSuc}, we know that two arrangements in the same connected component of the moduli space, or in two complex conjugate components, respectively, are not Zasiki pairs.  See more in Section \ref{subsec:linemethod}.

One very useful result on structure of moduli spaces of line arrangements is the following lemma of Nazir and Yoshinaga.

\begin{lemma}\cite[Lemma 2.4]{NazYos}
\label{lemma:extend-irreduciblity}
Let $\mathcal{A}=\{L_1,L_2,\dots, L_n\}$ be a line arrangement in $\mathbb{CP}^2$. Assume that the line $L_n$ passes through \textbf{at most two} multiple points of the arrangement $\mathcal{A}$. Set $\mathcal{A}'=\mathcal{A}\setminus\{L_n\}$. Then the moduli space $\mathcal{M}_{\mathcal{A}}$ is irreducible if $\mathcal{M}_{\mathcal{A}'}$ is irreducible.
\end{lemma}

With the support of this lemma, we make the following reasonable assumption throughout, except for the general results in Sections \ref{sec:Background} and \ref{subsec:linemethod}.

\begin{assumption}
\label{ass:reduce}
Let $\mathcal{A}$ be a line arrangement with only double and triple points such that each line passes through at least \textbf{three} triple points.
\end{assumption}



Departing from the previous work, the techniques presented in this paper prioritize the points instead of the lines.  This perspective comes from matroid theory, which due to its more combinatorial nature serves as a better model for enumeration.

\subsection{The matroid perpective}
\label{subsec:matroid}

A hyperplane arrangement can be realized in more combinatorial language as an \textit{oriented matroid}, an object well-documented in the textbook \cite{OM} by Bj{\"o}rner, Las Vergnas, Sturmfels, White, and Ziegler.

  We investigate an arrangement of interest above by looking at its \textit{underlying matroid}.  Matroids, now stripped of their topology and appearing as purely combinatorial objects, are also heavily documented in the literature, with special introductory care taken in the textbook \cite{Oxley} by Oxley.

A \textit{matroid} is the simultaneous generalization of a matrix and a graph:  the common property is the notion of dependence, achieved by linear independence of column vectors in a matrix or by cycles of edges in a graph.  
A matroid obtained from a matrix (whose entries belong to some field) is called \emph{representable} (over this field).  
 
 The notion of representability over the field $\mathbb{R}$ or $\mathbb{C}$ corresponds to the notion of geometrically realizable in the setting above over $\mathbb{R}$ or $\mathbb{C}$, respectively.

Thus the arrangements produced in the work below may be used to find 
 \emph{forbidden minors} (also called \emph{excluded minors} or \emph{minor-minimal obstructions}) studied in infinite classes of matroids.  
%
Although it is most natural to ask about forbidden minors for the class of $\mathbb{R}$-representable matroids, this list is infinite.  Mayhew, Newman, and Whittle \cite{MNW} recently showed that for any infinite field $\mathbb{K}$ and any $\mathbb{K}$-representable matroid $N$, there is an excluded minor in $\mathbb{K}$ that has $N$ as a minor.

However, as our list produces new, relatively larger matroids whose base sets have between nineteen and twenty-five elements, they might also be useful for the study of forbidden minors of finite fields, another topic of popular interest.

A \emph{geometric matroid} represents elements from its base set by points, showing dependency geometrically:  a circuit of size three by a line through three points, a circuit of size four by a plane through four points, etc.  This is the setting that will be most useful for us below, as it naturally gives rise to line arrangements.

\begin{remark}
Assumption \ref{ass:reduce} is the right choice from the matroid perspective, as well.  Every two distinct points in a geometric matroid are independent, but the lines between them are omitted to avoid confusion.   The lines included are ones that give dependencies and so must contain at least three points.
\end{remark}

In order to describe the matroids in our investigation one would have to enumerate all of the intersection points.  However in this context we only enumerate the triple points.  For this approach we use \emph{configuration tables} following the textbook \emph{Configurations of points and lines} by Gr\"unbaum \cite{Grun}.   Examples of configuration tables can be found in Table \ref{tab:93} and throughout the paper.





A configuration table is \emph{geometrically realizable} if it is the configuration table associated to a geometrically realizable line arrangement. In Section \ref{subsec:linemethod}, we illustrate how to determine whether a given configuration table is realizable.

\subsection{Results and organization}
\label{subsec:results}

The two main results are the classification of ten-line arrangements that satisfy Assumption \ref{ass:reduce} with only double and triple points given in Theorem \ref{thm:classification} and Table \ref{tab:summaryALT} and also the discovery of nine potential Zariski pairs from this classification based on the moduli spaces of the arrangements given in Theorem \ref{thm:Zariskipairs} and Table \ref{tab:finallist}.

Aside from these, other important counterexamples are given:  nine non-geometrically realizable arrangements in Theorem \ref{thm:nongeom} and Table \ref{tab:nongeom}. 

This work is important for its own sake, as a complete list of such large arrangements cannot be found in the literature; but furthermore it will be instrumental in answering the question of whether there exist other small Zariski pairs of arrangements.

The paper is structured as follows.  The combinatorics is discussed in Section \ref{sec:Background}, giving restrictions that will be used as lemmas throughout the main proofs.   The geometric methodology is outlined in Section \ref{subsec:linemethod}, highlighting intermediate steps and results that are part of Algorithm \ref{algorithm} used to classify the moduli spaces of the arrangements.  Some examples are given that demonstrate the procedures we use.

In Section \ref{sec:9} earlier results are discussed for arrangements of nine lines with nine and ten triples.  These appear as reductions in later proofs of our arrangements of ten lines.

Our main results are outlined in Section \ref{sec:Main}, with Theorem \ref{thm:classification} giving a classification of our arrangements and Theorem \ref{thm:Zariskipairs} giving nine potential Zariski pairs.

The outline of the proof is then explained in Subsection \ref{subsec:proofs}.  The proof is presented in four parts:  Sections \ref{sec:10}, \ref{sec:13}, \ref{sec:12}, \ref{sec:11} for ten, thirteen, twelve, and eleven triples, respectively.  These are ordered based on the relative straightforwardness of the proofs, with the final section on eleven lines being the most fragmented via casework.

%
%
%
%
%

%

\section{The combinatorial data}
\label{sec:Background}


For the uninitiated reader, we review some background that sets up the proof of our classification Theorem \ref{thm:classification} and Table \ref{tab:summaryALT}.  We then introduce some basic arrangement results and intuitions that will be used as lemmas in the main proofs below.


Let $\mathcal{A}$ be an arrangement of ten lines, and let $n_i$ be the number of intersection points of order $i$.  We assume throughout that $n_i=0$ for $i>3$, as the other cases have been handled already by the work of \cite{Fei:10}, where $n_4\neq0$.

\begin{lemma}
\label{lemma:basic} Assume that $\mathcal{A}$ satisfies Assumption \ref{ass:reduce}.
Let $\ell_i$ be the number of lines with exactly $i$ triple points on it.  Then we have
\begin{eqnarray}
\label{eqn:basic1}
\ell_3+\ell_4 & = & 10\\
\label{eqn:basic2}
3\ell_3+4\ell_4 & = & 3n_3,
\end{eqnarray}
yielding the results of Table \ref{tab:triplenumbers}.
\begin{table}[htbp]
\begin{tabular}{|c|c|c|c|c|c|}
\hline
$n_3$ & $\#$ of triples & 10 & 11 & 12 & 13 \\
\hline
\hline
$\ell_3$ & $\#$ of lines with three triples & 10 & 7 & 4 & 1 \\
\hline
$\ell_4$ & $\#$ of lines with four triples & 0 & 3 & 6 & 9 \\
\hline 
\end{tabular}
\caption{The four cases for the number of triple points in a ten-line arrangement.}
\label{tab:triplenumbers}
\end{table}
\end{lemma}

\begin{proof}
Assumption \ref{ass:reduce} gives that $\ell_i= 0$ for $i\leq 2$. For a line to have at least five triple points on it, there must be at least ten other lines.  Because there are only ten lines total, we have $\ell_i= 0$ for $i\geq 5$.  This gives Equation \ref{eqn:basic1}.

Since each triple point is on three lines, we need a total of $3n_3$ lines \emph{counted with multiplicities} to form $n_3$ triple points. On the other hand, if a line passes through $i$ triple points, then it contributes $i$ lines.  This gives Equation \ref{eqn:basic2}.
\end{proof}

This can also be expressed as a corollary of a more general result.

\begin{lemma}
\label{lemma:eqns}
Let $\mathcal{A}$ be an arrangement of $k$ lines with at most triple points. 
Denote by $\ell_i$ the number of lines that each passes through exactly $i$ triple points. Then the following equalities hold. 
\begin{eqnarray*}
\ell_0+\ell_1+\ell_2+\ldots+\ell_k & = & k\\
\ell_1+2\ell_2+\ldots+k\ell_k & = & 3n_3,
\end{eqnarray*}
where $n_3$ is the number of triple points in the arrangement.
\end{lemma}

There are two other useful formulae involving the number of lines and the number of points of different multiplicities.

\begin{fact}[see for instance Section 6 of \cite{Hirz}]
\label{lemma:intersection-formula}
Let $\mathcal{A}$ be an arrangement of $k$ projective lines. 
Then
\begin{equation}
\label{eqn:equality}
\dfrac{k(k-1)}{2}=\sum\limits_{i=2}^k\dfrac{i(i-1)n_i}{2}.
\end{equation}
\end{fact}

\begin{theorem}[Hirzebruch \cite{Hirz},  Equation (9)]
\label{lemma:Hirzebruch}
 Let $\mathcal{A}$ be an arrangement of $k$ projective lines. Assume that $n_k=n_{k-1}=n_{k-2}=0$. Then we have the following inequality
\label{ineqn:Hirzebruch}
\begin{equation}
n_2+\dfrac{3}{4}n_3\geq k+\sum_{i\geq 5}(2i-9)n_i.
\end{equation}
\end{theorem}


We now include several more facts and lemmas that appear in the main proof below.


\begin{fact}
\label{fact:2trip5lines}
An arrangement containing two triples must contain at least five lines.
\end{fact}

\begin{fact}
\label{fact:3col6lines}
A line containing three triples must be part of an arrangement with at least six other lines.
\end{fact}

\begin{fact}
\label{fact:4col8lines}
A line containing four triples must be part of an arrangement with at least eight other lines.
\end{fact}


\begin{lemma}
\label{lem:0col6lines}
Three triples that are not colinear must be part of an arrangement with at least six lines.
\end{lemma}

\begin{proof}
Let $\mathcal{A}=\{L_1, L_2,\dots, L_k\}$ be an arrangement in $\mathbb{CP}^2$, and recall that $\ell_i$ is the number of lines that each passes through exactly $i$ triple points. Since no three triple points of $\mathcal{A}$ are collinear, then $\ell_i=0$ for $i\geq3$. By lemma \ref{lemma:eqns}, we have $\ell_0+\ell_1+\ell_2=k$ and $\ell_1+ 2\ell_2=3n_3=9$. Since there are only three triple points and there exists a unique line passing through two distinct points, then $\ell_2\leq 3$. Therefore, $k=9+\ell_0-\ell_2\geq 6$.
\end{proof}

\begin{lemma}
\label{lem:6lines4pts}
A subarrangement of six lines can have at most four triple points.
\end{lemma}

\begin{proof} 
Since $n_i\geq 0$ for $i\geq 2$, using
the Fact \ref{lemma:intersection-formula}, we see that $n_3=\dfrac{1}{3}(15-n_2-6n_4-10n_5-15n_6)\leq 5$.  
 If $n_3=5$, then $n_2=n_4=n_5=n_6=0$. By Theorem \ref{lemma:Hirzebruch}, we would have that $n_2+\dfrac{3}{4}n_3=\dfrac{15}{4}\geq 6$, a contradiction. Therefore,  $n_3\leq 4$.
\end{proof}

\begin{lemma}
\label{lem:6lines3triples}
There is a unique combinatorial arrangement $\mathcal{B}$ of six lines with exactly three triples as shown in Figure \ref{fig:lemmas}.
\end{lemma}

\begin{proof}
Let $L_1\cap L_2\cap L_3$ be a triple point. We claim that the other two triple points must be on $L_1\cup L_2\cup L_3$.  Otherwise, let $L_4\cap L_5\cap L_6$ be another triple point apart from $L_1\cup L_2\cup L_3$, then we need one more line to form the third triple point.  Let $L_1\cap L_4\cap L_5$ be the second triple point. Since there is only one more line,  then one of the intersections $(L_2\cup L_3)\cap (L_4\cup L_5)$  must be a triple point. Up to a permutation, we may assume that $L_2\cap L_4\cap L_6$ is the third triple points.  It is not hard to check that up to a permutation, this is the unique arrangement of $6$ line with exactly $3$ triple points.
\end{proof}

\begin{lemma}
\label{lem:6lines4triples}
There is a unique combinatorial arrangement $\mathcal{C}$ called the \emph{Ceva arrangement} of six lines with exactly four triples as shown in Figure \ref{fig:lemmas}.
\end{lemma}
\begin{proof}Since there are four triple points, then each line must pass through at least one but no more than two triple points. By Lemma \ref{lemma:basic}, we know that each line passes through exactly two triple points. Let $L_1\cap L_2\cap L_3$ be a triple point. Then all the other three triple points should be on $L_1\cup L_2\cup L_3$.  Let $L_1\cap L_4\cap L_5$ be the second triple point on $L_1$.  Then $L_2\cap L_4$ or $L_2\cap L_5$ should be a triple point. Without loss of generality, we assume that $L_2\cap L_4\cap L_6$ is a triple point. Then $L_3\cap L_5\cap L_6$ must be a triple point. Hence we get a unique arrangement of 6 lines with 4 triple points.
\end{proof}

\begin{figure}[htbp]
\begin{center}
\begin{tikzpicture}

\draw[domain=-.2:2,smooth,variable=\x] plot (\x,{-1.73*\x+2});
\draw[domain=-2:.2,smooth,variable=\x] plot (\x,{1.73*\x+2});
\draw[domain=-1.75:2.5,smooth,variable=\y] plot (0,\y);
\draw[domain=-2.25:2.25,smooth,variable=\x] plot (\x,{-1});
\draw[domain=-1.75:2.25,smooth,variable=\x] plot (\x,{-.58*\x});
\draw[domain=-2.25:1.75,smooth,variable=\x] plot (\x,{.58*\x});

	\fill[color=black] (0,0) circle (2pt);
	\fill[color=black] (0,2) circle (2pt);
	\fill[color=black] (1.73,-1) circle (2pt);
	\fill[color=black] (-1.73,-1) circle (2pt);

\begin{scope}[shift={(-5,0)}]

\draw[domain=-.2:2,smooth,variable=\x] plot (\x,{-1.73*\x+2});
\draw[domain=-2:.2,smooth,variable=\x] plot (\x,{1.73*\x+2});
\draw[rotate around={15:(0,2)}, domain=-1.75:2.5,smooth,variable=\y] plot (0,\y);
\draw[domain=-2.25:2.25,smooth,variable=\x] plot (\x,{-1});
\draw[rotate around={15:(1.73,-1)}, domain=-1.75:2.25,smooth,variable=\x] plot (\x,{-.58*\x});
\draw[rotate around={15:(-1.73,-1)}, domain=-2.25:1.75,smooth,variable=\x] plot (\x,{.58*\x});

	\fill[color=black] (0,2) circle (2pt);
	\fill[color=black] (1.73,-1) circle (2pt);
	\fill[color=black] (-1.73,-1) circle (2pt);

\end{scope}

\end{tikzpicture}
	\caption{The unique arrangements $\mathcal{B}$ and $\mathcal{C}$ given in Lemmas \ref{lem:6lines3triples} and \ref{lem:6lines4triples}, respectively.}
	\label{fig:lemmas}
\end{center}
\end{figure}
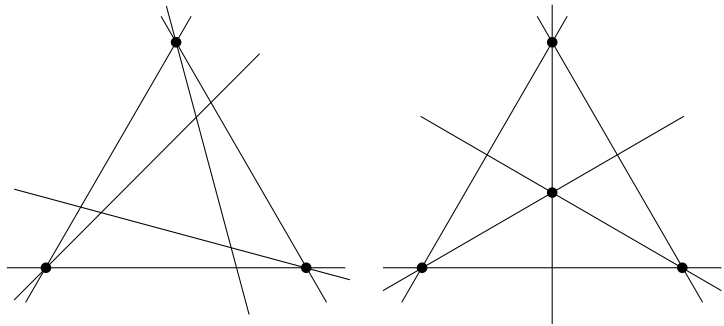

\section{The geometric methodology}
\label{subsec:linemethod}



The two goals of this section are to determine the realizability of the configuration tables that are produced in the main proofs below and to determine the irreducibility of these spaces to rule out arrangements which cannot produce Zariski pairs.

To begin we must first convert the combinatorial data of the configuration tables to algebraic equations of geometric lines.  From here we can study the moduli spaces of such arrangements to answer our main questions.  What follows is a lead up to Algorithm \ref{algorithm} that we used in the main proofs below.

We fix some notations from projective geometry.

\begin{definition}
Let $P_1=[a_1, b_1, c_1]$, $P_2=[a_2, b_2, c_2]$, and $P_3=[a_3, b_3, c_3]$ be three points in the projective plane. We define the determinant of $P_1$, $P_2$ and $P_3$ as
\[\det(P_1, P_2, P_3):=\det\begin{pmatrix}a_1&b_1&c_1\\a_2&b_2&c_2\\ a_3&b_3&c_3\end{pmatrix}.\]
For a projective line $L$ defined by $ax+by+cz=0$, we denote by $L^*=[a, b, c]$ the point in the dual projective plane $(CP^2)^*$. 
We define the \emph{determinant of the coefficient matrix} of three lines $L_1$, $L_2$ and $L_3$ to be $\det(L^*_1, L^*_2, L_3^*)$.
\end{definition}

To study the geometry, especially irreducibility, of the moduli space $\mathcal{M}_\mathcal{A}$ of a projective line arrangement $\mathcal{A}$, the idea is to convert combinatorial data into polynomial equations.

Here is the first such idea.

\begin{property}
\label{property:det}
For any three projective lines $L_i$ defined by $a_ix+b_iy+c_iz=0$, $i=1, 2, 3$,  the intersection $L_1\cap L_2\cap L_3$ is nonempty if and only if the determinant of the coefficient matrix, which we notate as $\det(L^*_1,L^*_2,L^*_3)$, is zero.  
\end{property}

We now use this to define spaces: the first by applying only one direction of this property to the triples that appear on the configuration tables; and the second by applying both directions to ensure that no other triples appear in arrangement except for those in the configuration tables.

\begin{definition}
\label{def:mod}

Let $C$ be a configuration table with $n$ columns  $L_1$, $L_2$, $\dots$, $L_n$. We define the total space $\mathcal{T}_C$ of the configuration table $C$  to be the Zariski  subset of $((\mathbb{CP}^2)^*)^n$ which consists of points $(P_1, P_2, \dots, P_n)$  such that $P_1$, $P_2$, $\dots$, $P_n$ are distinct and 
satisfy the following condition:
\begin{enumerate}
\item
$\det(P_i,P_j,P_k) = 0$ whenever the three columns $L_i$, $L_j$ and $L_k$ share a point.
\end{enumerate}

Let $\mathcal{A}=\{L_1, L_2,\dots, L_n\}$ be an arrangement of $n$ projective lines associated to the configuration table $C$. 
We define the total space $\mathcal{T}_\mathcal{A}$ of the arrangement $\mathcal{A}$ to be the Zariski subset of $((\mathbb{CP}^2)^*)^n$ which consists of points $(P_1, P_2, \dots, P_n)$ such that $P_1$, $P_2$, $\dots$, $P_n$ are distinct and satisfy the following two conditions:
\begin{enumerate}\item $\det(P_i, P_j, P_k)=0$ if the intersection of $L_i$, $L_j$ and $L_k$ is non-empty, and
\item $\det(P_r, P_s, P_t)\neq 0$ if the intersection of $L_r$, $L_s$, $L_t$ is empty.
\end{enumerate}  

The quotient of  $\mathcal{T}_C$ by the automorphism group $PGL(3, \mathbb{C})$ of the dual plane, denoted by $\mathcal{M}_C:=\mathcal{T}_C/PGL(3,\mathbb{C})$  is called the \emph{moduli space} of the configuration table $C$.

The quotient of  $\mathcal{T}_{\mathcal{A}}$ by the automorphism group $PGL(3, \mathbb{C})$ of the dual plane, denoted by $\mathcal{M}_\mathcal{A}:=\mathcal{T}_\mathcal{A}/PGL(3,\mathbb{C})$  is called the \emph{moduli space} of the line arrangement $\mathcal{A}$.

\end{definition}

In the above definition, the points $P_i$ should be considered as duals of projective lines.

\begin{proposition}
\label{prop:correspondence}Let $\mathcal{A}$ be an arrangement of projective lines associated to a configuration table $C$.
The moduli space $\mathcal{M}_\mathcal{A}$ is irreducible if the moduli space space $\mathcal{M}_C$ is irreducible.
\end{proposition}


\begin{proof}
We see that $\mathcal{T}_\mathcal{A}$ is a Zariski open subset of $\mathcal{T}_C$.  Hence, the moduli space $\mathcal{M}_\mathcal{A}$ is a Zariski open subset of $\mathcal{M}_C$. 
\end{proof}

This prescribes Algorithm \ref{algorithm} that we present in the next subsection, where the reader can find more details.

First we present an important theorem from the literature that  will be used to classify moduli spaces.

\begin{theorem}[Randell's Lattice-Isotopy Theorem \cite{Ran}]
\label{thm:Rand}
If $\mathcal{A}_t$ is a lattice-isotopy between two line arrangements $\mathcal{A}_0$ and $\mathcal{A}_1$, then the complement of $\mathcal{A}_0$ is diffeomorphic to the complement  of $\mathcal{A}_1$. 
\end{theorem}

We denote by $\mathcal{M}^\mathbb{C}_{\mathcal{A}}$ the quotient of $\mathcal{M}_{\mathcal{A}}$ under complex conjugation.  Using Randell's Lattice-Isotopy Theorem, we observe the following rigidity result on diffeomorphic types of complements of line arrangements.

\begin{proposition}
\label{prop:blue}
 If $\mathcal{M}^\mathbb{C}_\mathcal{A}$ is irreducible, 
 then the complements of any two line arrangements in $\mathcal{M}_\mathcal{A}$ are diffeomorphic. In particular, the fundamental group of the complement of $\mathcal{A}$ is a combinatorial invariant.
\end{proposition}

\begin{proof}

Let $[\mathcal{A}_1]$ and $[\mathcal{A}_2]$ be two points in $\mathcal{M}_\mathcal{A}$. If the two arrangements $\mathcal{A}_1$ and $\mathcal{A}_2$ are in the same irreducible component of $\mathcal{M}_\mathcal{A}$, then by Randell's Lattice-Isotopy Theorem \cite{Ran} the  complements of  $\mathcal{A}_1$ and $\mathcal{A}_2$ are diffeomorphic. 

Suppose that  the two arrangements $\mathcal{A}_1$ and $\mathcal{A}_2$  are in two different irreducible components of $\mathcal{M}_\mathcal{A}$. Since $\mathcal{M}^\mathbb{C}_\mathcal{A}$ is irreducible, the two irreducible components must be complex conjugated to each other.  Therefore, there exist a line arrangement $\mathcal{A}'_1$ which is complex conjugated to $\mathcal{A}_1$ and in the same irreducible component of $[\mathcal{A}_2]$.  Notice that complex conjugation is a diffeomorphism. Then the complements of  $\mathcal{A}_1$ and $\mathcal{A}'_1$ are diffeomorphic.  Since $\mathcal{A}'_1$  and $\mathcal{A}_2$ are in the same irreducible component, then their complements are diffeomorphic.  Therefore,  the complement of  $\mathcal{A}_1$ is diffeomorphic to the complement of $\mathcal{A}_2$.
\end{proof}

\begin{remark}Notice that in  \cite{CoSuc}, Cohen-Suciu prove the following theorem which implies that the fundamental groups of complements of complex conjugated curves are isomorphic. Together with Randell's Lattice-Isotopy Theorem, this also implies that the fundamental group of the complement of $\mathcal{A}$ is combinatorially invariant if $\mathcal{M}^\mathbb{C}_\mathcal{A}$ is irreducible.
\end{remark}

\begin{theorem}[Cohen-Suciu \cite{CoSuc}, Theorem 3.9]
\label{thm:CS}
The braid monodromies of complex conjugated curves are equivalent. 
\end{theorem}






\subsection{A practical algorithm and examples}
\label{subsec:ex}

The highlight of this subsection is Algorithm \ref{algorithm}, which uses a configuration table to obtain a line arrangement and its moduli space, and to determine whether this space is irreducible and  whether the configuration table is geometrically realizable.

Several examples are given highlighting the problematic spots of this algorithm, and the reader should pay careful attention to Example \ref{example:12.B.2.iv} which produces a potential Zariski pair.

To set this up we begin with an important lemma that will allow us to use the same three-by-three grid of lines for all of our cases.  Afterwards we mention several results and techniques that are used to discern the moduli spaces.

\begin{definition}
\label{def:pencil}
A \emph{pencil} is an arrangement of $k$ lines with a multiple point of multiplicity $k$.
\end{definition}

 We start from a basic observation that a line arrangement can be specified to a special position in the projective plane.  This fact, without proof, has been used in \cite{NazYos}, \cite{Fei9}, and \cite{Fei:10} to calculate moduli spaces of arrangements.  We include the proof here for the sake of completion.
 
  Note that we use Lemma \ref{grid} to draw pictures in the affine plane which is  the complement of the line at infinity  (not included in the arrangement) defined by $z=0$ in $\mathbb{CP}^2$.

\begin{lemma}
\label{grid}
Let $\{L_1, L_2, L_3\}$ and $\{L_4, L_5, L_6\}$ be two pencils of lines which intersect transversally in $9$ points.  Then there is a unique automorphism of the dual projective plane such that the six lines under the automorphism are defined by  $x=0$, $x=z$, $x=t_1z$, $y=0$, $y=z$, $y=t_2z$. 
\end{lemma}
\begin{proof}We know that the automorphism group of the projective plane is $PGL(3,\mathbb{C})$. We can find an automorphism $M_0$ sending the line passing through the two triple points $L_1\cap L_2\cap L_3$ and $L_4\cap L_5\cap L_6$ to the line at infinity $z=0$.  Under this automorphism, the six lines can be defined by $ax+by+c_iz=0$ for $i=1,2,3$ and $dx+ey+f_{i-3}z=0$ for $i=4,5,6$, respectively, such that $D=\det \begin{pmatrix}a & b\\ d & e\end{pmatrix}\neq 0$ and $c_i\neq c_j$, $f_i\neq f_j$ if $i\neq j$.

In the next step we will find an automorphism fixing the line at infinity and changing the six lines to three horizontal and three vertical lines.   Consider  the general linear matrix
 \[M_1=-\dfrac{1}{D}\begin{pmatrix}b&-e & 0\\ -a&d &0\\ 0&0 & -D \end{pmatrix}\] 
 which acts on the dual projective plane from right and fixes the point $[0, 0, 1]$.  We see that $M_1$ sends the  points  $[a, b, c_i]$ and $[d, e, f_i]$ to $[0, 1, c_i]$ and $[1, 0, f_i]$, respectively, for $i=1, 2, 3$, as we expected.

To complete the proof, let us consider the matrix 
 \[M_2=\begin{pmatrix}f_2-f_1&0 & -f_1\\ 0&c_2-c_1 &-c_1\\ 0&0 & 1 \end{pmatrix}.\]
  One can check that $M_2$ sends $[0, 1, c_1]$, $[0, 1, c_2]$, $[1, 0, f_1]$ and $[1, 0, f_2]$ to $[0, 1, 0]$, $[0,1, 1]$, $[1, 0, 0]$ and $[1, 1, 0]$, 
respectively. The automorphism $M_0M_1M_2$ is then  the unique automorphism we want. 
\end{proof}

{ We now present the full details of the procedures used throughout the proofs below.  We conclude with examples that illustrate the different methods and situations that occur in the algorithm.}

\begin{algorithm}
\label{algorithm}
This algorithm is used to determine the irreducibility and realizability of an arrangement obtained from a given configuration table.
\begin{enumerate}
	\item[(A1)] Use the Grid Lemma \ref{grid} and its intersection points to determine equations for the lines.
	\item[(A2)] Obtain special defining equations from triple points not used in (A1) above.
	\item[(A3)] Determine the \textbf{irreducibility} of the moduli space $\mathcal{M}$ by checking the irreducibility of the equations from (A2) above.
		\begin{enumerate}
			\item True for $\dim \mathcal{M}=0$.
			\item If $\dim \mathcal{M}>0$, either use Mathematica's irreducibility test given by the command\\ \verb=IrreduciblePolynomialQ[poly, Extension->All]=\\
			and notated by ($*$) or equations ending in $=0^*$ in the tables\\
			or check by hand, as in Example \ref{11.B.3.a.-irr} for the case 11.B.3.a.iii.
		\end{enumerate}
	\item[(A4)] Determine the \textbf{realizability} of the configuration table.
		\begin{enumerate}
			\item Whenever the set $\mathcal{M}$ is empty, a contradiction arises, and it is noted in the tables.
			\item True for $\dim \mathcal{M}>0$.
			\item If $\dim \mathcal{M}=0$, look for a realization for one of the given points:  if a real solution exists, draw the arrangement to check, as in Example \ref{example:12.B.2.iv} for the case 12.B.3.b.iii.;\\ 
 if only complex solutions exist, as notated by ($\mathbb{C}$) or equations ending in $=0^\mathbb{C}$ in the tables, need to check, as in Example \ref{example:cpx-cnjgt} for the case $(9_3)$.ii.DFH, that no double points coincide in a triple.
		\end{enumerate}
\end{enumerate}
Potential Zariski paris are notated by ($\mathcal{Z}$) or equations ending in $=0^\mathcal{Z}$ in the tables.
\end{algorithm}

\begin{example}[The moduli space of the arrangement 12.B.3.b.iii.:  two real points]
\label{example:12.B.2.iv}

Using Lemma \ref{grid}, we may assume that the lines $L_7$,  $L_5$, $L_4$, $L_8$, $L_9$ and $L_3$ are defined by equations  $y=0$,  $y=z$, $y=bz$,   $x=0$,  $x=z$, and $x=az$, respectively, where $a$ and $b$ are complex numbers in $\mathbb{C}\setminus\{0, 1\}$.

Since the three triples $e_{11}=L_4\cap L_8=[0, b, 1]$, $e_6=L_3\cap L_7=[a, 0, 1]$ and $e_{12}=L_5\cap L_9=[1, 1, 1]$ are on $L_6$, the equation of the line $L_6$ can be written as $y=-\frac{b}{a}x+bz$ such that the first defining equation $1=b-\frac{b}{a}$, or equivalently, $ab=a+b$, holds.

Now we determine the equations for the rest of lines, $L_1$, $L_2$, and $L_{10}$. Since the line $L_1$ passes though the two triples $e_2=L_3\cap L_5=[a, 1, 1]$ and $e_4=L_7\cap L_9=[1, 0, 1]$, the line $L_1$ is defined by $y=\frac{1}{a-1}(x-z)$. Therefore, the triple $e_1=L_1\cap L_4$ is given by $[b(a-1)+1, b, 1]=[a+1, b, 1]$ by the equality $ab=a+b$.  The line $L_2$ passes through the two triples $e_5=L_7\cap L_8=[0, 0, 1]$ and $e_1=[a+1, b, 1]$, and so its equation can be written as $y=\frac{b}{a+1}x$. The last line $L_{10}$ passes through the triples $e_8=L_4\cap L_9=[1, b, 1]$ and $e_9=L_5\cap L_8=[0, 1, 1]$, and so its equation is given by $y=(b+1)x+z$.

Note that the triple $e_3=L_2\cap L_6\cap L_{10}$ did not yet enter into the equations of the lines; thus it will become the second defining equation
\begin{equation*}
\det(L_2, L_6, L_{10})=\det\begin{pmatrix}1&\frac{b}{a+1}& 0\\ 1&-\frac{b}{a}& b\\ 1&b-1&1 \end{pmatrix}=0.
\end{equation*}
Simplifying this second defining equation, we obtain the equation $a^2b-a^2+a+1=0$. Together with the first defining equation $ab=a+b$, we solve for $a$ and $b$ to obtain $a=\pm\frac{\sqrt{2}}{2}$ and $b=\frac{1}{1\mp\sqrt{2}}$.

We then have two real line arrangements shown in Figure \ref{12.B.2.iv}.
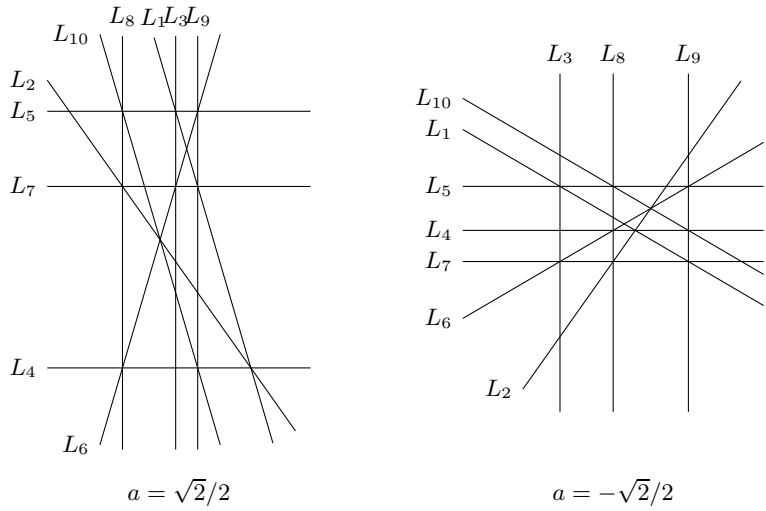
\begin{figure}[htbp]
\centering
\subfigure{}{
\begin{tikzpicture}[scale=1]
\tikzstyle{every node}=[font=\footnotesize]
\draw[domain=-3.5:2] plot (0, \x) node[above]{$L_8$};
\draw[domain=-3.5:2] plot (1, \x) node[above]{$L_9$};
\draw[domain=-3.5:2] plot ( {sqrt(2)/2}, \x) node[above]{$L_3$};

\draw[domain=2.5:-1] plot (\x, 0) node[left]{$L_7$};
\draw[domain=2.5:-1] plot (\x, 1) node[left]{$L_5$};
\draw[domain=2.5:-1] plot (\x, {1/(1-sqrt(2))}) node[left]{$L_4$};

\draw[domain={1.3:-0.3}] plot (\x, {sqrt(2)/(1-sqrt(2))*\x+1}) node[left]{$L_{10}$};
\draw[domain={2:0.42}] plot (\x, {2/(sqrt(2)-2)*(\x-1)}) node[above]{$L_1$};
\draw[domain={1.3:-0.3}] plot (\x, {-2/(sqrt(2)-2)*\x+1/(1-sqrt(2))}) node[left]{$L_6$};
\draw[domain={2.3:-1}] plot (\x, {-sqrt(2)*\x}) node[left]{$L_2$};
\node at (0.75,-3.75)[below]{$a=\sqrt{2}/2$};
\end{tikzpicture}
}\hspace{1cm}
\subfigure{}{\begin{tikzpicture}[scale=1]
\tikzstyle{every node}=[font=\footnotesize]
\draw[domain=-2:2.5] plot (0, \x) node[above]{$L_8$};
\draw[domain=-2:2.5] plot (1, \x) node[above]{$L_9$};
\draw[domain=-2:2.5] plot (-{sqrt(2)/2}, \x) node[above]{$L_3$};

\draw[domain=2:-2] plot (\x, 0) node[left]{$L_7$};
\draw[domain=2:-2] plot (\x, 1) node[left]{$L_5$};
\draw[domain=2:-2] plot (\x, {1/(1+sqrt(2))}) node[left]{$L_4$};

\draw[domain={2:-2}] plot (\x, {-sqrt(2)/(1+sqrt(2))*\x+1}) node[left]{$L_{10}$};
\draw[domain={2:-2}] plot (\x, {-2/(sqrt(2)+2)*(\x-1)}) node[left]{$L_1$};
\draw[domain={2:-2}] plot (\x, {2/(sqrt(2)+2)*\x+1/(1+sqrt(2))}) node[left]{$L_6$};
\draw[domain={1.7:-1.2}] plot (\x, {sqrt(2)*\x}) node[left]{$L_2$};
\node at (0,-2.75)[below]{$a=-\sqrt{2}/2$};
\end{tikzpicture}
}
\caption{The two real realizations of the arrangement 12.B.2.iv \label{12.B.2.iv}}
\end{figure}

From the figure, we see that there are exactly twelve triple points and no points of higher multiplicities. 
Therefore, the moduli space of the arrangement 12.B.2.iv consists of two distinct real points.
\end{example}

\begin{example}[The moduli space of the arrangement  11.B.3.a.iii.]\label{11.B.3.a.-irr}
Again using Lemma \ref{grid}, we may assume that the lines $L_1$,  $L_2$, $L_4$, $L_3$, $L_{10}$ and $L_9$ are defined by equations  $y=0$,  $y=z$, $y=bz$,   $x=0$,  $x=z$, and $x=az$, respectively, where $a$ and $b$ are complex numbers in $\mathbb{C}\setminus\{0, 1\}$.

Since the line $L_{6}$ passes through the triples $e_3=L_2\cap L_3=[0, 1, 1]$ and $e_4=L_1\cap L_{9}=[a, 0, 1]$, the defining equation of the line $L_{6}$ can be written as $y=-\frac{1}{a}x+z$. Therefore the triple $e_{10}=L_4\cap L_5\cap L_{6}$ can be given by $[a(1-b), b, 1]$. Notice that the triple $e_2=L_2\cap L_5\cap L_{10}$ can be given by $[0, 0, 1]$.

The defining equation of the line $L_5$ is $y=\frac{b}{a(1-b)}x$.  Then we can find the triples $e_6=L_2\cap L_5\cap L_7=[\frac{a(1-b)}{b}, 1, 1]$.

Since $L_7$ passes through $e_6$ and $e_5=L_1\cap L_{10}=[1, 0, 1]$, its defining equation can be written as $y=\frac{b}{a-b-ab}(x-1)$. 

Now only the defining equation of the line $L_8$ was not written down yet.  Since $L_8$ passes through the points $e_7=L_2\cap L_8\cap L_{10}=[1, 1, 1]$ and $e_8=L_3\cap L_4\cap L_8=[0, b, 1]$, we can write  its defining equations  as $y=(1-b)x+bz$.

We have used all but $e_{11}=L_7\cap L_8\cap L_9$ to calculate the defining equations. The parameters $a$ and $b$ must satisfy Property \ref{property:det} for the determinant of the coefficient matrix of $L_4$, $L_5$ and $L_6$, i.e.
\[\det\begin{pmatrix}1&\frac{b}{a-b-ab}& -\frac{b}{a-b-ab}\\ 1&1-b& b\\ 0&-a&1 \end{pmatrix}=0.\]
Simplify this equation to get $a^2b^2-2a^2b+a^2-ab-b^2+b=0$.   So we see the moduli space is a curve defined by $a^2b^2-2a^2b+a^2-ab-b^2+b=0$.

We claim that this equation defines an irreducible curve in $\mathbb{C}^2$. Assume contrarily that the polynomial $p(a, b)=a^2b^2-2a^2b+a^2-ab-b^2+b=a^2(b-1)^2-ab-b(b-1)$ is reducible. Then by the definition of reducibility it must be factored as $p(a, b)=[f_1(b)a+g_1(b)][f_2(b)a+g_2(b)]$, where $f_1(b)$, $f_2(b)$, $g_1(b)$ and $g_2(b)$ are polynomials in $\mathbb{C}[b]$.  Viewing $p(a, b)$ as a quadratic polynomial of $a$, we see that the discriminant $b^2+4b(b-1)^3=(f_1(b)g_2(b)-g_1(b)f_2(b))^2$ must be a perfect square. We check that $b^2+4b(b-1)^3$ is not a perfect square. In fact, it has no double root at all. 

Therefore, the moduli space is irreducible.

Choose the two evaluations $a=4$ and $b=0.62404$ to get the arrangement pictured in Figure \ref{11.B.3.a}.
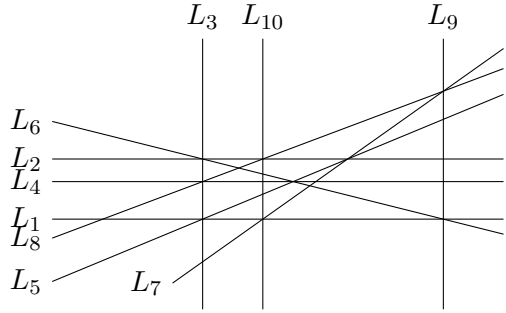
\begin{figure}[htbp]
\centering

\pgfmathsetmacro{\a}{4}
\pgfmathsetmacro{\b}{0.62404}

\begin{tikzpicture}[domain=-1.5:3,scale=0.8]
\draw plot (0, \x) node[above]{$L_{3}$};
\draw plot (1, \x) node[above]{$L_{10}$};
\draw plot (4, \x) node[above]{$L_9$};

\draw[domain=5:-2.5] plot (\x, 0) node[left]{$L_1$};
\draw[domain=5:-2.5] plot (\x, 1) node[left]{$L_2$};
\draw[domain=5:-2.5] plot (\x, 0.62404) node[left]{$L_4$};

\draw[domain=5:-2.5] plot (\x, {(-1/4)*\x+1}) node[left]{$L_6$};
\draw[domain=5:-2.5] plot (\x, {(0.62404/(4*(1-0.62404)))*\x}) node[left]{$L_{5}$};
\draw[domain=5:-0.5] plot (\x, {(0.62404/(4-0.62404-4*0.62404))*(\x-1)}) node[left]{$L_{7}$};
\draw[domain=5:-2.5] plot (\x, {(1-0.62404)*\x+0.62404)}) node[left]{$L_{8}$};
\end{tikzpicture}
\caption{The arrangement 11.B.3.a.iii.\label{11.B.3.a}}
\end{figure}
\end{example}

\begin{example}[A moduli space consists of two complex conjugate points]\label{example:cpx-cnjgt} Consider the arrangement in Lemma \ref{lem:12Reduction-b} obtained by adding the line $DFH$ to the ($9_3$).ii. arrangement (see Table \ref{tab:93.ii.EQNS}), i.e. add a line passing through the intersection points $L_2\cap L_5$, $L_3\cap L_6$, and $L_4\cap L_9$.

Assume that the lines $L_3$,  $L_2$, $L_1$, $L_4$, $L_5$, and $L_6$ are defined by equations  $y=0$,  $y=z$, $y=bz$,   $x=0$,  $x=z$, and $x=az$, respectively, as depicted in 
 Figure \ref{9.3.ii.DFH}.
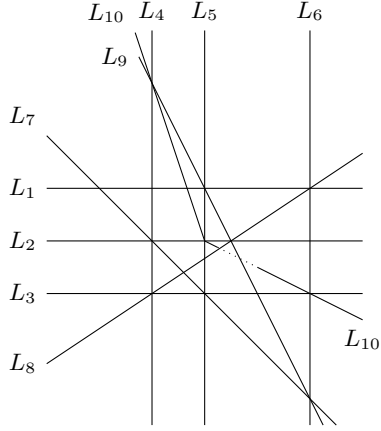
\begin{figure}[htbp]
\centering
\begin{tikzpicture}[scale=0.7]
\tikzstyle{every node}=[font=\footnotesize]
\draw[domain=-2.5:5] plot (0, \x) node[above]{$L_4$};
\draw[domain=-2.5:5] plot (1, \x) node[above]{$L_5$};
\draw[domain=-2.5:5] plot (3, \x) node[above]{$L_6$};
\draw[domain=4:-2] plot (\x, 0) node[left]{$L_3$};
\draw[domain=4:-2] plot (\x, 1) node[left]{$L_2$};
\draw[domain=4:-2] plot (\x, 2) node[left]{$L_1$};
\draw[domain=4:-2] plot (\x, {2/3*\x}) node[left]{$L_8$};
\draw[domain=3.25:-0.25] plot (\x, {-2*\x+4}) node[left]{$L_9$};
\draw[domain=3.5:-2] plot (\x, {-\x+1}) node[above left]{$L_7$};
\draw[domain=1:-0.32] plot (\x, {-3*\x+4}) node[above left]{$L_{10}$};
\draw[domain=1:1.25] plot (\x, {-1/2*(\x-3)});
\draw[domain=1.25:2,dotted] plot (\x, {-1/2*(\x-3)});
\draw[domain=2:4] plot (\x, {-1/2*(\x-3)}) node[below]{$L_{10}$};
\end{tikzpicture}
\caption{The arrangement ($9_3$).ii.DFH \label{9.3.ii.DFH}}
\end{figure}

By the same method used in the previous examples, we can write down the equations of the rest of lines:
$L_7: y=-x+z$, $L_8: y=\frac{b}{a} x$, $L_9: y=\frac{b(b-1)}{b-a}x-\frac{b(1-a)}{b-a}z$ and $L_{10}: y=\frac{1}{1-a}(x-1)$, where $a$ and $b$ satisfy the following equations $a=b^2-b+1$ and $b^2+1=0$.

Solving the system of the two equations, we get two pairs of solutions $(a=-i, b=i)$ and $(a=i, b=-i)$. Replacing $a$ and $b$ in the defining equations of the ten lines by $\mp i$ and $\pm i$, respectively, we see that the line $L_{10}$ given by the equation $y=\frac{1}{1+i}(x-z)$ is indeed distinct from the other nine lines.

Now we check whether there will be extra multiple points forced to appear when adding the tenth line. If there is an extra multiple point, then it must be on the line $L_{10}$.  By checking the Figure \ref{9.3.ii.DFH} or Figure \ref{fig:932}, we see that only the points $P=L_1\cap L_7=[1-i, i, 1]$ and $Q=L_7\cap L_8=[-1, 1, 0]$ may fall on the line $L_{10}$ accidentally. However, by plugging the two points into the equations of $L_{10}$, we get contradictions.

Therefore,  $P$ and $Q$ cannot be on the line $L_{10}$, which means that adding the line $L_{10}$ introduces no extra multiple points.
\end{example}

\begin{remark}
\label{rem:checkdegen}
This process outlined in Example \ref{example:cpx-cnjgt} for the case ($9_3$).ii.DFH also occurs for the cases ($9_3$).iii.ACG, ($9_3$).iii.AEG, and ($9_3$).iii.BEG.
\end{remark}

\section{Arrangements of nine lines with nine and ten triples}
\label{sec:9}

The purpose of this section is to highlight previous results for nine lines that will be used in reduction arguments for ten lines in Sections \ref{sec:13} and \ref{sec:12}.


\begin{definition}
\label{def:Ck}
\cite[Definition 3.1]{NazYos}
Let $k\in\mathbb{N}$.  We say a line arrangement $\mathcal{A}$ is of type $C_k$ if $k$ is the minimum number of lines of $\mathcal{A}$ containing all points of multiplicity at least three.
\end{definition}

\begin{definition}
\label{def:simpleC3}
\cite[Definition 3.4]{NazYos}
Let $\mathcal{A}$ be an arrangement of type $C_3$.  Then $\mathcal{A}$ is a \emph{simple $C_3$} arrangement if there are three lines $L_1,L_2,L_3\in\mathcal{A}$ such that all points of multiplicity at least three are contained in $L_1\cup L_2\cup L_3$ and one of the following holds:
\begin{itemize}
	\item $L_1\cap L_2\cap L_3=\emptyset$ and one of the lines contains only one point of multiplicity at least three apart from the other two lines; or
	\item $L_1\cap L_2\cap L_3\neq\emptyset$.
\end{itemize}
\end{definition}


From these definitions, Nazir and Yoshinaga proved the following results.

\begin{theorem}[Theorem 3.2 of \cite{NazYos}]
\label{thm:C012}
Let $\mathcal{A}=\{H_1,\ldots,H_n\}$ be a line arrangement in $\mathbb{P}^2_\mathbb{C}$ of class $C_{\leq2}$ (i.e., either $C_0$, $C_1$ or $C_2$). Then the realization space $R(I(A))$ 
is irreducible.
\end{theorem}

\begin{theorem}[Theorem 3.5 of \cite{NazYos}]
\label{thm:simpleC3}
Let $\mathcal{A}$ be an arrangement of $C_3$ of simple type. Then the realization space $R(I(A))$ 
 is irreducible.
\end{theorem}

The following is a more historical definition.

\begin{definition}
\label{def:n3}
An ($n_3$) arrangement is one with $n$ lines, $n$ triples, and no points of higher multiplicity.
\end{definition}

  The textbook \emph{Configurations of points and lines} by Gr\"unbaum \cite{Grun} gives a table summarizing the numbers of such combinatorial and geometric arrangements for small $n$ in Theorem 2.2.1. and Table 2.2.1. on p.69.

In the next proposition we consider the case $n=9$, first appearing as early as the 1880's, which Gr\"unbaum attributes to Kantor \cite{131}, Martinetti \cite{152}, Schr\"{o}ter \cite{199}, and again to Levi \cite[p.103]{145}, Hilbert and Cohn-Vossen \cite{126}, and Gropp \cite{94}.



\begin{proposition}[see Theorem 2.2.1. of \cite{Grun} or Proposition 3.7 of \cite{Fei9}]
\label{prop:9-3}
Let $\mathcal{A}$ be an arrangement of nine projective lines with nine triple points and no higher multiplicity points. If each line of $\mathcal{A}$ passes through exactly three triple points, then $\mathcal{A}$ is lattice isomorphic to one of the three arrangements given in Figure \ref{fig:9-3}.

We provide configuration tables in Table \ref{tab:93} (that differ from those in \cite{Grun}) and our own equations of these lines in Table \ref{tab:93EQNS}.

Note that this result holds for combinatorial as well as geometric arrangements.
\renewcommand{\thesubfigure}{}
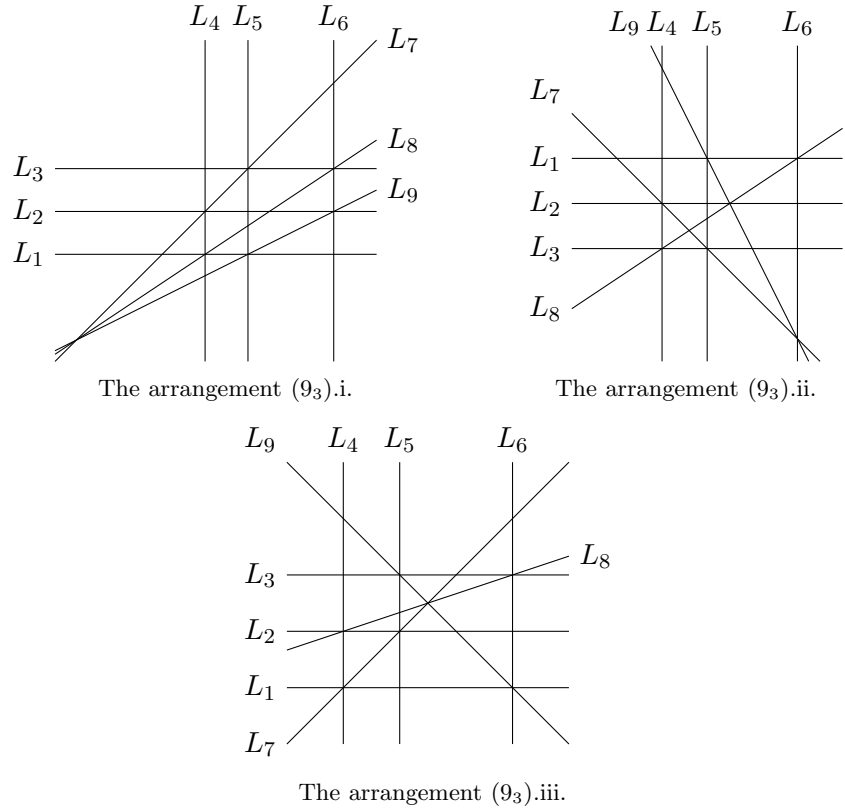
\begin{figure}[htbp]
\centering
\subfigure[The arrangement $(9_3)$.i.]{\begin{tikzpicture}[scale=0.57]
\draw[domain=-2.5:5]  plot (0, \x) node[above]{$L_4$};
\draw[domain=-2.5:5]  plot (1, \x) node[above]{$L_5$};
\draw[domain=-2.5:5]  plot (3, \x) node[above]{$L_6$};
\draw[domain=4:-3.5] plot (\x, 0) node[left]{$L_1$};
\draw[domain=4:-3.5] plot (\x, 1) node[left]{$L_2$};
\draw[domain=4:-3.5] plot (\x, 2) node[left]{$L_3$};
\draw[domain=-3.5:4] plot (\x, {\x+1}) node[right]{$L_7$};
\draw[domain=-3.5:4] plot (\x, {2*\x/3}) node[right]{$L_8$};
\draw[domain=-3.5:4] plot (\x, {(\x-1)/2}) node[right]{$L_9$};
\end{tikzpicture}
}
\hspace{2em}
\subfigure[The arrangement $(9_3)$.ii.]{
\begin{tikzpicture}[scale=0.6]
\draw[domain=-2.5:4.5] plot (0, \x) node[above]{$L_4$};
\draw[domain=-2.5:4.5] plot (1, \x) node[above]{$L_5$};
\draw[domain=-2.5:4.5] plot (3, \x) node[above]{$L_6$};
\draw[domain=4:-2] plot (\x, 0) node[left]{$L_3$};
\draw[domain=4:-2] plot (\x, 1) node[left]{$L_2$};
\draw[domain=4:-2] plot (\x, 2) node[left]{$L_1$};
\draw[domain=4:-2] plot (\x, {2/3*\x}) node[left]{$L_8$};
\draw[domain=3.25:-0.25] plot (\x, {-2*\x+4}) node[above left]{$L_9$};
\draw[domain=3.5:-2] plot (\x, {-\x+1}) node[above left]{$L_7$};
\end{tikzpicture}
}
\hspace{2em}
\subfigure[The arrangement $(9_3)$.iii.]{\begin{tikzpicture}[scale=0.75]
\draw[domain=-1:4] plot (0, \x) node[above]{$L_4$};
\draw[domain=-1:4] plot (1, \x) node[above]{$L_5$};
\draw[domain=-1:4] plot (3, \x) node[above]{$L_6$};
\draw[domain=4:-1] plot (\x, 0) node[left]{$L_1$};
\draw[domain=4:-1] plot (\x, 1) node[left]{$L_2$};
\draw[domain=4:-1] plot (\x, 2) node[left]{$L_3$};
\draw[domain=4:-1] plot (\x, \x) node[left]{$L_7$};
\draw[domain=-1:4] plot (\x, {\x/3+1}) node[right]{$L_8$};
\draw[domain=4:-1] plot (\x, {-\x+3}) node[above left]{$L_9$};
\end{tikzpicture}}
\caption{\label{fig:9-3} The three ($9_3$) arrangements of nine lines with nine triples and no higher multiplicity points as appearing in \cite{Fei9}.}
\end{figure}
\end{proposition}

\begin{table}[htbp]
{ \begin{tabular}{cc}

\begin{tabular}{|ccc|ccc|ccc|}
\hline
$L_1$ &	$L_2$ &	$L_3$ & $L_4$ &	$L_5$ &	$L_6$ & $L_7$ &	$L_8$ & $L_9$\\
\hline
$e_1$ & $e_1$ & $e_1$ & $e_8$	&	$e_8$ & $e_8$ & $e_9$ & $e_9$	&	$e_9$ \\
$e_2$ & $e_4$ & $e_6$ & $e_2$	&	$e_3$ & $e_5$ & $e_4$ & $e_2$	&	$e_3$  \\
$e_3$ & $e_5$ & $e_7$ & $e_4$ &	$e_6$ & $e_7$ & $e_6$ & $e_7$	&	$e_5$ \\
\hline
\end{tabular}

 & 

\begin{tabular}{|ccc|ccc|ccc|}
\hline
$L_1$ &	$L_2$ &	$L_3$ & $L_4$ &	$L_5$ &	$L_6$ & $L_7$ &	$L_8$ & $L_9$\\
\hline
$e_1$ & $e_1$ & $e_1$ & $e_8$	&	$e_8$ & $e_8$ & $e_4$ & $e_3$	&	$e_2$ \\
$e_2$ & $e_4$ & $e_6$ & $e_4$	&	$e_2$ & $e_3$ & $e_7$ & $e_5$	&	$e_5$  \\
$e_3$ & $e_5$ & $e_7$ & $e_6$ &	$e_7$ & $e_9$ & $e_9$ & $e_6$	&	$e_9$ \\
\hline
\end{tabular}

\\

 &  \\

The arrangement ($9_3$).i. & The arrangement ($9_3$).ii.\\

 &  \\

\begin{tabular}{|ccc|ccc|ccc|}
\hline
$e_1$ & $e_1$ & $e_1$ & $e_8$	&	$e_8$ & $e_8$ & $e_9$ & $e_9$	&	$e_9$ \\
$e_2$ & $e_4$ & $e_6$ & $e_2$	&	$e_5$ & $e_3$ & $e_2$ & $e_4$	&	$e_3$  \\
$e_3$ & $e_5$ & $e_7$ & $e_4$ &	$e_6$ & $e_7$ & $e_5$ & $e_7$	&	$e_6$ \\
\hline
\end{tabular}

& \\

 &  \\

The arrangement ($9_3$).iii. & \\

 &  
\end{tabular}
}
\caption{The arrangements ($9_3$).i., ($9_3$).ii., and ($9_3$).iii. given as configuration tables.}
\label{tab:93}
\end{table}

\begin{table}[htbp]\renewcommand{\arraystretch}{2.25}
\begin{tabular}{|c||c|c|c|c|c|}
\hline
\multicolumn{6}{|c|}{($9_3$). Nine lines with nine triples}\\
\hline
\hline
Arr. & $y=0,z,bz$ & $x=0,z,az$ & $y=Ax$ &	$y=Bx+z$ &	$y=C(x-z)$ \\
\hline
\hline
i. & $L_1,L_2,L_3$	& $L_4,L_5,L_6$	& $L_8$:  $\frac{b}{a}$ & $L_7$:  $b-1$	&	$L_9$: $\frac{1}{a-1}$ \\
\hline
\hline
ii. & $L_3,L_2,L_1$	& $L_4,L_5,L_6$	& $L_8$: $\frac{b}{a}$ & $L_7$:  $-1$	&	{}\\
\hline
 & \multicolumn{5}{|c|}{with $L_9: y=\frac{b(b-1)}{b-a}(x+\frac{b(1-a)}{b-a}z)$ and satisfying $a-(b^2-b+1)=0$}\\
\hline
\hline
iii. & $L_1,L_2,L_3$	& $L_4,L_5,L_6$	& $L_7$:  $1$ & $L_8$: $\frac{b-1}{a}$ &	{} \\
\hline
 & \multicolumn{5}{|c|}{with $L_9: y=\frac{b}{1-a}(x-az)$ and satisfying $a-(b+1)=0$}\\
\hline
\end{tabular}
\caption{Equations for geometric arrangements ($9_3$).}
\label{tab:93EQNS}
\end{table}

Two of these three arrangements are $C_{3}$:  all triple points lie on three lines.  For ($9_3$).i. three such lines are $L_3,L_4,L_9$, and for ($9_3$).iii. three such lines are $L_1,L_5,L_8$.

These can be used to produce arrangements of nine lines with ten triples by making the three specified lines meet at a tenth triple.  In the first case this results in the Pappus arrangement (and so ($9_3$).i. can be thought of a degeneration of the Pappus arrangement), and in the second case this results in a non-geometric arrangement.

\begin{remark}
\label{rem:pappus}
The Pappus arrangement is named after Pappus of Alexandria, one of the last great Greek mathematicians of Antiquity, whose hexagon theorem states that given two sets of three collinear points, the three intersection points of six lines are collinear.  The arrangement comes from the geometry of this theorem.
\end{remark}


\begin{proposition}
\label{prop:comb9lines10triples}
There are three combinatorial arrangements of nine lines with ten triples and no points of higher multiplicity.  We provide configuration tables in Table \ref{tab:9ten} and equations of these lines in Table \ref{tab:10-3EQNS}.  Just two are geometric:  we show these in Figure \ref{fig:9lines10triples}.
\end{proposition}

\begin{proof}
There are nine lines with ten triples, and so exactly three lines $L_1,L_2,L_3$ must each contain four triples.

If these three lines form a central subarrangement, they intersect at a triple, say $e_1$, and the arrangement is in fact simple $C_3$.  In this case, the triple $e_1$ can be degenerated to obtain an arrangement of nine lines with nine triples that is also simple $C_3$.  Therefore it must be either ($9_3$).i. or ($9_3$).iii., as discussed above.  In the first case our arrangement of nine lines with ten triples is the Pappus arrangement; in the second case we obtain a non-geometrically realizable one.

Otherwise these three lines form a generic subarrangement.  We proceed according to how many of these three doubles are in fact triples in the original
arrangement.

Suppose none are triples in the arrangement.  Then there must be four distinct triples on each of the three lines, totalling twelve triples, a contradiction.

Suppose one is a triple in the arrangement.  Then there must be three other distinct triples on each of the two lines meeting at this double, along with four distinct triples on the third line, totalling eleven triples, a contradiction.

Suppose two are triples in the arrangement.  Then although there are exactly ten triples accounted for here, one of these three lines passes through just one of these two triples that contains three additional triples.  However, this contradicts Fact \ref{fact:3col6lines} because there are just five other lines for these triples to appear.

Suppose all three are triples $e_1,e_2,e_3$ in the arrangement.  Then the three additional lines passing through these points must also intersect in $e_{10}$, the only point not on these first three lines.  This gives exactly one arrangement:  the Nazir-Yoshinaga arrangement given in Example 5.3 of \cite{NazYos} and proved to be the only geometric arrangement of nine lines with ten triples which is not simple $C_3$ in Proposition 3.8 of \cite{Fei9}.
\end{proof}


%

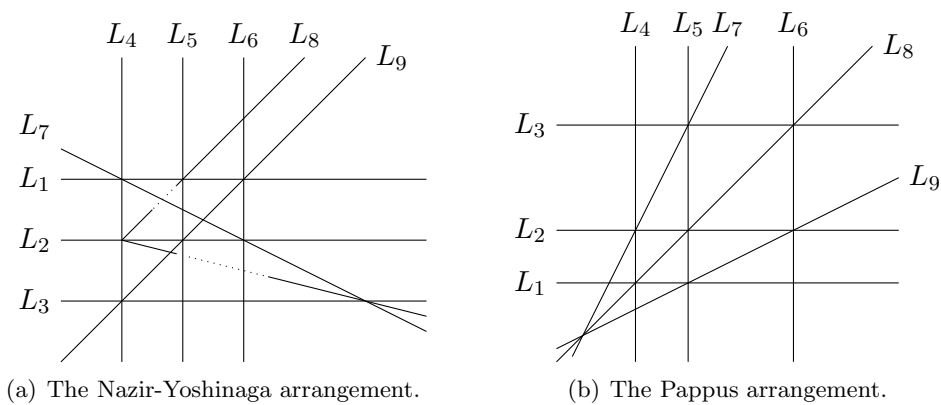
\begin{figure}[htbp]
\centering
\subfigure[The Nazir-Yoshinaga arrangement
.]{\label{fig:NY}
\begin{tikzpicture}[domain=5:-1, scale=0.81]
\draw plot (\x, 0) node[left]{$L_3$};
\draw plot (\x, 1) node[left]{$L_2$};
\draw plot (\x, 2) node[left]{$L_1$};
\draw[domain=-1:4] plot (0,\x) node[above]{$L_4$};
\draw[domain=-1:4] plot (1, \x) node[above]{$L_5$};
\draw[domain=-1:4] plot (2, \x) node[above]{$L_6$};
\draw[domain=-1:4] plot (\x, \x) node[right]{$L_9$};
\draw plot (\x, {-1/2*\x+2})  node[above left]{$L_7$};
\draw (0.9, 1.9)--(3, 4);
\draw (3,4) node[above]{$L_8$};
\draw[dotted] (0.9, 1.9)--(0.5, 1.5);
\draw (0.5, 1.5)--(0, 1);
\draw (0,1)--(0.9, {1-0.9/4});
\draw[dotted] (1.1, {1-1.1/4})--(2.4, {1-2.4/4});
\draw (2.4,{1-2.4/4})--(5, -1/4);
\end{tikzpicture}}
\hspace{2em}
\subfigure[The Pappus arrangement.]{\label{fig:Deg-Pappus}\begin{tikzpicture}[scale=0.7]
\draw[domain=-1.5:4.5]  plot (0, \x) node[above]{$L_4$};
\draw[domain=-1.5:4.5]  plot (1, \x) node[above]{$L_5$};
\draw[domain=-1.5:4.5]  plot (3, \x) node[above]{$L_6$};
\draw[domain=5:-1.5] plot (\x, 0) node[left]{$L_1$};
\draw[domain=5:-1.5] plot (\x, 1) node[left]{$L_2$};
\draw[domain=5:-1.5] plot (\x, 3) node[left]{$L_3$};
\draw[domain=-1.2:1.75] plot (\x, {2*\x+1}) node[above]{$L_7$};
\draw[domain=-1.5:4.5] plot (\x, {\x}) node[right]{$L_8$};
\draw[domain=-1.5:5] plot (\x, {1/2*(\x-1)}) node[right]{$L_9$};
\end{tikzpicture}}
\caption{The two geometric arrangements with nine lines and ten triples as appearing in \cite{Fei9}.}
\label{fig:9lines10triples}
\end{figure}

\begin{table}[htbp]
{ \begin{tabular}{cc}

\begin{tabular}{|ccc|ccc|ccc|}
\hline
$L_1$ &	$L_2$ &	$L_3$ & $L_4$ &	$L_5$ &	$L_6$ & $L_7$ &	$L_8$ & $L_9$\\
\hline
$e_1$ & $e_1$ & $e_1$ & $e_{10}$	&	$e_{10}$ & $e_{10}$ & $e_2$ & $e_3$	&	$e_4$ \\
$e_2$ & $e_5$ & $e_8$ & $e_2$	    &	$e_3$    & $e_4$    & $e_7$ & $e_5$	&	$e_6$  \\
$e_3$ & $e_6$ & $e_9$ & $e_5$     &	$e_6$    & $e_7$    & $e_9$ & $e_9$	&	$e_8$ \\
$e_4$ & $e_7$ &       & $e_8$     &	         &          &       &      	&	      \\
\hline
\end{tabular}

 & 

\begin{tabular}{|ccc|ccc|ccc|}
\hline
$L_1$ &	$L_2$ &	$L_3$ & $L_4$ &	$L_5$ &	$L_6$ & $L_7$ &	$L_8$ & $L_9$\\
\hline
$e_1$ & $e_1$ & $e_1$ & $e_9$	&	$e_9$ & $e_9$ & $e_{10}$ & $e_{10}$	&	$e_{10}$ \\
$e_2$ & $e_4$ & $e_7$ & $e_2$	&	$e_3$ & $e_6$ & $e_4$    & $e_2$	  &	$e_3$  \\
$e_3$ & $e_5$ & $e_8$ & $e_4$ &	$e_5$ & $e_8$ & $e_7$    & $e_5$	  &	$e_6$ \\
      & $e_6$ &       &       &	$e_7$ &       &          & $e_8$  	&	      \\
\hline
\end{tabular}

\\

 &  \\

The Nazir-Yoshinaga arrangement. & The Pappus arrangement.\\

  &  \\
 
 \begin{tabular}{|ccc|ccc|ccc|}
\hline
$e_1$ & $e_1$ & $e_1$ & $e_8$	&	$e_8$ & $e_8$ & $e_9$ & $e_9$	&	$e_9$ \\
$e_2$ & $e_4$ & $e_6$ & $e_2$	&	$e_5$ & $e_3$ & $e_2$ & $e_4$	&	$e_3$  \\
$e_3$ & $e_5$ & $e_7$ & $e_4$ &	$e_6$ & $e_7$ & $e_5$ & $e_7$	&	$e_6$ \\
    $e_{10}$&  &     &       &	$e_{10}$ &       &          & $e_{10}$  	&	      \\
\hline
\end{tabular}

& \\

 &  \\

{ Non-geometric degenerated ($9_3$).iii.} & \\

 & 
 
\end{tabular}
}
\caption{The arrangements of Proposition \ref{prop:comb9lines10triples} given as configuration tables.}
\label{tab:9ten}
\end{table}

\begin{table}[htbp]\renewcommand{\arraystretch}{2.25}
\begin{tabular}{|c||c|c|c|c|c|}
\hline
\multicolumn{6}{|c|}{Arrangements of nine lines with ten triples.}\\
\hline
\hline
Arr. & $y=0,z,bz$ & $x=0,z,az$ & $y=Ax$ &	$y=Bx+z$ &	$y=C(x-z)$ \\
\hline
\hline
NY. & $L_3,L_2,L_1$	& $L_4,L_5,L_6$	& $L_9$:  $1$ & $L_8$:  $a-1$	&	{} \\
\hline
 & \multicolumn{5}{|c|}{with $L_7: y=\frac{1-a}{a}x+az$ satisfying $a-b=0$ and $a^2+1=0$}\\
\hline
\hline
Pappus. & $L_1,L_2,L_3$	& $L_4,L_5,L_6$	& $L_8$:  $1$ & $L_7$:  $a-1$	&	$L_9$: $\frac{1}{a-1}$ \\
\hline
 & \multicolumn{5}{|c|}{satisfying $a-b=0$}\\
\hline
\hline
Non-geom. & $L_1,L_2,L_3$	& $L_4,L_5,L_6$	& $L_7$:  $1$ & $L_8$: $\frac{b-1}{a}$ &	{} \\
\hline
 & \multicolumn{5}{|c|}{with $L_9: y=\frac{b}{1-a}(x-az)$ and satisfying $a-(b+1)=0$}\\
 & \multicolumn{5}{|c|}{{ and $a+b-1=0$, a contradiction}}\\
\hline
\end{tabular}
\caption{Equations for arrangements of nine lines and ten triples.}
\label{tab:10-3EQNS}
\end{table}






 In this present work, whose focus is arrangements of ten lines, we are not concerned with combinatorial arrangements of nine lines with more than ten triple points, as these do not arise in our reduction arguments below. So we conclude this discussion here.  However, one may obviously ask for a complete list of combinatorial arrangements for nine lines.

\begin{question}
\label{q:9lines}
What other combinatorial (but not geometric)
line arrangements of ten lines and no points of multiplicity higher than three arise with some other number of triples  $n_3\notin[9,10]$?
\end{question}

\section{Main Results}
\label{sec:Main}

The main work of this paper gives a classification of non-trivially constructed geometric arrangements of ten lines with only triple points.  We ignore those arrangements that are constructed trivially by adding lines which do not contain at least three triple points.

\begin{theorem}[Classification]
\label{thm:classification}
The number of arrangements of ten lines with only triple points are as given in Table \ref{tab:summaryALT} and as follows:  there are \textbf{seventy-one} combinatorial arrangements and sixty-two geometric, with fifty-four of these having either irreducible or complex conjugate moduli spaces.  In particular, this classification gives exactly {nine} non-geometric arrangements and exactly \textbf{nine} potential Zariski pairs.
\end{theorem}

\begin{table}[htbp]
\begin{tabular}{|c||c||c|c||c|c||c|c|c|}
\hline
Case & $\#$ & $\#$ non- & $\#$ & $\#$  & $\#$  & Result & Config & Table \\
by $\#$ &	comb    & geom: &	geom    & ($\mathcal{Z}$):  & irred	&  & Table & of\\ 
triples &	    & Tab \ref{tab:nongeom} &	    & Tab \ref{tab:finallist} & or ($\mathbb{C}$)	&  &  & Eqns\\ 
\hline
\hline
$10$.  & $10$  & $1$  & $9$  &   & $9$	 & Thm \ref{thm:10triples}-\ref{thm:10triplesMOD} & \ref{tab:10triples} & \ref{tab:10.EQNS4}-\ref{tab:10.EQNS4two}\\
\hline
\hline
$13$.  & $2$  & $2$  &   &   & 	 & Thm \ref{thm:13} & \ref{tab:13triples} & \ref{tab:13.EQNS} \\
\hline
\hline
($9_3$).i. & 5 &  & 5 &  & 5 & Lem \ref{lem:12Reduction-a} &  & \ref{tab:93.i.EQNS}\\
($9_3$).ii. & 4 &  & 4 & 2 & 2 & Lem \ref{lem:12Reduction-b} &  & \ref{tab:93.ii.EQNS}\\
($9_3$).iii. &  5 & 2 & 3 & 2 & 1 & Lem \ref{lem:12Reduction-c} & &  \ref{tab:93.iii.EQNS}\\
12.B.3. & 4 & 1 & 3 & 3 &  & Lem \ref{lem:12non3} & \ref{tab:12non3} & \ref{tab:12.B.3.EQNS1}\\
12.B.2. & 4 & 2 & 2 & 1 & 1 & Lem \ref{lem:12non2} & \ref{tab:12non2} & \ref{tab:12.B.2.EQNS1}\\
\hline
12. total & 22 & 5 & 17 & 8 & 9 & Thm \ref{thm:12} & & \\
\hline
\hline
11.A.       & 10 &  & 10 &  & 10 & Lem \ref{lem:11central} & \ref{tab:11central} & \ref{tab:11.A.EQNS1}-\ref{tab:11.A.EQNS2}\\
11.B.2.     & 4 &  & 4 &  & 4 & Lem \ref{lem:11generic2} & \ref{tab:11generic2} & \ref{tab:11.B.2.EQNS1} \\
11.B.3.a.   & 3 &  & 3 &  & 3 &  & \ref{tab:11generic3central} & \ref{tab:11.B.3.a.EQNS1}\\
11.B.3.b.2. & 7 &  & 7 & 1 & 6 & Lem \ref{lem:11generic3generic2} & \ref{tab:11generic3generic2} & \ref{tab:11.B.3.b.2.EQNS1}\\
11.B.3.b.1. & 13 & 1 & 12 &  & 12 & Lem \ref{lem:11generic3generic1} & \ref{tab:11generic3generic1a}-\ref{tab:11generic3generic1b} & \ref{tab:11.B.3.b.1.EQNS1}-\ref{tab:11.B.3.b.1.EQNS2}\\
\hline
11. total & 37 & 1 & 36 & 1 & 35 & Thm \ref{thm:11} & & \\
\hline
\hline
Total & 71 & 9 & 62 & 9  & 53 & Thm \ref{thm:classification}+\ref{thm:Zariskipairs} & & \\
\hline
\end{tabular}
\caption{A summary of results with hyperlinks.}
\label{tab:summaryALT}
\end{table}


%
%
Here the arrangements considered in the last column either have irreducible moduli spaces or have moduli spaces of two components which are complex conjugate.


We are primarily interested in geometric line arrangements here, and so the only cases considered are arrangements of ten lines with some number of triples in order for a geometric arrangement to occur:  $n_3\in[10,13]$ following Lemma \ref{lemma:basic} and Table \ref{tab:triplenumbers}.  Since some combinatorial but not geometric arrangements arose in these situations, we have included them.  However, this is not the complete list.

\begin{question}
\label{q:nongeometric}
What combinatorial but not geometric line arrangements of ten lines and no points of multiplicity higher than three arise with some other number of triples  $n_3\notin[10,13]$?
\end{question}

In particular, this classification gives rise to several arrangements of particular interest:  those whose combinatorics do not determine their topology.

\begin{theorem}[Main Theorem]
\label{thm:Zariskipairs}
This present classification of arrangements of ten complex lines with only triple points gives a list of {nine} potential Zariski pairs listed in Table \ref{tab:finallist}.
\end{theorem}

\begin{table}[htbp]
{ \begin{tabular}{cc}

\begin{tabular}{|ccc|ccc|ccc|c|}
\hline
$L_1$ &	$L_2$ &	$L_3$ & $L_4$	& $L_5$ &	$L_6$ &	$L_7$ & $L_8$ &	$L_9$ & $L_{10}$\\
\hline
$e_1$ & $e_1$ & $e_1$ & $e_8$	&	$e_8$ & $e_8$ & $e_4$ & $e_3$	&	$e_2$ & $D$\\
$e_2$ & $e_4$ & $e_6$ & $e_4$	&	$e_2$ & $e_3$ & $e_7$ & $e_5$	&	$e_5$ & $F$  \\
$e_3$ & $e_5$ & $e_7$ & $e_6$ &	$e_7$ & $e_9$ & $e_9$ & $e_6$	&	$e_9$ & $I$ \\
$I$ & $D$ & $F$ & $I$      &	$D$ &  $F$     &  &      	&	 &          \\
\hline
\end{tabular}

	&

\begin{tabular}{|ccc|ccc|ccc|c|}
\hline
$L_1$ &	$L_2$ &	$L_3$ & $L_4$	& $L_5$ &	$L_6$ &	$L_7$ & $L_8$ &	$L_9$ & $L_{10}$\\
\hline
$e_1$ & $e_1$ & $e_1$ & $e_8$	&	$e_8$ & $e_8$ & $e_4$ & $e_3$	&	$e_2$ & $C$ \\
$e_2$ & $e_4$ & $e_6$ & $e_4$	&	$e_2$ & $e_3$ & $e_7$ & $e_5$	&	$e_5$ & $F$  \\
$e_3$ & $e_5$ & $e_7$ & $e_6$ &	$e_7$ & $e_9$ & $e_9$ & $e_6$	&	$e_9$ & $I$ \\
$I$ &       & $F$ &  $I$     &	$C$ & $F$ &  & $C$	&	      &          \\
\hline
\end{tabular}

	\\

 &  \\

The arrangement ($9_3$).ii.DFI. & The arrangement ($9_3$).ii.CFI.\\

 &  \\

\begin{tabular}{|ccc|ccc|ccc|c|}
\hline
$e_1$ & $e_1$ & $e_1$ & $e_8$	&	$e_8$ & $e_8$ & $e_9$ & $e_9$	&	$e_9$ & $B$ \\
$e_2$ & $e_4$ & $e_6$ & $e_2$	&	$e_5$ & $e_3$ & $e_2$ & $e_4$	&	$e_3$ & $D$ \\
$e_3$ & $e_5$ & $e_7$ & $e_4$ &	$e_6$ & $e_7$ & $e_5$ & $e_7$	&	$e_6$ & $F$ \\
  & $F$  & $B$  & $B$  &	  & $D$  & $D$  &  	&	$F$  &   \\
\hline
\end{tabular}

& 

\begin{tabular}{|ccc|ccc|ccc|c|}
\hline
$e_1$ & $e_1$ & $e_1$ & $e_8$	&	$e_8$ & $e_8$ & $e_9$ & $e_9$	&	$e_9$ & $A$ \\
$e_2$ & $e_4$ & $e_6$ & $e_2$	&	$e_5$ & $e_3$ & $e_2$ & $e_4$	&	$e_3$ & $C$ \\
$e_3$ & $e_5$ & $e_7$ & $e_4$ &	$e_6$ & $e_7$ & $e_5$ & $e_7$	&	$e_6$ & $G$ \\
$G$  &   & $C$  & $A$  &	$G$  &   & $C$  &  	&	$A$  &   \\
\hline
\end{tabular}

\\

 &  \\

The arrangement ($9_3$).iii.BDF. & The arrangement ($9_3$).iii.ACG. \\

 &  \\

\begin{tabular}{|cc|cc|c|c|cc|cc|}
\hline
$L_1$ &	$L_2$ &	$L_3$ & $L_{10}$	& $L_4$ &	$L_5$ &	$L_6$ & $L_7$ &	$L_8$ & $L_9$\\
\hline
$e_1$ & $e_1$ & $e_6$ & $e_6$	&	    $e_{11}$ & $e_{2}$ & $e_{11}$ & $e_{11}$	&	$e_{12}$ & $e_{12}$\\
$e_2$ & $e_4$ & $e_7$ & $e_9$	&	    $e_{12}$ & $e_{4}$ & $e_{3}$  & $e_{8}$		&	$e_{3}$  & $e_{5}$\\
$e_3$ & $e_5$ & $e_8$ & $e_{10}$ &	$e_{1}$ & $e_7$    & $e_5$    & $e_{10}$	&	$e_{10}$ & $e_{8}$\\
		  & 		  & 		  & 			&	    $e_{7}$ & $e_9$ 	 & $e_6$ 	  & $e_{2}$		&	$e_{4}$  & $e_{9}$\\
\hline
\end{tabular}

 & 

\begin{tabular}{|c|ccc|ccc|ccc|}
\hline
$L_1$ &	$L_2$ &	$L_3$ & $L_{10}$	& $L_4$ & $L_5$ &	$L_6$ &	$L_7$ & $L_8$ &	$L_9$\\
\hline
$e_1$ & $e_1$ & $e_2$ & $e_3$	&	$e_1$ 	 & $e_2$ 		& $e_3$ 	 & $e_{10}$	&	$e_{11}$ & $e_{12}$\\
$e_2$ & $e_4$ & $e_6$ & $e_8$	&	$e_{10}$ & $e_{10}$ & $e_{11}$ & $e_{4}$	&	$e_{5}$ & $e_{4}$\\
$e_3$ & $e_5$ & $e_7$ & $e_9$	&	$e_{11}$ & $e_{12}$ & $e_{12}$ & $e_{7}$	&	$e_{6}$ & $e_{6}$\\
		  & 		  & 		  & 			&	$e_9$ 	 & $e_5$ 		& $e_7$ 	 & $e_{8}$	&	$e_{8}$ & $e_{9}$\\
\hline
\end{tabular}

\\

 &  \\

The arrangement 12.B.2.iv. & The arrangement 12.B.3.a.i.\\

 & \\

\begin{tabular}{|cc|cc|ccc|ccc|}
\hline
$e_1$ & $e_1$ & $e_2$ & $e_3$	&	$e_1$ 	 & $e_2$ 		& $e_3$ 	 & $e_{10}$	&	$e_{11}$ & $e_{12}$\\
$e_2$ & $e_3$ & $e_6$ & $e_8$	&	$e_{10}$ & $e_{10}$ & $e_{11}$ & $e_{4}$	&	$e_{7}$ & $e_{7}$\\
$e_4$ & $e_5$ & $e_7$ & $e_9$	&	$e_{11}$ & $e_{12}$ & $e_{12}$ & $e_{5}$	&	$e_{9}$ & $e_{8}$\\
		  & 		  & 		  & 			&	$e_8$ 	 & $e_9$ 		& $e_6$ 	 & $e_{6}$	&	$e_{4}$ & $e_{5}$\\
\hline
\end{tabular}

&

\begin{tabular}{|cc|cc|ccc|ccc|}
\hline
$e_1$ & $e_1$ & $e_2$ & $e_3$	&	$e_1$ 	 & $e_2$ 		& $e_3$ 	 & $e_{10}$	&	$e_{11}$ & $e_{12}$\\
$e_2$ & $e_3$ & $e_6$ & $e_8$	&	$e_{10}$ & $e_{10}$ & $e_{11}$ & $e_{4}$	&	$e_{7}$ & $e_{7}$\\
$e_4$ & $e_5$ & $e_7$ & $e_9$	&	$e_{11}$ & $e_{12}$ & $e_{12}$ & $e_{5}$	&	$e_{9}$ & $e_{8}$\\
		  & 		  & 		  & 			&	$e_8$ 	 & $e_9$ 		& $e_6$ 	 & $e_{6}$	&	$e_{5}$ & $e_{4}$\\
\hline
\end{tabular}

\\

 &  \\

The arrangement 12.B.3.b.ii. & The arrangement 12.B.3.b.iii.\\

 & \\

\begin{tabular}{|ccc|ccc|cccc|}
\hline
$L_1$ &	$L_2$ &	$L_3$ & $L_4$ & $L_5$ &	$L_6$ &	$L_7$ & $L_8$ &	$L_9$ & $L_{10}$\\
\hline
$e_1$ & $e_1$ & $e_2$ & $e_1$			&	$e_2$ 	 & $e_3$ 		& $e_4$  	 & $e_5$		&	$e_5$ 	 & $e_7$\\
$e_2$ & $e_3$ & $e_3$ & $e_{10}$	&	$e_6$ 	 & $e_4$ 		& $e_7$ 	 & $e_8$		&	$e_6$ 	 & $e_9$\\
$e_4$ & $e_6$ & $e_8$ & $e_{11}$	&	$e_{10}$ & $e_{11}$ & $e_8$ 	 & $e_{10}$	&	$e_9$ 	 & $e_{11}$\\
$e_5$ & $e_7$ & $e_9$ & 	&	 & 	&  & 	&	 & \\
\hline
\end{tabular}

 & \\

 &  \\

The arrangement 11.B.3.b.2.v. &  \\

 & 

\end{tabular}
}
\caption{The {nine} potential Zariski pairs that arise from this present classification.}
\label{tab:finallist}
\end{table}

\begin{remark}
\label{rem:12triples}
It is interesting to note that all but one of the {nine} potential Zariski pairs listed in Table \ref{tab:finallist} from Theorem \ref{thm:Zariskipairs} have exactly twelve triple points, with this last one having exactly eleven triples.
\end{remark}

This finishes the classification of geometric arrangements of ten lines begun in previous work by three of the authors.

\begin{corollary}
\label{cor:Zariskipairs}
Together with the classification of ten lines with quadruple points given in \cite{Fei:10}, which provides a list of nine potential Zariski pairs, this gives a total of {eighteen} such pairs for all complex line arrangements of ten lines satisfying { Assumption \ref{ass:reduce}}.
\end{corollary}

This work has produced a list of further counterexamples:  arrangements which are combinatorial but not geometric.

\begin{theorem}
\label{thm:nongeom}
This present classification of arrangements of ten complex lines with only triple points gives a list of {nine} non-geometric arrangements listed in Table \ref{tab:nongeom}.
\end{theorem}

\begin{table}[htbp]
{ \begin{tabular}{cc}

\begin{tabular}{|ccc|cc|ccccc|}
\hline
$L_1$ &	$L_2$ &	$L_3$ & $L_4$ &	$L_5$ &	$L_6$ & $L_7$ &	$L_8$ & $L_9$ & $L_{10}$\\
\hline
$e_1$ & $e_1$ & $e_1$ & $e_4$   &   $e_5$ &        $e_8$ & $e_2$ & $e_3$    &   $e_2$ & $e_{3}$\\
$e_2$ & $e_4$ & $e_6$ & $e_6$   &   $e_7$ &        $e_9$ & $e_4$ & $e_6$    &   $e_5$ & $e_{7}$\\
$e_3$ & $e_5$ & $e_7$ & $e_{10}$    &   $e_{10}$ & $e_{10}$ & $e_8$ & $e_8$ &   $e_{9}$ & $e_{9}$\\
\hline
\end{tabular}

 &

\begin{tabular}{|ccc|ccc|c|ccc|}
\hline
$L_1$ &	$L_2$ &	$L_3$ & $L_4$ & $L_5$ &	$L_6$ &	$L_7$ & $L_8$ &	$L_9$ & $L_{10}$\\
\hline
$e_1$ & $e_1$ & $e_2$ & $e_1$			&	$e_2$ 	 & $e_3$ 		& $e_4$  	 & $e_5$		&	$e_5$ 	 & $e_7$\\
$e_2$ & $e_3$ & $e_3$ & $e_8$			&	$e_6$ 	 & $e_4$ 		& $e_9$ 	 & $e_7$		&	$e_6$ 	 & $e_{10}$\\
$e_4$ & $e_6$ & $e_8$ & $e_{11}$	&	$e_{10}$ & $e_{10}$ & $e_{11}$ & $e_{9}$	&	$e_{8}$  & $e_{11}$\\
$e_5$ & $e_7$ & $e_9$ & 	&	 & 	&  & 	&	 & \\
\hline
\end{tabular}

\\

 &  \\

The non-geometric arrangement $(10_3)$.iv. & The non-geometric arrangement 11.B.3.b.1.vii. \\

 &  \\

\begin{tabular}{|ccc|cc|cccc|c|}
\hline
$L_1$ &	$L_2$ &	$L_3$ & $L_4$	& $L_5$ &	$L_6$ &	$L_7$ & $L_8$ &	$L_9$ & $L_{10}$\\
\hline
$e_1$ & $e_1$ & $e_1$ & $e_2$	&	$e_3$ & $e_4$ & $e_2$ & $e_3$	&	$e_4$ & $A$\\
$e_2$ & $e_5$ & $e_8$ & $e_5$	&	$e_6$ & $e_7$ & $e_7$ & $e_5$	&	$e_6$ & $C$\\
$e_3$ & $e_6$ & $e_9$ & $e_8$	&	$e_{10}$ & $e_{10}$ & $e_9$ & $e_9$	&	$e_8$ & $E$\\
$e_4$ & $e_7$ & $C$   & $e_{10}$ &	$C$ & $A$ & $E$   & $A$	  &	$E$   & \\
\hline
\end{tabular}

&

\begin{tabular}{|ccc|cc|cc|	cc|c|}
\hline
$L_1$ &	$L_2$ &	$L_3$ & $L_4$	& $L_5$ &	$L_6$ &	$L_7$ & $L_8$ &	$L_9$ & $L_{10}$\\
\hline
$e_1$ & $e_1$ & $e_1$ & $e_2$	&	$e_2$ & $e_3$ & $e_3$ & $e_4$	&	$e_4$ & $e_8$\\
$e_2$ & $e_5$ & $e_{11}$ & $e_7$	&	$e_6$ & $e_7$ & $e_5$ & $e_6$	&	$e_5$ & $e_9$\\
$e_3$ & $e_6$ & $e_{12}$ & $e_8$	&	$e_{10}$ & $e_{10}$ & $e_9$ & $e_9$	&	$e_8$ & $e_{10}$\\
$e_4$ & $e_7$ & $e_{13}$ & $e_{13}$	&	$e_{11}$ & $e_{12}$ & $e_{13}$ & $e_{12}$	&	$e_{11}$ & \\
\hline
\end{tabular}

\\

 &  \\

The non-geometric arrangement 13.i. & The non-geometric arrangement 13.ii. \\

 &  \\

\begin{tabular}{|ccc|ccc|ccc|c|}
\hline
$e_1$ & $e_1$ & $e_1$ & $e_8$	&	$e_8$ & $e_8$ & $e_9$ & $e_9$	&	$e_9$ & $B$ \\
$e_2$ & $e_4$ & $e_6$ & $e_2$	&	$e_5$ & $e_3$ & $e_2$ & $e_4$	&	$e_3$ & $E$ \\
$e_3$ & $e_5$ & $e_7$ & $e_4$ &	$e_6$ & $e_7$ & $e_5$ & $e_7$	&	$e_6$ & $G$ \\
$G$  & {$E$}  & $B$  & $B$  & $G$ & {$E$} &   &  	&	  &   \\
\hline
\end{tabular}

& 

\begin{tabular}{|ccc|ccc|ccc|c|}
\hline
$e_1$ & $e_1$ & $e_1$ & $e_8$	&	$e_8$ & $e_8$ & $e_9$ & $e_9$	&	$e_9$ & $A$ \\
$e_2$ & $e_4$ & $e_6$ & $e_2$	&	$e_5$ & $e_3$ & $e_2$ & $e_4$	&	$e_3$ & $D$ \\
$e_3$ & $e_5$ & $e_7$ & $e_4$ &	$e_6$ & $e_7$ & $e_5$ & $e_7$	&	$e_6$ & $G$ \\
$G$  &   &   & $A$  &	$G$  & $D$  & $D$  &  	&	$A$  &   \\
\hline
\end{tabular}

\\

 &  \\

The non-geometric arrangement ($9_3$).iii.B{E}G. & The non-geometric arrangement ($9_3$).iii.ADG. \\

 &  \\

\begin{tabular}{|cc|cc|ccc|ccc|}
\hline
$L_1$ &	$L_2$ &	$L_3$ & $L_{10}$	& $L_4$ &	$L_5$ &	$L_6$ & $L_7$ &	$L_8$ & $L_9$\\
\hline
$e_1$ & $e_1$ & $e_6$ & $e_6$	&	    $e_{11}$ & $e_{11}$ & $e_{11}$ & $e_{12}$	&	$e_{12}$ & $e_{12}$\\
$e_2$ & $e_4$ & $e_7$ & $e_9$	&	    $e_1$  	 & $e_2$  	& $e_3$  	 & $e_{2}$	&	$e_{3}$ & $e_{4}$\\
$e_3$ & $e_5$ & $e_8$ & $e_{10}$ &	$e_7$ 	 & $e_4$    & $e_5$    & $e_{5}$	&	$e_{7}$ & $e_{10}$\\
		  & 		  & 		  & 			&	    $e_{10}$ & $e_9$ 	  & $e_8$ 	 & $e_{6}$	&	$e_{9}$ & $e_{8}$\\
\hline
\end{tabular}

&

\begin{tabular}{|cc|cc|c|c|cc|cc|}
\hline
$L_1$ &	$L_2$ &	$L_3$ & $L_{10}$	& $L_4$ &	$L_5$ &	$L_6$ & $L_7$ &	$L_8$ & $L_9$\\
\hline
$e_1$ & $e_1$ & $e_6$ & $e_6$	&	    $e_{11}$ & $e_{2}$ & $e_{11}$ & $e_{11}$	&	$e_{12}$ & $e_{12}$\\
$e_2$ & $e_4$ & $e_7$ & $e_9$	&	    $e_{12}$ & $e_{4}$ & $e_{3}$  & $e_{8}$		&	$e_{3}$  & $e_{5}$\\
$e_3$ & $e_5$ & $e_8$ & $e_{10}$ &	$e_{1}$ & $e_7$    & $e_5$    & $e_{10}$	&	$e_{10}$ & $e_{8}$\\
		  & 		  & 		  & 			&	    $e_{6}$ & $e_9$ 	 & $e_7$ 	  & $e_{2}$		&	$e_{4}$  & $e_{9}$\\
\hline
\end{tabular}

\\

 &  \\

The non-geometric arrangement 12.B.2.i.  & The non-geometric arrangement 12.B.2.iii. \\

 &  \\

\begin{tabular}{|cc|cc|ccc|ccc|}
\hline
$e_1$ & $e_1$ & $e_2$ & $e_3$	&	$e_1$ 	 & $e_2$ 		& $e_3$ 	 & $e_{10}$	&	$e_{11}$ & $e_{12}$\\
$e_2$ & $e_3$ & $e_6$ & $e_8$	&	$e_{10}$ & $e_{10}$ & $e_{11}$ & $e_{4}$	&	$e_{5}$ & $e_{4}$\\
$e_4$ & $e_5$ & $e_7$ & $e_9$	&	$e_{11}$ & $e_{12}$ & $e_{12}$ & $e_{7}$	&	$e_{6}$ & $e_{6}$\\
		  & 		  & 		  & 			&	$e_9$ 	 & $e_5$ 		& $e_7$ 	 & $e_{8}$	&	$e_{8}$ & $e_{9}$\\
\hline
\end{tabular}

 &  \\

 &  \\
The non-geometric arrangement 12.B.3.b.i. & \\
 &

\end{tabular}
}
\caption{The {nine} non-geometric arrangements that arise from this present classification.}
\label{tab:nongeom}
\end{table}

%
%
%
%

\subsection{A note about the proofs that follow}
\label{subsec:proofs}

For the sake of organization, the proofs of these results are split into Sections \ref{sec:10}, \ref{sec:13}, \ref{sec:12}, and \ref{sec:11} according to the cases $n_3=10,13,12,11$, respectively, as ordered by straightforwardness of the proofs.  Each section begins with the local classification result, followed by its proof.  The nature of these proofs are as follows.

The case for ten triples is already classified in \cite{Grun}.  We display this result in Section \ref{sec:10} and simply check these arrangements to determine their moduli spaces.

The case for thirteen triples comes next in Section \ref{sec:13} because it is the next easiest.  Here a single line is distinguished by the number of triples on it.  We delete this line leaving an arrangement of nine lines with ten triples, and so we can apply Proposition \ref{prop:comb9lines10triples} 
 to obtain only three arrangements.  We consider casework by adding back the tenth line passing through three of the nine double points.

The next easiest case is for twelve triples, as some of these arrangements can be considered by a reduction argument similar to that above taken care of in Subsection \ref{subsec:12red}.  The first subsection of Section \ref{sec:12} produces no arrangements, and all remaining arrangements are determined by casework in Subsection \ref{subsec:12nonred}.

The remaining case for eleven triples contains the most laborious casework; this is handled in Section \ref{sec:11}.

\begin{remark}
\label{rem:centralgeneric}
It may be useful here to explicitly describe the notation used for the cases, as in 11.B.3.b.1.iv., that appears throughout.

The first number stands for the number of triples in the arrangement.  The second letter refers to the first subarrangement considered: it is (A) when this subarrangement is central and (B) when it is generic.  The number that follows gives the number of doubles of the generic subarrangement that are taken as triples in the original arrangement.  The lowercase letters (a) and (b) refer to second subarrangement considered, similarly central and generic, respectively.  The number that follows is again the number of doubles of the second generic subarrangement that are taken as triples in the original arrangement.  The roman numeral that concludes it represents the specific case number within these confines.

Thus in Figure \ref{fig:lemmas}, the arrangement $\mathcal{B}$ may be thought of as having its first three outer lines being generic (B) and forming a triangle with its next three lines also being generic (b) and forming an inner triangle.

On the other hand the Ceva arrangement $\mathcal{C}$ in Figure \ref{fig:lemmas} may be thought of as having its first three outer lines being generic (B) and forming a triangle with its next three lines being central (a) and meeting at a single point.
\end{remark}

Furthermore, the notations ($*$), ($\mathbb{C}$), and ($\mathcal{Z}$) that appear in the tables below are as described in Algorithm \ref{algorithm}:  cases for which we use Mathematica's irreducibility test, for which only complex solutions exist, and cases which yield potential Zariski pairs.

\section{Arrangements of ten lines with ten triples}
\label{sec:10}

For completion we first present the results already known about the ($10_3$) arrangements:  those with ten lines and ten triples.

\begin{theorem}\cite[Theorem 2.2.1.]{Grun}
\label{thm:10triples}
For ten lines and ten triples, there are ten combinatorial configurations and nine geometric configurations, as shown in Table \ref{tab:10triples}.
\end{theorem}

\begin{table}[htbp]
{\begin{tabular}{cc}

\begin{tabular}{|ccc|ccccccc|}
\hline
$L_1$ & $L_2$ & $L_3$ & $L_4$   & $L_5$ &   $L_6$ & $L_7$ & $L_8$ & $L_9$ & $L_{10}$\\
\hline
$e_1$ & $e_1$ & $e_1$ & $e_2$   &   $e_3$ & $e_8$ & $e_2$ & $e_3$   &   $e_4$ & $e_{5}$\\
$e_6$ & $e_2$ & $e_4$ & $e_4$   &   $e_5$ & $e_9$ & $e_6$ & $e_7$   &   $e_6$ & $e_{7}$\\
$e_7$ & $e_3$ & $e_5$ & $e_8$   &   $e_8$ & $e_{10}$ & $e_{9}$ & $e_{9}$    &   $e_{10}$ & $e_{10}$\\
\hline
\end{tabular}
&
\begin{tabular}{|ccc|cc|ccccc|}
\hline
$L_1$ & $L_2$ & $L_3$ & $L_4$   & $L_5$ &   $L_6$ & $L_7$ & $L_8$ & $L_9$ & $L_{10}$\\
\hline
$e_1$ & $e_1$ & $e_1$ & $e_4$   &       $e_5$ &    $e_8$ & $e_2$ & $e_3$    &   $e_2$ & $e_{3}$\\
$e_2$ & $e_4$ & $e_6$ & $e_6$   &       $e_7$ &    $e_9$ & $e_4$ & $e_7$    &   $e_6$ & $e_{5}$\\
$e_3$ & $e_5$ & $e_7$ & $e_{10}$    &   $e_{10}$ & $e_{10}$ & $e_8$ & $e_{8}$   &   $e_9$ & $e_{9}$\\
\hline
\end{tabular}
\\
 &  \\
The arrangement $(10_3)$.i. & The arrangement $(10_3)$.ii. \\
 &  \\
\begin{tabular}{|ccc|cc|ccccc|}
\hline
$e_1$ & $e_1$ & $e_1$ & $e_4$   &   $e_5$ &        $e_8$ & $e_2$ & $e_3$    &   $e_2$ & $e_{3}$\\
$e_2$ & $e_4$ & $e_6$ & $e_6$   &   $e_7$ &        $e_9$ & $e_4$ & $e_6$    &   $e_7$ & $e_{5}$\\
$e_3$ & $e_5$ & $e_7$ & $e_{10}$    &   $e_{10}$ & $e_{10}$ & $e_8$ & $e_8$ &   $e_{9}$ & $e_{9}$\\
\hline
\end{tabular}
&
\begin{tabular}{|ccc|cc|ccccc|}
\hline
$e_1$ & $e_1$ & $e_1$ & $e_4$   &   $e_5$ &        $e_8$ & $e_2$ & $e_3$    &   $e_2$ & $e_{3}$\\
$e_2$ & $e_4$ & $e_6$ & $e_6$   &   $e_7$ &        $e_9$ & $e_4$ & $e_6$    &   $e_5$ & $e_{7}$\\
$e_3$ & $e_5$ & $e_7$ & $e_{10}$    &   $e_{10}$ & $e_{10}$ & $e_8$ & $e_8$ &   $e_{9}$ & $e_{9}$\\
\hline
\end{tabular}
\\
 &  \\
The arrangement $(10_3)$.iii. & The non-geometric arrangement $(10_3)$.iv. \\
 &  \\
\begin{tabular}{|ccc|ccccccc|}
\hline
$e_1$ & $e_1$ & $e_1$ & $e_2$   &   $e_3$ & $e_8$ & ${e_2}$ &    ${e_4}$    &   ${ e_3}$ & $e_{5}$\\
$e_2$ & $e_4$ & $e_6$ & $e_4$   &   $e_7$ & $e_9$ & $e_5$ &      $e_6$  &   $e_6$ & $e_{7}$\\
$e_3$ & $e_5$ & $e_7$ & $e_{8}$ &   $e_8$ & $e_{10}$ & $e_{9}$ & $e_9$  &   $e_{10}$ & $e_{10}$\\
\hline
\end{tabular}
 &
\begin{tabular}{|ccc|ccccccc|}
\hline
$e_1$ & $e_1$ & $e_1$ & $e_2$   &   $e_3$ & $e_8$ & ${e_2}$ &    ${e_5}$    &   ${e_3}$ & $e_{4}$\\
$e_2$ & $e_4$ & $e_6$ & $e_4$   &   $e_7$ & $e_9$ & $e_6$ &      $e_7$  &   $e_5$ & $e_{6}$\\
$e_3$ & $e_5$ & $e_7$ & $e_{8}$ &   $e_8$ & $e_{10}$ & $e_{9}$ & $e_9$  &   $e_{10}$ & $e_{10}$\\
\hline
\end{tabular}
\\
 &  \\
The arrangement $(10_3)$.v. & The arrangement $(10_3)$.vi. \\
& \\
\begin{tabular}{|ccc|ccccccc|}
\hline
$e_1$ & $e_1$ & $e_1$ & $e_2$   &   $e_4$ & $e_5$ & ${ e_2}$ &   ${ e_3}$   &   ${e_7}$ & $e_{6}$\\
$e_2$ & $e_4$ & $e_6$ & $e_8$   &   $e_8$ & $e_7$ & $e_4$ &      $e_5$  &       $e_3$ & $e_{9}$\\
$e_3$ & $e_5$ & $e_7$ & $e_{9}$ &   $e_{10}$ & $e_8$ & $e_{6}$ & $e_9$  &       $e_{10}$ & $e_{10}$\\
\hline
\end{tabular}
 &
\begin{tabular}{|ccc|ccccccc|}
\hline
$e_1$ & $e_1$ & $e_1$ & $e_3$   &   ${e_5}$  &    $e_7$ &    $e_2$ &     ${e_{6}}$  &   $e_4$ &    $e_{2}$\\
$e_2$ & $e_4$ & $e_6$ & $e_8$   &   $e_8$    &    ${e_9}$ & ${e_{7}}$ & $e_5$   &       $e_3$ &    $e_{4}$\\
$e_3$ & $e_5$  & $e_7$ & $e_9$  &   $e_{10}$ &    $e_{10}$ & $e_8$ &     $e_9$  &       $e_{10}$ & $e_{6}$\\
\hline
\end{tabular}
\\ &
\\
The arrangement $(10_3)$.vii. & The arrangement $(10_3)$.viii. \\
& \\
\begin{tabular}{|ccc|ccccccc|}
\hline
$e_1$ & $e_1$ & $e_1$ & $e_2$   &   ${e_4}$ &  $e_6$ &    $e_5$ &    ${e_{3}}$  &   $e_2$ &    $e_{3}$\\
$e_2$ & $e_4$ & $e_6$ & $e_8$   &   $e_8$ &    $e_9$ &    ${ e_7}$ & $e_5$  &       $e_7$ &    $e_{4}$\\
$e_3$ & $e_5$ & $e_7$ & $e_9$   &   $e_{10}$ & $e_{10}$ & $e_8$ &    $e_9$  &       $e_{10}$ & $e_{6}$\\
\hline
\end{tabular}
 &
\begin{tabular}{|ccc|ccccccc|}
\hline
$e_1$ & $e_1$ & $e_1$ & $e_3$   &   ${e_{2}}$ & $e_7$ &    $e_5$ &      ${e_6}$ &   $e_4$ & $e_{2}$\\
$e_2$ & $e_4$ & $e_6$ & $e_8$   &   $e_8$ &     ${ e_9}$ & ${ e_{7}}$ & $e_5$   &   $e_3$ & $e_{4}$\\
$e_3$ & $e_5$ & $e_7$ & $e_9$   &   $e_{10}$ &  $e_{10}$ & $e_8$ &      $e_9$   &   $e_{10}$ & $e_{6}$\\
\hline
\end{tabular}
\\
 &  \\
The arrangement $(10_3)$.ix. & The arrangement $(10_3)$.x. \\
 &  \\
\end{tabular}
}
\caption{The ($10_3$) arrangements of ten lines with ten triples as found in \cite{Grun}.  Nine of the ten are geometric, and each of these nine has an irreducible moduli space.}
\label{tab:10triples}
\end{table}

\begin{remark}
\label{rem:Desargues}
The arrangement ($10_3$).i. is called the Desargues arrangement, named after Girard Desargues, a French mathematician of the seventeenth century, whose theorem in projective geometry can be stated:  two triangles are in perspective axially if and only if they are in perspective centrally.  The arrangement comes from the geometry of this theorem.
\end{remark}

We perform geometric checks on these nine arrangements and arrive at the following conclusion.

\begin{theorem}
\label{thm:10triplesMOD}
For ten lines and ten triples, the nine geometric configurations have irreducible moduli spaces.  Thus there are no potential Zariski pairs of ten lines and ten triples.
\end{theorem}

\begin{proof}
The equations of the lines for each of the ten cases are given in Tables \ref{tab:10.EQNS4} and \ref{tab:10.EQNS4two}.  These were produced in the manner described above in Section \ref{subsec:linemethod}.
\end{proof}

\begin{table}[htbp]\renewcommand{\arraystretch}{2.25}
\begin{tabular}{|c||c|c|c|c|c|c|}
\hline
 \multicolumn{7}{|c|}{($10_3$). } \\
\hline
\hline
Arr. & $y=0,z,bz$ & $x=0,z,az$ & $y=Ax$ &	$y=Bx+z$ &	$y=C(x-z)$ & $y=D(x-z)+z$ \\
\hline
\hline
($10_3$).i. & $L_3,L_2,L_1$	& $L_4,L_5,L_6$	& $L_9$:  $\frac{c}{a}$ & $L_7$:  $\frac{c(b-1)}{ab}$	&	$L_{10}$:  $\frac{c}{a-1}$ &	$L_8$:  $\frac{c(b-1)}{b(a-1)}$ \\
\hline
 & \multicolumn{6}{|c|}{with $e_{10}=(a,c)$ }\\
\hline
\hline
($10_3$).ii. & $L_3,L_2,L_1$	& $L_4,L_5,L_6$	& $L_9$:  $\frac{b(c-1)}{a(b-1)}$ & $L_7$:  $\frac{c-1}{a}$	&	$L_8$:  $\frac{c}{a-1}$ &	$L_{10}$:  $\frac{(a-1)(b-1)}{c(b-1)-(a-1)}$ \\
\hline
& \multicolumn{6}{|c|}{with $e_8=(a,c)$ and satisfying $2b-1=0$ } \\
\hline
\hline
($10_3$).iii. & $L_3,L_2,L_1$	& $L_4,L_5,L_6$	& $L_8$:  $\frac{c}{a}$ & $L_7$:  $\frac{c-1}{a}$	&	$L_9$:  $\frac{b(c-1)}{ab-a-c+1}$ &	$L_{10}$:  $\frac{c(b-1)}{ab-c}$ \\
\hline
 & \multicolumn{6}{|c|}{with $e_8=(a,c)$ and satisfying $1-a-b-c+2ab=0^*$ }\\
\hline
\hline
($10_3$).iv. & $L_3,L_2,L_1$	& $L_4,L_5,L_6$	& $L_8$:  $\frac{c}{a}$ & $L_7$:  $\frac{c-1}{a}$	&	$L_{10}$:  $\frac{bc}{ab-c}$ &	$L_9$:  $\frac{(b-1)(c-1)}{ab-a-c+1}$ \\
\hline
 & \multicolumn{6}{|c|}{with $e_8=(a,c)$ and satisfying $c-b=0$, a contradiction }\\
\hline
\hline
($10_3$).v. & $L_1,L_2,L_3$	& $L_4,L_5,L_6$	& $L_7$:  $\frac{c}{a}$ & $L_8$:  $\frac{c-1}{a}$	&	$L_9$:  $\frac{b(c-1)}{ab-a-c+1}$ &	{} \\
\hline
 & \multicolumn{6}{|c|}{with $e_9=(a,c)$, $L_{10}: y=\frac{c(b-1)}{c-a}(x-z)+bz$ and  }\\ 
 & \multicolumn{6}{|c|}{satisfying $b(c-1)(a-1)(c-a)-(ab-a-c+1)(abc-ab-ac+c)=0^*$ } \\
\hline
\hline
($10_3$).vi. & $L_1,L_2,L_3$	& $L_4,L_5,L_6$	& $L_7$:  $\frac{c}{a}$ & $L_{10}$:  $\frac{c(b-1)}{ab}$	&	$L_9$:  $\frac{b-c}{(a-1)(b-1)}$ &	{} \\
\hline
 & \multicolumn{6}{|c|}{with $e_9=(a,c)$, $L_8: y=\frac{c-b}{a-1}(x-z)+bz$ and satisfying $c(b^2-b+1)-b=0$ } \\
\hline
\end{tabular}
\caption{Equations for arrangements with ten triples.}
\label{tab:10.EQNS4}
\end{table}

\begin{table}[htbp]\renewcommand{\arraystretch}{2.25}
\begin{tabular}{|c||c|c|c|c|c|c|}
\hline
 \multicolumn{7}{|c|}{($10_3$). continued } \\
\hline
\hline
Arr. & $y=0,z,bz$ & $x=0,z,az$ & $y=Ax$ &	$y=Bx+z$ &	$y=C(x-z)$ & $y=D(x-z)+z$ \\
\hline
\hline
($10_3$).vii. & $L_2,L_3,L_1$	& $L_6,L_5,L_4$	& $L_8$:  $\frac{b(c-1)}{b-1}$ & $L_9$:  $c-1$	&	$L_7$: $\frac{b}{a-1}$ &	{} \\
\hline
 & \multicolumn{6}{|c|}{with $e_{10}=(1,c)$, $L_{10}: y=\frac{b(1-c)}{a-1}(x-z)+cz$ } \\
 & \multicolumn{6}{|c|}{and satisfying $(ab+(b-1)^2)c-(a+b-a)b=0^*$} \\
\hline
\hline
($10_3$).viii. & $L_{10},L_1,L_7$	& $L_2,L_8,L_5$	& $L_9$:  $\frac{c}{a}$ & $L_3$:  $-1$	&	{} &	{} \\
\hline
 & \multicolumn{6}{|c|}{with $e_{10}=(a,c)$, $L_4: y=\frac{c(b-1)}{a(c-1)}(x-az)+bz$, $L_6: y=\frac{c-b}{a+b-1}(x-az)+cz$ } \\
 & \multicolumn{6}{|c|}{and satisfying $c(b-a)(1-a)(a+b-1)+ab(b-c)(c-1)=0^*$ } \\
\hline
\hline
($10_3$).ix. & $L_2,L_3,L_1$	& $L_7,L_5,L_4$	& $L_8$:  $\frac{b}{bc-b+1}$ & $L_7$:  $\frac{b-1}{a}$	&	$L_{10}$:  $\frac{1}{c-1}$ &	{} \\
\hline
 & \multicolumn{6}{|c|}{with $e_6=(c,1)$, $L_6: y=\frac{b-1}{a(1-c)}(x-cz)+z$ and satisfying} \\
 & \multicolumn{6}{|c|}{$a^2b(c-1)+(b(c-1)+1)((b-1)(a-c)+a(1-c))=0$, irreducible by hand } \\
\hline
\hline
($10_3$).x. & $L_1,L_2,L_3$	& $L_9,L_5,L_6$	& $L_4$:  $c$ & $L_{10}$:  $-1$	&	{} &	{} \\
\hline
 & \multicolumn{6}{|c|}{with $e_8=(1,c)$, $L_7: y=\frac{b-c}{a-1}(x-z)+cz$, $L_8: y=\frac{b-ac}{1-a-b}(x-z+bz)+bz$ and satisfying } \\
 & \multicolumn{6}{|c|}{ $(b-ac)((1-c)(a-1)+b(b-c))+(b-1)(b-c)(1-b-a)=0$, irreducible by hand } \\
\hline
\end{tabular}
\caption{Equations for arrangements with ten triples.}
\label{tab:10.EQNS4two}
\end{table}

\section{Arrangements of ten lines with thirteen triples}
\label{sec:13}

\begin{theorem}
\label{thm:13}
There are two combinatorial configurations of ten lines and thirteen triples, given as configuration tables in Table \ref{tab:13triples}.  Neither of these are geometric, as shown by contradicting equations in Table \ref{tab:13.EQNS}.  Thus there are no potential Zariski pairs of ten lines and thirteen triples.
\end{theorem}

\begin{table}[htbp]
\begin{tabular}{cc}

\begin{tabular}{|ccc|cc|cccc|c|}
\hline
$L_1$ &	$L_2$ &	$L_3$ & $L_4$	& $L_5$ &	$L_6$ &	$L_7$ & $L_8$ &	$L_9$ & $L_{10}$\\
\hline
$e_1$ & $e_1$ & $e_1$ & $e_2$	&	$e_3$ & $e_4$ & $e_2$ & $e_3$	&	$e_4$ & $A$\\
$e_2$ & $e_5$ & $e_8$ & $e_5$	&	$e_6$ & $e_7$ & $e_7$ & $e_5$	&	$e_6$ & $C$\\
$e_3$ & $e_6$ & $e_9$ & $e_8$	&	$e_{10}$ & $e_{10}$ & $e_9$ & $e_9$	&	$e_8$ & $E$\\
$e_4$ & $e_7$ & $C$   & $e_{10}$ &	$C$ & $A$ & $E$   & $A$	  &	$E$   & \\
\hline
\end{tabular}

&

\begin{tabular}{|ccc|cc|cc|	cc|c|}
\hline
$L_1$ &	$L_2$ &	$L_3$ & $L_4$	& $L_5$ &	$L_6$ &	$L_7$ & $L_8$ &	$L_9$ & $L_{10}$\\
\hline
$e_1$ & $e_1$ & $e_1$ & $e_2$	&	$e_2$ & $e_3$ & $e_3$ & $e_4$	&	$e_4$ & $e_8$\\
$e_2$ & $e_5$ & $e_{11}$ & $e_7$	&	$e_6$ & $e_7$ & $e_5$ & $e_6$	&	$e_5$ & $e_9$\\
$e_3$ & $e_6$ & $e_{12}$ & $e_8$	&	$e_{10}$ & $e_{10}$ & $e_9$ & $e_9$	&	$e_8$ & $e_{10}$\\
$e_4$ & $e_7$ & $e_{13}$ & $e_{13}$	&	$e_{11}$ & $e_{12}$ & $e_{13}$ & $e_{12}$	&	$e_{11}$ & \\
\hline
\end{tabular}

\\

 &  \\

The non-geometric arrangement 13.i. & The non-geometric arrangement 13.ii. \\

 &  
 
\end{tabular}
\caption{The two combinatorial (but not geometric) arrangements of ten lines with thirteen triples.}
\label{tab:13triples}
\end{table}

\begin{table}[htbp]\renewcommand{\arraystretch}{2.25}
\begin{tabular}{|c||c|c|c|c|c|c|}
\hline
 \multicolumn{7}{|c|}{13. } \\
\hline
\hline
Arr. & $y=0,z,bz$ & $x=0,z,az$ & $y=Ax$ &	$y=Bx+z$ &	$y=C(x-z)$ & $y=Ex+bz$ \\
\hline
\hline
13.i. & $L_3,L_2,L_1$	& $L_4,L_5,L_6$	& $L_9$:  $1$ & $L_8$:  $a-1$	&	$L_{10}$:  $\frac{a^2-a+1}{a-1}$ &	$L_7$: $\frac{1-a}{a}$ \\
\hline
 & \multicolumn{6}{|c|}{satisfying $a-b=0$, $a^2+1=0$, and $a^3-3a^2+2a-1=0$, a contradiction }\\
\hline
\hline
13.ii. & $L_1,L_2,L_3$	& $L_5,L_6,L_{10}$	&$L_4$:  $1$ & $L_8$:  $b-1$	&	$L_7$:  $\frac{b}{b-1}$ & $L_{9}$:  $\frac{(a-b)}{a}$ \\
\hline
& \multicolumn{6}{|c|}{satisfying $a(b^2+b-1)-b(2b-1)=0$, $a(-b^2+b+1)-b=0$,}\\
& \multicolumn{6}{|c|}{ and $a(b-(b-1)^2)-(2b-1)=0$, a contradiction  } \\
\hline
\end{tabular}
\caption{Equations for (non-geometric) arrangements with thirteen triples.}
\label{tab:13.EQNS}
\end{table}

\begin{proof}
Let $\mathcal{A}$ be an arrangement of ten lines with thirteen triples.  Nine of the lines pass through exactly four triples, and one line, say $L_{10}$, passes through exactly three triples.  Let $\mathcal{A}':=\mathcal{A}\backslash\{L_{10}\}$ be the deletion of $L_{10}$ from the original arrangement.  Then $\mathcal{A}'$ has nine lines with ten triples, each line passing through at least three triples.



According to Proposition \ref{prop:comb9lines10triples}, the arrangement $\mathcal{A}'$ is either the Nazir-Yoshinaga arrangement, the Pappus arrangement, or a non-geometrically realizable arrangement.

The Nazir-Yoshinaga arrangement $\mathcal{A}^{\pm i}$ has six doubles, and so when the tenth line is added back, it must pass through exactly three of these.  The six doubles can be arranged as the vertices of a hexagon with lines as edges passing between them in this cyclic manner:  $L_3,L_5,L_7,L_9,L_8,L_6$.  Thus for a line to pass through three of these doubles, it must skip every other one in this cycle.  This should give two possibilities, but the symmetry $(L_1 L_2)(L_5 L_6)(L_7 L_8)$ sends one to the other.

The combinatorial arrangement here is given as 13.i. in Table \ref{tab:13triples} and Figure \ref{fig:13triples1}, but the equations given in Table \ref{tab:13.EQNS} preclude it from being geometric.

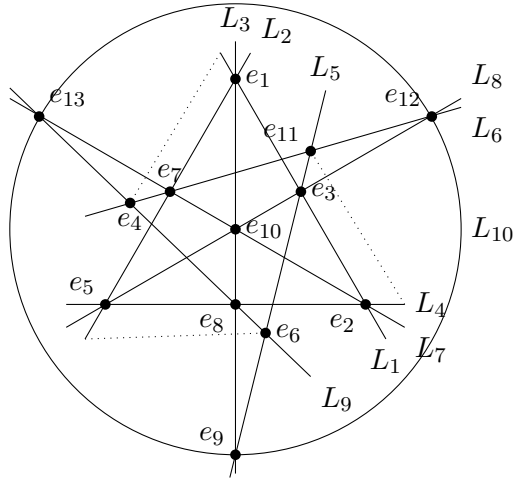
\begin{figure}[htbp]
\begin{center}
\begin{tikzpicture}


\draw[domain=-.2:2,smooth,variable=\x] plot (\x,{-1.73*\x+2}) node[below] {$L_1$};
\draw[domain=-2:.2,smooth,variable=\x] plot (\x,{1.73*\x+2}) node[above right] {$L_2$};
\draw[domain=-3.25:2.5,smooth,variable=\y] plot (0,\y) node[above] {$L_3$};
\draw[domain=-2.25:2.25,smooth,variable=\x] plot (\x,{-1}) node[right] {$L_4$};
\draw[domain=-.075:1.2,smooth,variable=\x] plot (\x,{4.04*\x-3}) node[above] {$L_5$};
\draw[domain=-2:3,smooth,variable=\x] plot (\x,{.29*(\x+.87)+.5}) node[below right] {$L_6$};
\draw[domain=-3:2.25,smooth,variable=\x] plot (\x,{-.58*\x}) node[below right] {$L_7$};
\draw[domain=-2.25:3,smooth,variable=\x] plot (\x,{.58*\x}) node[above right] {$L_8$};
\draw[domain=-3:1,smooth,variable=\x] plot (\x,{-.96*\x-1}) node[below right] {$L_9$};
\draw (0,0) circle (3);
\draw (3,0) node [right] {$L_{10}$};

	\fill[color=black] (0,2) circle (2pt) node[right] {$e_1$};
	\fill[color=black] (1.73,-1) circle (2pt) node[below left] {$e_2$};
	\fill[color=black] (.87,.5) circle (2pt) node[right] {$e_3$};
	\fill[color=black] (-1.4,.35) circle (2pt) node[below] {$e_4$};
	\fill[color=black] (-1.73,-1) circle (2pt) node[above left] {$e_5$};
	\fill[color=black] (.4,-1.38) circle (2pt) node[right] {$e_6$};
	\fill[color=black] (-.87,.5) circle (2pt) node[above] {$e_7$};
	\fill[color=black] (0,-1) circle (2pt) node[below left] {$e_8$};
	\fill[color=black] (0,-3) circle (2pt) node[above left] {$e_9$};
	\fill[color=black] (0,0) circle (2pt) node[right] {$e_{10}$};
	\fill[color=black] (1,1.04) circle (2pt) node[above left] {$e_{11}$};
	\fill[color=black] (2.61,1.5) circle (2pt) node[above left] {$e_{12}$};
	\fill[color=black] (-2.61,1.5) circle (2pt) node[above right] {$e_{13}$};

	\draw[-, dotted] (-2,-1.46) -- (.4,-1.38);
	\draw[-, dotted] (2.25,-1) -- (1,1.04);
	\draw[-, dotted] (-.2,2.35) -- (-1.4,.35);

\end{tikzpicture}
	\caption{The combinatorial arrangement 13.i. of ten lines with thirteen triples.}
	\label{fig:13triples1}
\end{center}
\end{figure}



The Pappus arrangement has six doubles, and so when the tenth line is added back, it must pass through exactly three of these.  The six doubles can be arranged as the vertices of two disjoint triangles with lines as edges passing between them in this (disjoint) cyclic manner:  $L_1,L_7,L_6$ and $L_2,L_8,L_5$.  Thus no line may pass through three of these, and so there is not even a combinatorial arrangement that arises in this way.

The non-geometrically realizable arrangement of nine lines with ten triples has six doubles, and so when the tenth line is added back, it must pass through exactly three of these.  The six doubles can be arranged as the vertices of a hexagon with lines as edges passing between them in this cyclic manner:  $L_2,L_9,L_4,L_3,L_7,L_6$.  Thus for a line to pass through three of these doubles, it must skip every other one in this cycle.  This should give two possibilities, but symmetry sends one to the other.

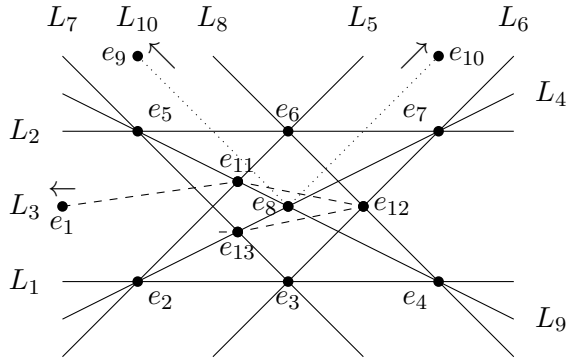
\begin{figure}[htbp]
\begin{center}
\begin{tikzpicture}
	\draw[-] (-1,0) -- +(6,0);
	\draw[-] (-1,2) -- +(6,0);
	\draw (-1.5,0) node {$L_1$};
	\draw (-1.5,2) node {$L_2$};

	\draw[-] (-1,-1) -- (3,3);
	\draw[-] (5,-1) -- (1,3);
	\draw[-] (-1,3) -- (3,-1);
	\draw[-] (5,3) -- (1,-1);
	\draw (-1,3.5) node {$L_7$};
	\draw (1,3.5) node {$L_8$};
	\draw (3,3.5) node {$L_5$};
	\draw (5,3.5) node {$L_6$};

	\draw[-] (-1,-.5) -- (5,2.5);
	\draw[-] (5,-.5) -- (-1,2.5);
	\draw (5.5,-.5) node {$L_9$};
	\draw (5.5,2.5) node {$L_4$};

	\draw (-1.5,1) node {$L_3$};
	\draw (0,3.5) node {$L_{10}$};

	\fill[color=black] (-1,1) circle (2pt) node[below] {$e_1$} node[above] {$\leftarrow$};
	\fill[color=black] (0,0) circle (2pt) node[below right] {$e_2$};
	\fill[color=black] (2,0) circle (2pt) node[below] {$e_3$};
	\fill[color=black] (4,0) circle (2pt) node[below left] {$e_4$};
	\fill[color=black] (0,2) circle (2pt) node[above right] {$e_5$};
	\fill[color=black] (2,2) circle (2pt) node[above] {$e_6$};
	\fill[color=black] (4,2) circle (2pt) node[above left] {$e_7$};
	\fill[color=black] (2,1) circle (2pt) node[left] {$e_8$};
	\fill[color=black] (0,3) circle (2pt) node[left] {$e_9$} node[right] {$\nwarrow$};
	\fill[color=black] (4,3) circle (2pt) node[right] {$e_{10}$} node[left] {$\nearrow$};
	\fill[color=black] (1.33,1.33) circle (2pt) node[above] {$e_{11}$};
	\fill[color=black] (3,1) circle (2pt) node[right] {$e_{12}$};
	\fill[color=black] (1.33,.66) circle (2pt) node[below] {$e_{13}$};

	\fill[color=black] (2,1) circle (2pt);
	\fill[color=black] (3,1) circle (2pt);
	\fill[color=black] (1.33,.66) circle (2pt);
	\fill[color=black] (1.33,1.33) circle (2pt);


\draw[dashed] (-1,1) -- (1.33,1.33) -- (3,1) -- (1.33,.66) -- (1,.66);
\draw[dotted] (0,3) -- (2,1) -- (4,3);

\end{tikzpicture}
	\caption{The combinatorial arrangement 13.ii. of ten lines with thirteen triples.}
	\label{fig:13triples2}
\end{center}
\end{figure}

  The combinatorial arrangement here is given as 13.ii. in Table \ref{tab:13triples} and Figure \ref{fig:13triples2}, but the equations given in Table \ref{tab:13.EQNS} preclude it from being geometric. 
\end{proof}

\section{Arrangements of ten lines with twelve triples}
\label{sec:12}

\begin{theorem}
\label{thm:12}
There are twenty-two combinatorial configurations of ten lines and twelve triples.  
  All but three of these are geometric, and four of these remaining nineteen are potential Zariski pairs.
\end{theorem}

\begin{proof}
Let $\mathcal{A}$ be arrangement of ten lines with twelve triples.  There are exactly four lines that have exactly three triple points on each of them; call these $L_1$, $L_2$, $L_3$, and $L_{10}$.  We consider the possible subarrangements of just these four lines.  Observe that there can be no central subarrangement of these four lines because we do not consider quadruple points.  Other than the generic subarrangement with six doubles, there is only one other subarrangement:  $\mathpzc{E}$ with one triple and three doubles, resembling an artist's easel.

\subsection{Easel subarrangement}
\label{subsec:12easel}

Let $e_1$ be this triple on $L_1$, $L_2$, and $L_3$, and consider the three double points on the subarrangement $\mathpzc{E}$ of these lines with $L_{10}$.

If all three of these doubles in $\mathpzc{E}$ are indeed triples in $\mathcal{A}$, then the lines $L_1$, $L_2$, $L_3$, and $L_{10}$ contribute a total of seven triples, leaving five triples left among the remaining six lines.  However, this contradicts Lemma \ref{lem:6lines4pts}.

If exactly two of these three doubles in $\mathpzc{E}$ are indeed triples in $\mathcal{A}$, then the lines $L_1$, $L_2$, $L_3$, and $L_{10}$ contribute a total of eight triples, leaving four triples left among the remaining six lines.  By Lemma \ref{lem:6lines4triples}, there is a unique subarrangement $\mathcal{C}$ of these six lines with exactly four triples.  However, this subarrangement $\mathcal{C}$ has only three doubles, and since $\mathpzc{E}$ has no more doubles that can become triples in $\mathcal{A}$, there are not enough triples to account for the remaining five in $\mathcal{A}$, a contradiction.

If exactly one of these three doubles in $\mathpzc{E}$ is indeed a triple in $\mathcal{A}$, then the lines $L_1$, $L_2$, $L_3$, and $L_{10}$ contribute a total of nine triples, leaving three triples left among the remaining six lines.  By Lemma \ref{lem:6lines3triples}, there is a unique subarrangement $\mathcal{B}$ of these six lines with exactly three triples.  However, this subarrangement $\mathcal{B}$ has only six doubles, and since $\mathpzc{E}$ has no more doubles that can become triples in $\mathcal{A}$, there are not enough triples to account for the remaining seven in $\mathcal{A}$, a contradiction.

 If none of these three doubles in $\mathpzc{E}$ is indeed a triple in $\mathcal{A}$, then we use a reduction argument as in the case of thirteen triples.  Let $\mathcal{A}':=\mathcal{A}\backslash\{L_{10}\}$ be the deletion of $L_{10}$ from the original arrangement.  Then $\mathcal{A}'$ has nine lines with nine triples, each line passing through \emph{exactly} three triples, and we consider this case below.


\subsection{A reduction as in thirteen triples}  
\label{subsec:12red}  
Consider the four lines $L_1$, $L_2$, $L_3$, and $L_{10}$ with exactly three triples on each, and suppose that in general the line $L_{10}$ intersects the other three lines at points which are only doubles in the original arrangment $\mathcal{A}$.
%
  Let $\mathcal{A}':=\mathcal{A}\backslash\{L_{10}\}$ be the deletion of $L_{10}$ from the original arrangement.  Then $\mathcal{A}'$ has nine lines with nine triples, each line passing through \emph{exactly} three triples.

By Proposition \ref{prop:9-3} there are only three such combinatorial or geometric arrangements:  $(9_3)$.i., $(9_3)$.ii., and $(9_3)$.iii., as appearing in Figure \ref{fig:9-3}.

\begin{lemma}
\label{lem:12Reduction-a}
For the reduction of twelve triples giving case $(9_3)$.i., we have five combinatorial configurations, each with a one-parameter family of geometric configurations as given in Table \ref{tab:93.i.EQNS}, and so each has an irreducible moduli space.  Thus there are no potential Zariski pairs.
\end{lemma}

\begin{proof}
We begin by considering the combinatorial configuration without worrying about its geometric realization as in Figure \ref{fig:931}.  The nine doubles $A,\ldots,I$ are labelled; observe that these are arranged in three disjoint triangles using the nine lines: $A,C,E$; $B,D,F$; and $G,H,I$.  Thus for the line $L_{10}$ to pass through three of these, it must pass through exactly one from each triangle.

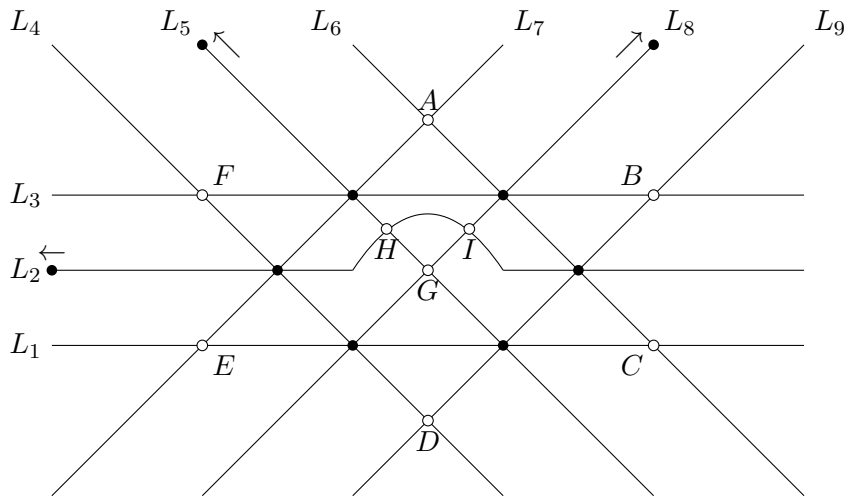
\begin{figure}[htbp]
\begin{center}
\begin{tikzpicture}
	\draw[-] (9,1) -- +(-10,0) node[left] {$L_1$};
	\draw[-] (3,2) -- (-1,2) node[left] {$L_2$};
	\draw[-] (5,2) -- (9,2);
	\draw[-] (9,3) -- +(-10,0) node[left] {$L_3$};

	\draw (3,2) .. controls (3.66,3) and (4.33,3) .. (5,2);

	\draw[-] (-1,-1) -- (5,5) node[above right] {$L_7$};
	\draw[-] (1,-1) -- (7,5) node[above right] {$L_8$};
	\draw[-] (3,-1) -- (9,5) node[above right] {$L_9$};
	\draw[-] (5,-1) -- (-1,5) node[above left] {$L_4$};
	\draw[-] (7,-1) -- (1,5) node[above left] {$L_5$};
	\draw[-] (9,-1) -- (3,5) node[above left] {$L_6$};



	\fill[color=white] (4,4) circle (2pt);
	\draw (4,4) circle (2pt) node[above] {$A$};
	\fill[color=white] (7,3) circle (2pt);
	\draw (7,3) circle (2pt) node[above left] {$B$};
	\fill[color=white] (7,1) circle (2pt);
	\draw (7,1) circle (2pt) node[below left] {$C$};
	\fill[color=white] (4,0) circle (2pt);
	\draw (4,0) circle (2pt) node[below] {$D$};
	\fill[color=white] (1,1) circle (2pt);
	\draw (1,1) circle (2pt) node[below right] {$E$};
	\fill[color=white] (1,3) circle (2pt);
	\draw (1,3) circle (2pt) node[above right] {$F$};
	\fill[color=white] (4,2) circle (2pt);
	\draw (4,2) circle (2pt) node[below] {$G$};
	\fill[color=white] (4.55,2.55) circle (2pt);
	\draw (4.55,2.55) circle (2pt) node[below] {$I$};
	\fill[color=white] (3.45,2.55) circle (2pt);
	\draw (3.45,2.55) circle (2pt) node[below] {$H$};

	\fill[color=black] (1,5) circle (2pt) node[right] {$\nwarrow$}; 
	\fill[color=black] (7,5) circle (2pt) node[left] {$\nearrow$}; 
	\fill[color=black] (-1,2) circle (2pt) node[above] {$\leftarrow$}; 

\foreach \x/\y in {5/3,6/2,5/1,3/1,2/2,3/3}
	\fill[color=black] (\x,\y) circle (2pt);

\end{tikzpicture}
	\caption{A combinatorial arrangement for $(9_3)$.i. showing its symmetry.}
	\label{fig:931}
\end{center}
\end{figure}

Without loss of generality we may assume the line passes through $C$.  Up to symmetry there are just two choices for the next triangle: { $D(=B)$ and $F$. } In the first case there is no symmetry left so there are three choices ($G,H,I$), but in the second case there is still symmetry leaving two choices {$G(=I)$ and $H$.} 

Equations for a geometric realization of $(9_3)$.i. given in Table \ref{tab:93EQNS} give coordinates for these doubles, and the equations given in Table \ref{tab:93.i.EQNS} demonstrate that all five of these combinatorial arrangements are geometric with irreducible moduli spaces.
\end{proof}

\begin{table}[htbp]\renewcommand{\arraystretch}{2.25}
\begin{tabular}{|c||c|c|}
\hline
 \multicolumn{3}{|c|}{Adding a tenth line to ($9_3$).i. }\\
\hline
\hline
$10^{th}$ line & Equation & satisfying  \\
\hline
\hline
CDG & $y=\frac{1}{a(a-1)}(x-az)$	& $b+1=0$ \\
\hline
\hline
CDH & $y=\frac{1}{a(a-1)}(x-az)$	& $a+1=0$ \\
\hline
\hline
CDI & $y=\frac{1}{a(a-1)}(x-az)$	& $ab-1=0$ \\
\hline
\hline
CFG & $y=\frac{-b}{a}x+bz$	& $a-2=0$ \\
\hline
\hline
CFH & $y=\frac{-b}{a}x+bz$	& $a(b-1)-b=0$ \\
\hline
\end{tabular}
\caption{Equations for (geometric) arrangements from ($9_3$).i.}
\label{tab:93.i.EQNS}
\end{table}

\begin{lemma}
\label{lem:12Reduction-b}
For the reduction of twelve triples giving case $(9_3)$.ii., we have four combinatorial configurations, all geometric as given in Table \ref{tab:93.ii.EQNS}: one with one real solution, one with irreducible moduli space by complex conjugation, and two potential Zariski pairs.
\end{lemma}

\begin{proof}
We begin by considering the combinatorial configuration without worrying about its geometric realization as in Figure \ref{fig:932}.  The nine doubles $A,\ldots,I$ are labelled; observe that these are arranged in a nine-gon.  The line $L_{10}$ must pass through three of these but cannot pass through any two that are adjacent.

Without loss of generality we may assume the line passes through $F=(a,0)$.  There are four possible cyclic partitions of the six remaining doubles $A$, $B$, $C$, $D$, $H$, $I$ into three nonempty sets:  the partition (1,1,4) giving {$D,F,H$}, the partition (1,2,3) giving {$D,F,I$}, the partition (1,3,2) giving {$D,F,A$}, and the partition (2,2,2) giving {$C,F,I$}.

Equations for a geometric realization of $(9_3)$.ii. given in Table \ref{tab:93EQNS} give coordinates for these doubles, and the equations given in Table \ref{tab:93.ii.EQNS} demonstrate that all four of these combinatorial arrangements are geometric.  The first gives complex conjugate solutions, the second is a potential Zariski pair, the third gives a single (real) solution, and the last gives two real (Galois conjugate) solutions and thus is also a potential Zariski pair.  Thus two of these give potential Zariski pairs that appear on the final list of Theorem \ref{thm:Zariskipairs}. 
\end{proof}

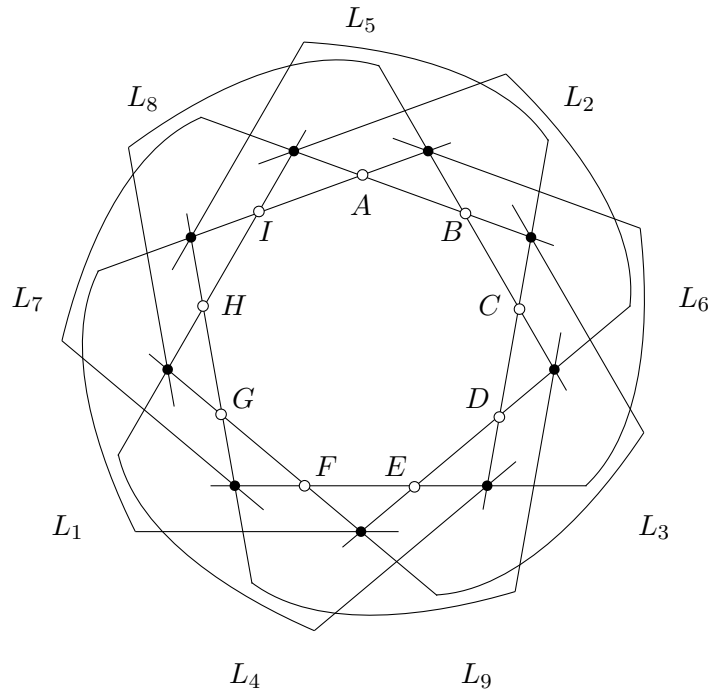
\begin{figure}[htbp]
\begin{center}
\begin{tikzpicture}

\foreach \x in {0,...,8}
	\draw[rotate around={(360*\x/9):(0,0)}] (-2,-2) -- (3,-2);
\foreach \x in {0,...,8}
	\fill[rotate around={(360*\x/9):(0,0)},color=black] (0,-2.61) circle (2pt);
\foreach \x in {0,...,8}
	\fill[rotate around={(360*\x/9):(0,0)},color=white] (-.75,-2) circle (2pt);
\foreach \x in {0,...,8}
	\draw[rotate around={(360*\x/9):(0,0)}] (-.75,-2) circle (2pt);
\foreach \x in {0,...,8}
	\draw[rotate around={(360*\x/9):(0,0)}] (-3,-2.61) -- (.5,-2.61);
\foreach \x in {0,...,8}
	\draw[rotate around={(360*\x/9):(0,0)}] (-3,-2.61) .. controls (-3.5,-1.61) and (-4,-.14) .. (-3.49,.86);

	\draw (0,2) node[below] {$A$};
	\draw (1.5,1.65) node[below left] {$B$};
	\draw (-1.5,1.65) node[below right] {$I$};
	\draw (2,.4) node[left] {$C$};
	\draw (-2,.4) node[right] {$H$};
	\draw (1.85,-1.1) node[above left] {$D$};
	\draw (-1.85,-1.1) node[above right] {$G$};
	\draw (.75,-2) node[above left] {$E$};
	\draw (-.75,-2) node[above right] {$F$};

	\draw (0,4.5) node[below] {$L_5$};
	\draw (2.9,3.45) node[below] {$L_2$};
	\draw (-2.9,3.45) node[below] {$L_8$};
	\draw (4.43,.78) node[below] {$L_6$};
	\draw (-4.43,.78) node[below] {$L_7$};
	\draw (3.90,-2.25) node[below] {$L_3$};
	\draw (-3.90,-2.25) node[below] {$L_1$};
	\draw (1.54,-4.23) node[below] {$L_9$};
	\draw (-1.54,-4.23) node[below] {$L_4$};

\end{tikzpicture}
	\caption{A combinatorial arrangement for $(9_3)$.ii. showing its symmetry.}
	\label{fig:932}
\end{center}
\end{figure}

\begin{table}[htbp]\renewcommand{\arraystretch}{2.25}
\begin{tabular}{|c||c|c|}
\hline
 \multicolumn{3}{|c|}{Adding a tenth line to ($9_3$).ii.  }\\
\hline
\hline
$10^{th}$ line & Equation & satisfying  \\
\hline
\hline
{ DFH} & $y=\frac{1}{1-a}(x-a)$	& $a-(b^2-b+1)=0$, $b^2+1=0^\mathbb{C}$ \\
\hline
\hline
{ DFI} & $y=\frac{1}{1-a}(x-a)$	& $a-(b^2-b+1)=0$, $b^3-2b^2+b-1=0^\mathcal{Z}$ \\
\hline
\hline
{ DFA} & $y=\frac{1}{1-a}(x-a)$	& $a-(b^2-b+1)=0$, $b-2=0$ \\
\hline
\hline
{ CFI} & $y=\frac{-b}{a}(x-a)$	& $a-(b^2-b+1)=0$, $b^2-b-1=0^\mathcal{Z}$ \\
\hline
\end{tabular}
\caption{Equations for (geometric) arrangements from ($9_3$).ii.}
\label{tab:93.ii.EQNS}
\end{table}

\begin{lemma}
\label{lem:12Reduction-c}
For the reduction of twelve triples giving case $(9_3)$.iii., we have five combinatorial configurations, three of which are geometric as given in Table \ref{tab:93.iii.EQNS}: one with irreducible moduli space by complex conjugation and two with two real solutions, giving two potential Zariski pairs.
\end{lemma}

\begin{proof}
We begin by considering the combinatorial configuration without worrying about its geometric realization as in Figure \ref{fig:933}.  The nine doubles $A,\ldots,I$ are labelled; observe that these are arranged in a hexagon $A,B,C,D,E,F$ and a triangle $G,H,I$.  The line $L_{10}$ must pass through three of these points, and so it can pass through at most one point on the triangle.

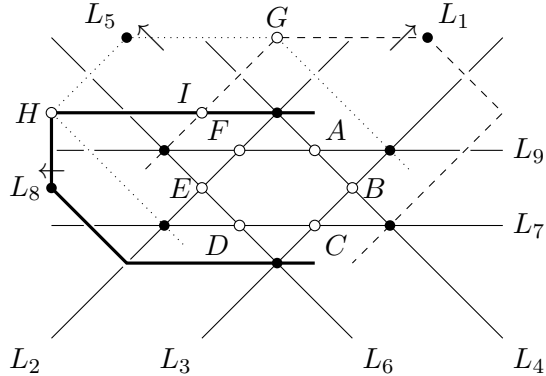
\begin{figure}[htbp]
\begin{center}
\begin{tikzpicture}
	\draw[-] (-1,.5) -- +(6,0) node[right] {$L_7$};
	\draw[-] (-1,1.5) -- +(6,0) node[right] {$L_9$};


	\draw[-] (3,3) -- (-1,-1) node[below left] {$L_2$};
	\draw[-] (1,3) -- (5,-1) node[below right] {$L_4$};
	\draw[-] (-1,3) -- (3,-1) node[below right] {$L_6$};
	\draw[-] (5,3) -- (1,-1) node[below left] {$L_3$};


	\fill[color=white] (2.5,1.5) circle (2pt);
	\draw (2.5,1.5) circle (2pt) node[above right] {$A$};
	\fill[color=white] (3,1) circle (2pt);
	\draw (3,1) circle (2pt) node[right] {$B$};
	\fill[color=white] (2.5,.5) circle (2pt);
	\draw (2.5,.5) circle (2pt) node[below right] {$C$};
	\fill[color=white] (1.5,1.5) circle (2pt);
	\draw (1.5,1.5) circle (2pt) node[above left] {$F$};
	\fill[color=white] (1,1) circle (2pt);
	\draw (1,1) circle (2pt) node[left] {$E$};
	\fill[color=white] (1.5,.5) circle (2pt);
	\draw (1.5,.5) circle (2pt) node[below left] {$D$};

\foreach \x/\y in {2/2,2/0,3.5/1.5,3.5/.5,.5/1.5,.5/.5}
	\fill[color=black] (\x,\y) circle (2pt);


	\fill[color=black] (0,3) circle (2pt) node[right] {$\nwarrow$} node[above left] {$L_5$}; 
	\fill[color=black] (4,3) circle (2pt) node[left] {$\nearrow$} node[above right] {$L_1$}; 
	\fill[color=black] (-1,1) circle (2pt) node[above] {$\leftarrow$} node[left] {$L_8$}; 

	\fill[color=white] (-.5,2.5) circle (2pt);
	\fill[color=white] (0,2) circle (2pt);
	\fill[color=white] (1.5,2.5) circle (2pt);
	\fill[color=white] (1,3) circle (2pt);
	\fill[color=white] (2.5,2.5) circle (2pt);
	\fill[color=white] (3,3) circle (2pt);
	\fill[color=white] (4.5,2.5) circle (2pt);

	\fill[color=white] (-1,1.5) circle (2pt);
	\fill[color=white] (-.5,1.5) circle (2pt);
	\fill[color=white] (4.5,1.5) circle (2pt);

	\fill[color=white] (-.5,.5) circle (2pt);
	\fill[color=white] (0,0) circle (2pt);

\draw[very thick] (2.5,2) -- (2,2) -- (0,2) -- (-1,2)-- (-1,1) -- (0,0) -- (2,0) -- (2.5,0);
\draw[dotted] (3.75,1.25) -- (3.5,1.5) -- (2,3) -- (0,3) -- (-1,2) -- (.5,.5) -- (.75,.25);
\draw[dashed] (.25,1.25) -- (.5,1.5) -- (2,3) -- (4,3) -- (5,2) -- (3.5,.5) -- (3,0);

	\fill[color=white] (1,2) circle (2pt);
	\draw (1,2) circle (2pt) node[above left] {$I$};
	\fill[color=white] (2,3) circle (2pt);
	\draw (2,3) circle (2pt) node[above] {$G$};
	\fill[color=white] (-1,2) circle (2pt);
	\draw (-1,2) circle (2pt) node[left] {$H$};

\end{tikzpicture}
	\caption{A combinatorial arrangement for $(9_3)$.iii. showing its symmetry.}
	\label{fig:933}
\end{center}
\end{figure}

\begin{table}[htbp]\renewcommand{\arraystretch}{2.25}
\begin{tabular}{|c||c|c|}
\hline
 \multicolumn{3}{|c|}{Adding a tenth line to ($9_3$).iii. }\\
\hline
\hline
$10^{th}$ line & Equation & satisfying  \\
\hline
\hline
BDF & $y=\frac{a-b}{a}x+bz$	& $a-b-1=0$, $a^2-a-1=0^\mathcal{Z}$ \\
\hline
\hline
ACG & $y=\frac{1}{1-a}(x-bz)+bz$	& $a-b-1=0$, $a^2-a-1=0^\mathcal{Z}$ \\
\hline
\hline
AEG & $y=\frac{a-1-ab}{a(a-1)}(x-az)+z$	& $a-b-1=0$, $b^2+b+1=0^\mathbb{C}$ \\
\hline
\hline
ADG & $y=\frac{a-b-1}{a-1}(x-az)+az$	& $a-b-1=0$, $a-1=0$, a contradiction \\
\hline
\hline
BEG & $y=\frac{1-b}{a}x+bz$	& $a-b-1=0$, $b^2+1=0$, a contradiction \\
    & & (see Remark \ref{rem:BEG}) \\
\hline
\end{tabular}
\caption{Equations for arrangements from ($9_3$).iii.}
\label{tab:93.iii.EQNS}
\end{table}

If the line passes through none of the points on the triangle, this gives just one choice up to symmetry:  $B,D,F$.  If the line passes through one point on the triangle {(say G)} and two on the hexagon, these two on the hexagon are either one apart, in which case there are two choices up to symmetry, $AC$(=$DF$) or $AE$($=FB=DB=CE$), or two apart, in which case there are two choices up to symmetry, $BE$ or $AD$(=$FC$).

Equations for a geometric realization of $(9_3)$.iii. given in Table \ref{tab:93EQNS} give coordinates for these doubles, and the equations given in Table \ref{tab:93.iii.EQNS} demonstrate that only {three} of these five combinatorial arrangements are geometric: {one} with {an} irreducible moduli space by complex conjugation and two with two real solutions, which give potential Zariski pairs.
\end{proof}



\begin{remark}
\label{rem:BEG}
Observe that geometrically the line BEG must pass through the triple $L_7\cap L_8\cap L_9$, making it a quadruple.  This arrangement already appears in \cite[Proposition 5.1]{Fei:10}, as it has ten lines and a quadruple.  However in the non-geometric setting we need not worry about this quadruple.
\end{remark}

\subsection{Remaining generic non-reduction subarrangements}  
\label{subsec:12nonred}  
Lastly we consider the remaining cases when the four initial lines form a generic subarrangement but there does not exist a line of it whose three doubles all fail to be triples in $\mathcal{A}$, the original arrangement.

We proceed according to how many of the six doubles are in fact triples in $\mathcal{A}$.

Suppose all six doubles are triples in $\mathcal{A}$.  Then the remaining six lines must form a subarrangement with exactly six triples, contradicting Lemma \ref{lem:6lines4pts}.

Suppose five of the six doubles are triples in $\mathcal{A}$; then two additional triples must lie on this subarrangement.  However the remaining six lines must form another subarrangement with exactly five triples, contradicting Lemma \ref{lem:6lines4pts}.

Suppose four of the six doubles are triples in $\mathcal{A}$; then four additional triples must lie on this subarrangement.  The remaining six lines must form another subarrangement with exactly four triples.  By Lemma \ref{lem:6lines4triples} there is a unique subarrangement $\mathcal{C}$ satisfying this.  However this subarrangement $\mathcal{C}$ has only three double points, fewer than the four needed to produce triples on the original subarrangement, a contradiction.

Suppose three of the six doubles are triples in $\mathcal{A}$.  Some cases already accounted for are those in which the intersections of a
given line with the other three lines are not triple points; these were
handled in Subsection \ref{subsec:12red}.  Those not accounted for are:
\begin{itemize}
	\item[(a)] those in which these three triples are colinear and
	\item[(b)] those in which these three triples are not colinear but do not form a triangle.
\end{itemize}


Here six additional triples must lie on this subarrangement.  Then the remaining six lines must form another subarrangement with exactly three triples.  By Lemma \ref{lem:6lines3triples} there is a unique subarrangement $\mathcal{B}$ satisfying this.  

\begin{lemma}[12.B.3.]
\label{lem:12non3}
For the non-reduction of twelve triples where three of the six doubles are triples, there are four combinatorial configurations, three of which are geometric configurations.  Three of these give potential Zariski pair.   
  Configuration tables are given in Table \ref{tab:12non3} and equations are given in Table \ref{tab:12.B.3.EQNS1}.
\end{lemma}

\begin{proof}
Let $e_1,e_2,e_3$ be the three triples coming from doubles as mentioned above.  The two remaining cases not handled by the reduction above are (a) when these three points are colinear and (b) when they are not but do not form a triangle.  In both we may assume $e_1$ is the intersection $L_1\cap L_2$ and $e_2$ is the intersection $L_1\cap L_3$.  In (a) we take $e_3$ to be the intersection $L_1\cap L_{10}$, and in (b) we take $e_3$ to be the intersection $L_2\cap L_{10}$. 

By re-ordering the remaining triple points, we may say that in both the line $L_{2}$ contains $e_5$, the line $L_{3}$ contains $e_6$ and $e_7$, and the line $L_{10}$ contains $e_8$ and $e_9$ as in Table \ref{tab:12non3}.  Note that the element $e_4$ is contained in line $L_2$ for (a) but line $L_1$ for (b).

Note that there are now three remaining triple points $e_{10},e_{11},e_{12}$ unaccounted for in the remaining six lines.
  By Lemma \ref{lem:6lines3triples}, there is a unique subarrangement $\mathcal{B}$ of six lines with these three triples. 
  We take the lines $L_4,L_5,L_6$ to be those with exactly two of these three triples as in Table \ref{tab:12non3}.  



On the subarrangement $\mathcal{B}$ there are six doubles.  Each of the lines $L_7,L_8,L_9$ contains three of these, so that all four triples are accounted for by multiple points of $\mathcal{B}$.  However, each of the lines $L_4,L_5,L_6$ contains just one of these doubles, and so the three triples $e_1,e_2,e_3$ from the initial subarrangement must appear on these lines.  
  Without loss of generality we may assume $e_1$ is on $L_4$, $e_2$ is on $L_5$, and $e_3$ is on $L_6$.

We now consider the three remaining triples on the lines $L_4,L_5,L_6$.

In (a) there is only one choice up to the symmetries $(e_4 e_5)$, $(e_6 e_7)$, and $(e_8 e_9)$, as well as the symmetries between the lines $L_2,L_3,L_{10}$.  This gives only one arrangement, the arrangement 12.B.3.a.i. in Table \ref{tab:12non3}.


A similar arrangement holds for (b) with no overlapping of elements $e_4$ and $e_5$ despite them not being colinear anymore, giving the arrangement 12.B.3.b.i. in Table \ref{tab:12non3}.

Finally for (b) we take the triples $e_4$ and $e_5$ to be colinear on the line $L_7$, 
and we may assume up to symmetry that the final triple on the line is $e_6$.  Then up to symmetry of the triples $(e_8 e_9)$ and $(e_{11} e_{12})$, this forces the positions of the triples $e_8$ and $e_9$ on the remaining lines.  It must be the case that the line $L_6$ also contains the triple $e_6$, and this gives two choices for the positions of the triples $e_4$ and $e_5$.  Up to symmetry this gives two arrangements, the arrangements { 12.B.3.b.ii.} and { 12.B.3.b.iii.}, which differ only by the placement of $e_4,e_5$ on lines $L_8,L_9$, in Table \ref{tab:12non3}.
\end{proof}

\begin{table}[htbp]
{ \begin{tabular}{cc}

\begin{tabular}{|c|ccc|ccc|ccc|}
\hline
$L_1$ &	$L_2$ &	$L_3$ & $L_{10}$	& $L_4$ & $L_5$ &	$L_6$ &	$L_7$ & $L_8$ &	$L_9$\\
\hline
$e_1$ & $e_1$ & $e_2$ & $e_3$	&	$e_1$ 	 & $e_2$ 		& $e_3$ 	 & $e_{10}$	&	$e_{11}$ & $e_{12}$\\
$e_2$ & $e_4$ & $e_6$ & $e_8$	&	$e_{10}$ & $e_{10}$ & $e_{11}$ & $e_{4}$	&	$e_{5}$ & $e_{4}$\\
$e_3$ & $e_5$ & $e_7$ & $e_9$	&	$e_{11}$ & $e_{12}$ & $e_{12}$ & $e_{7}$	&	$e_{6}$ & $e_{6}$\\
		  & 		  & 		  & 			&	$e_9$ 	 & $e_5$ 		& $e_7$ 	 & $e_{8}$	&	$e_{8}$ & $e_{9}$\\
\hline
\end{tabular}

 &

\begin{tabular}{|cc|cc|ccc|ccc|}
\hline
$L_1$ &	$L_2$ &	$L_3$ & $L_{10}$	& $L_4$ & $L_5$ &	$L_6$ &	$L_7$ & $L_8$ &	$L_9$\\
\hline
$e_1$ & $e_1$ & $e_2$ & $e_3$	&	$e_1$ 	 & $e_2$ 		& $e_3$ 	 & $e_{10}$	&	$e_{11}$ & $e_{12}$\\
$e_2$ & $e_3$ & $e_6$ & $e_8$	&	$e_{10}$ & $e_{10}$ & $e_{11}$ & $e_{4}$	&	$e_{5}$ & $e_{4}$\\
$e_4$ & $e_5$ & $e_7$ & $e_9$	&	$e_{11}$ & $e_{12}$ & $e_{12}$ & $e_{7}$	&	$e_{6}$ & $e_{6}$\\
		  & 		  & 		  & 			&	$e_9$ 	 & $e_5$ 		& $e_7$ 	 & $e_{8}$	&	$e_{8}$ & $e_{9}$\\
\hline
\end{tabular}
\\

 &  \\

The ($\mathcal{Z}$) arrangement 12.B.3.a.i. & The non-geometric arrangement 12.B.3.b.i.\\

 &  \\

\begin{tabular}{|cc|cc|ccc|ccc|}
\hline
$e_1$ & $e_1$ & $e_2$ & $e_3$	&	$e_1$ 	 & $e_2$ 		& $e_3$ 	 & $e_{10}$	&	$e_{11}$ & $e_{12}$\\
$e_2$ & $e_3$ & $e_6$ & $e_8$	&	$e_{10}$ & $e_{10}$ & $e_{11}$ & $e_{4}$	&	$e_{7}$ & $e_{7}$\\
$e_4$ & $e_5$ & $e_7$ & $e_9$	&	$e_{11}$ & $e_{12}$ & $e_{12}$ & $e_{5}$	&	$e_{9}$ & $e_{8}$\\
		  & 		  & 		  & 			&	$e_8$ 	 & $e_9$ 		& $e_6$ 	 & $e_{6}$	&	$e_{4}$ & $e_{5}$\\
\hline
\end{tabular}

&

\begin{tabular}{|cc|cc|ccc|ccc|}
\hline
$e_1$ & $e_1$ & $e_2$ & $e_3$	&	$e_1$ 	 & $e_2$ 		& $e_3$ 	 & $e_{10}$	&	$e_{11}$ & $e_{12}$\\
$e_2$ & $e_3$ & $e_6$ & $e_8$	&	$e_{10}$ & $e_{10}$ & $e_{11}$ & $e_{4}$	&	$e_{7}$ & $e_{7}$\\
$e_4$ & $e_5$ & $e_7$ & $e_9$	&	$e_{11}$ & $e_{12}$ & $e_{12}$ & $e_{5}$	&	$e_{9}$ & $e_{8}$\\
		  & 		  & 		  & 			&	$e_8$ 	 & $e_9$ 		& $e_6$ 	 & $e_{6}$	&	$e_{5}$ & $e_{4}$\\
\hline
\end{tabular}

\\

 &  \\

The ($\mathcal{Z}$) arrangement {12.B.3.b.ii.} & The ($\mathcal{Z}$) arrangement {12.B.3.b.iii.}\\

 &  \\

\end{tabular}
}
\caption{Arrangements with twelve triples where three of six doubles are triples, not arising from the reduction.}
\label{tab:12non3}
\end{table}


\begin{table}[htbp]\renewcommand{\arraystretch}{2.25}
\begin{tabular}{|c||c|c|c|c|c|c|}
\hline
 \multicolumn{7}{|c|}{12.B.3. } \\
\hline
\hline
Arr. & $y=0,z,bz$ & $x=0,z,az$ & $y=Ax$ &	$y=Bx+z$ &	$y=C(x-z)$ & $y=D(x-z)+z$ \\
\hline
\hline
a.i. & $L_7,L_5,L_4$	& $L_8,L_9,L_3$	& $L_{10}$:  $b$ & $L_2$:  $-1$	&	{} &	$L_6$:  $\frac{-b}{a}$ \\
\hline
 & \multicolumn{6}{|c|}{with $L_1: y=\frac{1-b}{a+b-1}(x-az)+z$ and satisfying $a(b-1)-b=0$ and $b^3-3b^2+2b-1=0^\mathcal{Z}$}\\
\hline
\hline
b.i. & $L_7,L_5,L_4$	& $L_8,L_9,L_3$	& $L_{10}$:  $b$ & $L_2$: $\frac{b-1}{ab-b+1}$ &	$L_1$: $\frac{1}{a-1}$ &	$L_6$:  $-\frac{b}{a}$ \\
\hline
 & \multicolumn{6}{|c|}{and satisfying $a(b-1)-b=0$ and $b-1=0$, a contradiction }\\
\hline
\hline
{b.ii.} & $L_7,L_5,L_4$	& $L_8,L_9,L_3$	& $L_1$:  $\frac{1}{a}$ & $L_{10}$: $b-1$ &	$L_2$: $\frac{b}{ab-1}$ &	$L_6$:  $-\frac{b}{a}$ \\
\hline
 & \multicolumn{6}{|c|}{and satisfying $a(b-1)-b=0$ and $b^3+b^2-b+1=0^\mathcal{Z}$ }\\
\hline
\hline
{b.iii.} & $L_7,L_5,L_4$	& $L_8,L_9,L_3$	& $L_2$:  $\frac{b}{a+1}$ & $L_{10}$: $b-1$ &	$L_1$: $\frac{1}{a-1}$ &	$L_6$: $\frac{-b}{a}$ \\
\hline
 & \multicolumn{6}{|c|}{and satisfying $b(a-1)-a=0$ and $2a^2-1=0^\mathcal{Z}$ } \\
\hline
\end{tabular}
\caption{Equations for (geometric) arrangements 12.B.3.}
\label{tab:12.B.3.EQNS1}
\end{table}

%
%
%
%


Suppose two of the six doubles are triples in $\mathcal{A}$.  Some of these cases were handled in Subsection \ref{subsec:12red}; the ones that were not have these two doubles not colinear.

\begin{lemma}[12.B.2.]
\label{lem:12non2}
For the non-reduction of twelve triples where two of the six doubles are triples, there are four combinatorial configurations, two of which are geometric, and one of these represents a potential Zariski pair.  
  Configuration tables are given in Table \ref{tab:12non2} and equations are given in Table \ref{tab:12.B.2.EQNS1}.
\end{lemma}

\begin{proof}
Let $e_1$ be the intersection $L_1\cap L_2$ with $e_2,e_3$ on $L_1$ and $e_4,e_5$ on $L_2$.  Similarly, let $e_6$ be the intersection $L_3\cap L_{10}$ with $e_7,e_8$ on $L_3$ and $e_9,e_{10}$ on $L_{10}$.  Then the remaining triples $e_{11}$ and $e_{12}$ must lie on the six remaining lines.

First suppose that these triples $e_{11}$ and $e_{12}$ are not colinear.  Then we may assume that the first three remaining lines $L_4,L_5,L_6$ contain the triple $e_{11}$ and that the last three remaining lines $L_7,L_8,L_9$ contain the triple $e_{12}$.  Each of these remaining lines contains three additional triples.

None of these six remaining lines may contain more than two of the triples $e_1,e_2,e_3,e_4,e_5$ or more than two of the triples $e_6,e_7,e_8,e_9,e_{10}$, and so each line must contain one triple from the first set and two triples from the second set or vice versa.


\begin{claim*}
\label{claim:12B2i}
Up to symmetry this gives only one arrangement, the arrangement 12.B.2.i. in Table \ref{tab:12non2}.
\end{claim*}

\begin{proof}
We may assume that the line $L_9$ does not contain the triples $e_1,e_2,e_3$ by the symmetry $(e_{11} e_{12})$ and by re-ordering the lines $L_4$ through $L_9$.  Furthermore we may assume the line $L_4$ contains the triple $e_1$ and that $e_2=L_5\cap L_7$ and $e_3=L_6\cap L_8$.

By the assertion just above the claim, the line $L_9$ must contain one of the triples $e_4$ and $e_5$, say $e_4$.  Then since this line also contains the triple $e_{12}$, the other line that contains the triple $e_4$ must be $L_5$ or $L_6$.  By the symmetry $(e_2 e_3)$ we may assume it is $L_5$.  This leaves three lines $L_6,L_7,L_8$, two of which must contain the triple $e_5$.  However we already have $e_3=L_6\cap L_8$ and $e_{12}=L_7\cap L_8$, and so the lines $L_6$ and $L_7$ must contain the triple $e_5$.

This gives three lines $L_4,L_8,L_9$ that contain single triples and three lines $L_5,L_6,L_7$ that contain pairs of triples.

A similar argument with the triples $e_6,e_7,e_8,e_9,e_{10}$ gives three lines containing the pair $e_7,e_{10}$, the pair $e_7,e_9$, and the pair $e_8,e_{10}$ and three lines containing the single triples $e_6,e_8,e_9$.  Since these pairs overlap in both $e_7$ and $e_{10}$ and since $e_{12}=L_8\cap L_9$, the line $L_4$ must contain $e_7$ and $e_{10}$.

We may assume up to the symmetry $(e_7 e_{10})(e_8 e_9)$ that the line $L_8$ contains the triple $e_7$ (and therefore $e_9$) and the line $L_9$ contains the triple $e_{10}$ (and therefore $e_8$).  The line $L_7$ must contain the triple $e_6$ since $e_{12}=L_8\cap L_9$, and finally the line $L_5$ must contain $e_9$ since $e_3=L_6\cap L_8$ and the line $L_6$ must contain $e_8$ since $e_4=L_5\cap L_9$.
\end{proof}

Now supposing that there is a line $L_4$ containing both $e_{11}$ and $e_{12}$, then there must be some other of the six lines $L_5$ not containing either.  Since any line can contain at most two elements from each the first two multisets $\{e_1,e_2,e_2,e_3,e_3,e_4,e_4,e_5,e_5\}$ and $\{e_6,e_7,e_7,e_8,e_8,e_9,e_9,e_{10},e_{10}\}$, and since this last line $L_5$ must contain four elements, we may assume that $e_2,e_4,e_7,e_9$ are on $L_5$.

None of the triples $e_3,e_5,e_8,e_{10}$ can be on the line $L_4$ because all the remaining lines contain one of the triples $e_{11}$ and $e_{12}$.  { At most one of the triples $e_2,e_4,e_7,e_9$ can be on this line because they are colinear on the line $L_5$.  This gives two possibilities for the remaining elements on the line $L_4$:  either both $e_1,e_6$ or just one of these, say $e_1$, with one of the other set, say $e_7$, up to the appropriate symmetry.

First suppose that both the triples $e_1,e_6$ are on the line $L_4$.  Again, since the triples $e_2,e_4,e_7,e_9$ are already colinear on the line $L_5$, they must appear on the remaining four lines $L_6,L_7,L_8,L_9$.  Up to symmetry of the triples ($e_{11} e_{12}$) we have only two possibilities:  when both $e_2,e_4$ are colinear with the same triple of $e_{11},e_{12}$ or when they are not.  Each of these gives just a single arrangement, giving two arrangements  in total:  the arrangements 12.B.2.ii. and 12.B.2.iii. in Table \ref{tab:12non2}.

If on the other hand the triples $e_1,e_7$ are on the line $L_4$, then the triple $e_6$ must appear on one of the remaining four lines, and we may assume it is $L_6$ with the triple $e_{11}$ up to symmetry.  These four lines must contain exactly two copies of each of the triples $e_3,e_5,e_8,e_{10},e_{11},e_{12}$, and since this accounts for all $\binom{4}{2}=6$ possibilities, the line $L_4$ must contain three of these including $e_{11}$.  Thus the remaining two triples on this line must be $e_3$ and $e_5$ since the triples $e_8$ and $e_{10}$ are already colinear with $e_6$.  Up to the symmetry ($e_2 e_4$)($e_3 e_5$) and ($L_1 L_2)$ this gives just one arrangement, the arrangement 12.B.2.iv. in Table \ref{tab:12non2}.
}

Two of these four arrangements are non-geometric.  The equations given in Table \ref{tab:12.B.2.EQNS1} demonstrate that one of the two remaining geometric arrangements has an irreducible moduli space while the other is a potential Zariski pair, having three roots and appearing on the final list of Theorem \ref{thm:Zariskipairs}.
\end{proof}

\begin{table}[htbp]
{\begin{tabular}{cc}

\begin{tabular}{|cc|cc|ccc|ccc|}
\hline
$L_1$ &	$L_2$ &	$L_3$ & $L_{10}$	& $L_4$ &	$L_5$ &	$L_6$ & $L_7$ &	$L_8$ & $L_9$\\
\hline
$e_1$ & $e_1$ & $e_6$ & $e_6$	&	    $e_{11}$ & $e_{11}$ & $e_{11}$ & $e_{12}$	&	$e_{12}$ & $e_{12}$\\
$e_2$ & $e_4$ & $e_7$ & $e_9$	&	    $e_1$  	 & $e_2$  	& $e_3$  	 & $e_{2}$	&	$e_{3}$ & $e_{4}$\\
$e_3$ & $e_5$ & $e_8$ & $e_{10}$ &	$e_7$ 	 & $e_4$    & $e_5$    & $e_{5}$	&	$e_{7}$ & $e_{10}$\\
		  & 		  & 		  & 			&	    $e_{10}$ & $e_9$ 	  & $e_8$ 	 & $e_{6}$	&	$e_{9}$ & $e_{8}$\\
\hline
\end{tabular}

&

\begin{tabular}{|cc|cc|c|c|cc|cc|}
\hline
$L_1$ &	$L_2$ &	$L_3$ & $L_{10}$	& $L_4$ &	$L_5$ &	$L_6$ & $L_7$ &	$L_8$ & $L_9$\\
\hline
$e_1$ & $e_1$ & $e_6$ & $e_6$	&	    $e_{11}$ & $e_2$ 	 & $e_{11}$ & $e_{11}$	&	$e_{12}$ & $e_{12}$\\
$e_2$ & $e_4$ & $e_7$ & $e_9$	&	    $e_{12}$ & $e_4$ 	 & $e_{3}$  & $e_{5}$		&	$e_{3}$  & $e_{5}$\\
$e_3$ & $e_5$ & $e_8$ & $e_{10}$ &	$e_{1}$ & $e_7$    & $e_8$    & $e_{10}$	&	$e_{10}$ & $e_{8}$\\
		  & 		  & 		  & 			&	    $e_{6}$ & $e_9$ 	 & $e_4$ 	  & $e_{2}$		&	$e_{7}$  & $e_{9}$\\
\hline
\end{tabular}

\\

 &  \\

The non-geometric arrangement 12.B.2.i.  & The arrangement 12.B.2.ii. \\
 &  \\

\begin{tabular}{|cc|cc|c|c|cc|cc|}
\hline
$e_1$ & $e_1$ & $e_6$ & $e_6$	&	    $e_{11}$ & $e_{2}$ & $e_{11}$ & $e_{11}$	&	$e_{12}$ & $e_{12}$\\
$e_2$ & $e_4$ & $e_7$ & $e_9$	&	    $e_{12}$ & $e_{4}$ & $e_{3}$  & $e_{8}$		&	$e_{3}$  & $e_{5}$\\
$e_3$ & $e_5$ & $e_8$ & $e_{10}$ &	$e_{1}$ & $e_7$    & $e_5$    & $e_{10}$	&	$e_{10}$ & $e_{8}$\\
		  & 		  & 		  & 			&	    $e_{6}$ & $e_9$ 	 & $e_7$ 	  & $e_{2}$		&	$e_{4}$  & $e_{9}$\\
\hline
\end{tabular}

&

\begin{tabular}{|cc|cc|c|c|cc|cc|}
\hline
$e_1$ & $e_1$ & $e_6$ & $e_6$	&	    $e_{11}$ & $e_{2}$ & $e_{11}$ & $e_{11}$	&	$e_{12}$ & $e_{12}$\\
$e_2$ & $e_4$ & $e_7$ & $e_9$	&	    $e_{12}$ & $e_{4}$ & $e_{3}$  & $e_{8}$		&	$e_{3}$  & $e_{5}$\\
$e_3$ & $e_5$ & $e_8$ & $e_{10}$ &	$e_{1}$ & $e_7$    & $e_5$    & $e_{10}$	&	$e_{10}$ & $e_{8}$\\
		  & 		  & 		  & 			&	    $e_{7}$ & $e_9$ 	 & $e_6$ 	  & $e_{2}$		&	$e_{4}$  & $e_{9}$\\
\hline
\end{tabular}

\\

 &  \\

The non-geometric arrangement 12.B.2.iii. & The {($\mathcal{Z}$)} arrangement 12.B.2.iv. \\
 &  
\end{tabular}
}
\caption{Arrangements with twelve triples where two of six doubles are triples, not arising from the reduction.}
\label{tab:12non2}
\end{table}

\begin{table}[htbp]\renewcommand{\arraystretch}{2.25}
\begin{tabular}{|c||c|c|c|c|c|c|}
\hline
 \multicolumn{7}{|c|}{12.B.2.} \\
\hline
\hline
Arr. & $y=0,z,bz$ & $x=0,z,az$ & $y=Ax$ &	$y=Bx+z$ &	$y=C(x-z)$ & $y=D(x-z)+z$ \\
\hline
\hline
i. & $L_6,L_5,L_4$	& $L_8,L_9,L_7$	& $L_1$:  $\frac{1}{a}$ & $L_{10}$:  $b-1$	&	$L_3$: $-b$ &	$L_2$:  $\frac{1}{1-a}$ \\
\hline
 & \multicolumn{6}{|c|}{satisfying $2ab-a-b+1=0$ and $2ab-a-b=0$, a contradiction}\\
\hline
\hline
ii. & $L_4,L_8,L_9$	& $L_1,L_5,L_7$	& $L_2$:  $\frac{b}{a}$ & $L_6$:  $-\frac{1}{a}$	&	{} &	$L_3$:  $\frac{b-1}{a-ab-1}$ \\
\hline
 & \multicolumn{6}{|c|}{with $L_{10}: y=\frac{1-b}{a-1}(x-z)+bz$ and satisfying $a-(b+1)=0$ and $2b-1=0$}\\
\hline
\hline
iii. & $L_{10},L_5,L_9$	& $L_4,L_6,L_7$	& $L_3$:  $1$ & {}	&	{} &	{} \\
\hline
 & \multicolumn{6}{|c|}{ $L_8: y=-x+a$, $L_1: y=\frac{a}{1-a}(x-a)+1$, $L_2: y=\frac{1-a}{a}(x-1)+a$}\\
 & \multicolumn{6}{|c|}{ and satisfying $a^2+a-1=0$, a contradiction (see Remark \ref{rem:12.B.2.iii.}) }\\ 
\hline
\hline
iv. & $L_4,L_7,L_6$	& $L_5,L_9,L_{10}$	& $L_3$:  $1$ & $L_1$: $\frac{a-1}{a^2-a+1}$ &	$L_8$: $\frac{1}{a-1}$ &	{} \\
\hline
 & \multicolumn{6}{|c|}{with $L_2: y=\frac{a^2-a+1}{a-1}(x-z)+az$ }\\
 & \multicolumn{6}{|c|}{and satisfying $a-b=0$ and $a^3-2a^2+3a-1=0^\mathcal{Z}$} \\
\hline
\end{tabular}
\caption{Equations for arrangements 12.B.2.}
\label{tab:12.B.2.EQNS1}
\end{table}

\begin{remark}
\label{rem:12.B.2.iii.}
A geometric restriction to 12.B.2.iii. forces a quintuple point to occur.  However in the non-geometric setting we need not worry about this quintuple.
\end{remark}

Finally suppose that only one or none of the six doubles are triples in $\mathcal{A}$; then all cases were handled already in {the reduction argument of} Subsection \ref{subsec:12red}.

\bigskip

This completes the proof of Theorem \ref{thm:12} for arrangements of ten lines with twelve triples.
\end{proof}

\section{Arrangements of ten lines with eleven triples}
\label{sec:11}

\begin{theorem}
\label{thm:11}
There are thirty-eight combinatorial configurations of ten lines and eleven triples, given as configuration tables in Tables \ref{tab:11central}, \ref{tab:11generic2}, \ref{tab:11generic3central}, \ref{tab:11generic3generic2}, \ref{tab:11generic3generic1a} and \ref{tab:11generic3generic1b}.  All but two of these are geometric, as shown in Tables \ref{tab:11.A.EQNS1}, \ref{tab:11.A.EQNS2}, \ref{tab:11.B.2.EQNS1}, \ref{tab:11.B.3.a.EQNS1}, \ref{tab:11.B.3.b.2.EQNS1}, \ref{tab:11generic3generic1a}, and \ref{tab:11generic3generic1b}, but there are no potential Zariski pairs.
\end{theorem}

\begin{proof}
Let $\mathcal{A}$ be arrangement of ten lines with eleven triples.  Here there are exactly three lines that have exactly four triple points on them; call these $L_1$, $L_2$, and $L_3$.  We consider the possible subarrangements of just these three lines: either they are central or generic.

\subsection{Central subarrangement}
\label{subsec:11central}

Suppose that the three lines $L_1,L_2,L_3$ form a central subarrangement, and let $e_{1}$ be the name of this common triple.

\begin{lemma}[11.A.]
\label{lem:11central}
For the central subarrangement of $L_1,L_2,L_3$ for eleven triples, there are ten 
 combinatorial configurations, all of which are geometric with irreducible moduli spaces.  Thus there are no potential Zariski pairs.  
  Configuration tables are given in Table \ref{tab:11central} and equations are given in Tables \ref{tab:11.A.EQNS1} and \ref{tab:11.A.EQNS2}.
\end{lemma}

\begin{proof}
These three lines $L_1,L_2,L_3$ contain an additional nine triples.  We may assume that the triples $e_1,e_2,e_3,e_4$ are on $L_1$, the triples $e_1,e_5,e_6,e_7$ are on $L_2$, and the triples $e_1,e_8,e_9,e_{10}$ are on $L_3$.  This leaves one additional triple $e_{11}$ not on these lines; we may assume it is on the lines $L_4,L_5,L_6$.  By Lemma \ref{grid} we may consider the first three lines as horizontal and the second three lines as vertical with the triples $e_1$ and $e_{11}$ at infinity.

The three horizontal lines contain the remaining nine triples $e_2,\ldots,e_{10}$, three on each line.  However the vertical lines only contain two additional triples each, and so only six of the nine points of the three-by-three grid from Lemma \ref{grid} can be triples.

Suppose by way of contradiction that four of these form a rectangle on the grid; by symmetry we may suppose that these lie in a close square on the lines $L_1\cup L_2$ with the triple $e_2,e_5$ on the line $L_4$ and the triples $e_3,e_6$ on the line $L_5$.  Then the line $L_3$ will have only one of the grid points as a triple.  However, this contradicts Fact \ref{fact:4col8lines} with $L_3$ containing $e_1,e_8,e_9,e_{10}$ needing eight additional lines amongst $L_1,L_2$ and $L_6,\ldots,L_{10}$.

Instead we must have, up to symmetry, the triples $e_2,e_5$ on the line $L_4$, the triples $e_3,e_8$ on the line $L_5$, and the triples $e_6,e_9$ on the line $L_6$.  Unaccounted for are the triples $e_4,e_7,e_{10}$.

If just three of the four remaining lines $L_7,\ldots,L_{10}$ contain these last three points, then up to symmetry there is just one arrangement:  the arrangement 11.A.i. in Table \ref{tab:11central}.

Otherwise these three points fall on all four remaining lines.  Suppose that two of these three points are not colinaer, say $e_4$ and $e_{10}$.  Then up to symmetry we have the triple $e_4$ on the line $L_7$, the triples $e_4,e_7$ on the line $L_8$, the triples $e_7,e_{10}$ on the line $L_9$, and the triple $e_{10}$ on $L_{10}$.  Supposing that the triples $e_3$ and $e_5$ are on the line $L_{10}$, there is just one arrangement:  the arrangement 11.A.ii. in Table \ref{tab:11central}.  Supposing that the triples $e_3$ and $e_6$ are on the line $L_{10}$, there are two arrangements:  the arrangements 11.A.iii. and 11.A.iv. in Table \ref{tab:11central}.  Supposing that the triple $e_3$ is not on the line $L_{10}$, there are two arrangements:  the arrangements 11.A.v. and 11.A.vi. in Table \ref{tab:11central}.

Suppose that one of these lines, say $L_7$, contains all three of these points.  Then up to symmetry we have the triple $e_4$ on the line $L_8$, the triple $e_7$ on the line $L_9$, and the triple $e_{10}$ on $L_{10}$.  Supposing that the triples $e_3$ and $e_6$ are on the line $L_{10}$, there are two arrangements:  the arrangements 11.A.vii. and 11.A.viii. in Table \ref{tab:11central}.  Supposing that the triples $e_3$ and $e_5$ are on the line $L_{10}$, there is just one arrangement:  the arrangement 11.A.ix. in Table \ref{tab:11central}.  Supposing that the triple $e_3$ is not on the line $L_{10}$, there is just one arrangement:  the arrangement 11.A.x. in Table \ref{tab:11central}.
\end{proof}

\begin{table}[htbp]
{\begin{tabular}{cc}

\begin{tabular}{|ccc|ccc|ccc|c|}
\hline
$L_1$ &	$L_2$ &	$L_3$ & $L_4$	& $L_5$ &	$L_6$ &	$L_7$ & $L_8$ &	$L_9$ & $L_{10}$\\
\hline
$e_1$ & $e_1$ & $e_1$ & $e_2$	&	$e_3$ & $e_6$ & $e_4$ & $e_4$	&	$e_7$ & $e_2$\\
$e_2$ & $e_5$ & $e_8$ & $e_5$	&	$e_8$ & $e_9$ & $e_7$ & $e_{10}$	&	$e_{10}$ & $e_6$\\
$e_3$ & $e_6$ & $e_9$ & $e_{11}$	&	$e_{11}$ & $e_{11}$ & $e_9$ & $e_5$	&	$e_3$ & $e_8$\\
$e_4$ & $e_7$ & $e_{10}$ & 	&	 &  &  & 	&	 & \\
\hline
\end{tabular}

&

\begin{tabular}{|ccc|ccc|cccc|}
\hline
$L_1$ &	$L_2$ &	$L_3$ & $L_4$	& $L_5$ &	$L_6$ &	$L_7$ & $L_8$ &	$L_9$ & $L_{10}$\\
\hline
$e_1$ & $e_1$ & $e_1$ & $e_2$	&	$e_3$ & $e_6$ & $e_4$ & $e_4$	&	$e_2$ & $e_3$\\
$e_2$ & $e_5$ & $e_8$ & $e_5$	&	$e_8$ & $e_9$ & $e_6$ & $e_7$	&	$e_7$ & $e_5$\\
$e_3$ & $e_6$ & $e_9$ & $e_{11}$	&	$e_{11}$ & $e_{11}$ & $e_8$ & $e_9$	&	$e_{10}$ & $e_{10}$\\
$e_4$ & $e_7$ & $e_{10}$ & 	&	 &  &  & 	&	 & \\
\hline
\end{tabular}

\\
 &  \\
The arrangement 11.A.i. & The  arrangement 11.A.ii. \\
 &  \\

\begin{tabular}{|ccc|ccc|cccc|}
\hline
$e_1$ & $e_1$ & $e_1$ & $e_2$	&	$e_3$ & $e_6$ & $e_4$ & $e_4$	&	$e_2$ & $e_3$\\
$e_2$ & $e_5$ & $e_8$ & $e_5$	&	$e_8$ & $e_9$ & $e_5$ & $e_7$	&	$e_7$ & $e_6$\\
$e_3$ & $e_6$ & $e_9$ & $e_{11}$	&	$e_{11}$ & $e_{11}$ & $e_8$ & $e_9$	&	$e_{10}$ & $e_{10}$\\
$e_4$ & $e_7$ & $e_{10}$ & 	&	 &  &  & 	&	 & \\
\hline
\end{tabular}

&

\begin{tabular}{|ccc|ccc|cccc|}
\hline
$e_1$ & $e_1$ & $e_1$ & $e_2$	&	$e_3$ & $e_6$ & $e_4$ & $e_4$	&	$e_2$ & $e_3$\\
$e_2$ & $e_5$ & $e_8$ & $e_5$	&	$e_8$ & $e_9$ & $e_5$ & $e_7$	&	$e_7$ & $e_6$\\
$e_3$ & $e_6$ & $e_9$ & $e_{11}$	&	$e_{11}$ & $e_{11}$ & $e_9$ & $e_8$	&	$e_{10}$ & $e_{10}$\\
$e_4$ & $e_7$ & $e_{10}$ & 	&	 &  &  & 	&	 & \\
\hline
\end{tabular}

\\
 &  \\
The  arrangement 11.A.iii. & The  arrangement 11.A.iv. \\
 &  \\

\begin{tabular}{|ccc|ccc|cccc|}
\hline
$e_1$ & $e_1$ & $e_1$ & $e_2$	&	$e_3$ & $e_6$ & $e_4$ & $e_4$	&	$e_3$ & $e_2$\\
$e_2$ & $e_5$ & $e_8$ & $e_5$	&	$e_8$ & $e_9$ & $e_5$ & $e_7$	&	$e_7$ & $e_6$\\
$e_3$ & $e_6$ & $e_9$ & $e_{11}$	&	$e_{11}$ & $e_{11}$ & $e_8$ & $e_9$	&	$e_{10}$ & $e_{10}$\\
$e_4$ & $e_7$ & $e_{10}$ & 	&	 &  &  & 	&	 & \\
\hline
\end{tabular}

&

\begin{tabular}{|ccc|ccc|cccc|}
\hline
$e_1$ & $e_1$ & $e_1$ & $e_2$	&	$e_3$ & $e_6$ & $e_4$ & $e_4$	&	$e_3$ & $e_2$\\
$e_2$ & $e_5$ & $e_8$ & $e_5$	&	$e_8$ & $e_9$ & $e_5$ & $e_7$	&	$e_7$ & $e_6$\\
$e_3$ & $e_6$ & $e_9$ & $e_{11}$	&	$e_{11}$ & $e_{11}$ & $e_9$ & $e_8$	&	$e_{10}$ & $e_{10}$\\
$e_4$ & $e_7$ & $e_{10}$ & 	&	 &  &  & 	&	 & \\
\hline
\end{tabular}

\\
 &  \\
The  arrangement 11.A.v. & The  arrangement 11.A.vi. \\
 &  \\

\begin{tabular}{|ccc|ccc|c|ccc|}
\hline
$e_1$ & $e_1$ & $e_1$ & $e_2$	&	$e_3$ & $e_6$ & $e_4$ & $e_4$	&	$e_2$ & $e_3$\\
$e_2$ & $e_5$ & $e_8$ & $e_5$	&	$e_8$ & $e_9$ & $e_7$ & $e_5$	&	$e_7$ & $e_6$\\
$e_3$ & $e_6$ & $e_9$ & $e_{11}$	&	$e_{11}$ & $e_{11}$ & $e_{10}$ & $e_8$	&	$e_9$ & $e_{10}$\\
$e_4$ & $e_7$ & $e_{10}$ & 	&	 &  &  & 	&	 & \\
\hline
\end{tabular}

&

\begin{tabular}{|ccc|ccc|c|ccc|}
\hline
$e_1$ & $e_1$ & $e_1$ & $e_2$	&	$e_3$ & $e_6$ & $e_4$ & $e_4$	&	$e_2$ & $e_3$\\
$e_2$ & $e_5$ & $e_8$ & $e_5$	&	$e_8$ & $e_9$ & $e_7$ & $e_5$	&	$e_7$ & $e_6$\\
$e_3$ & $e_6$ & $e_9$ & $e_{11}$	&	$e_{11}$ & $e_{11}$ & $e_{10}$ & $e_9$	&	$e_8$ & $e_{10}$\\
$e_4$ & $e_7$ & $e_{10}$ & 	&	 &  &  & 	&	 & \\
\hline
\end{tabular}

\\
 &  \\
The arrangement 11.A.vii. & The  arrangement 11.A.viii. \\
 &  \\

\begin{tabular}{|ccc|ccc|c|ccc|}
\hline
$e_1$ & $e_1$ & $e_1$ & $e_2$	&	$e_3$ & $e_6$ & $e_4$ & $e_4$	&	$e_2$ & $e_3$\\
$e_2$ & $e_5$ & $e_8$ & $e_5$	&	$e_8$ & $e_9$ & $e_7$ & $e_6$	&	$e_7$ & $e_5$\\
$e_3$ & $e_6$ & $e_9$ & $e_{11}$	&	$e_{11}$ & $e_{11}$ & $e_{10}$ & $e_8$	&	$e_9$ & $e_{10}$\\
$e_4$ & $e_7$ & $e_{10}$ & 	&	 &  &  & 	&	 & \\
\hline
\end{tabular}

&

\begin{tabular}{|ccc|ccc|c|ccc|}
\hline
$e_1$ & $e_1$ & $e_1$ & $e_2$	&	$e_3$ & $e_6$ & $e_4$ & $e_4$	&	$e_3$ & $e_2$\\
$e_2$ & $e_5$ & $e_8$ & $e_5$	&	$e_8$ & $e_9$ & $e_7$ & $e_5$	&	$e_7$ & $e_6$\\
$e_3$ & $e_6$ & $e_9$ & $e_{11}$	&	$e_{11}$ & $e_{11}$ & $e_{10}$ & $e_8$	&	$e_9$ & $e_{10}$\\
$e_4$ & $e_7$ & $e_{10}$ & 	&	 &  &  & 	&	 & \\
\hline
\end{tabular}

\\
 &  \\
The  arrangement 11.A.ix. & The  arrangement 11.A.x. \\
 &  

\end{tabular}
}
\caption{Arrangements with eleven triples whose distinguished three lines form a central subarrangement.  All of these have irreducible moduli spaces.}
\label{tab:11central}
\end{table}

\begin{table}[htbp]\renewcommand{\arraystretch}{2.25}
\begin{tabular}{|c||c|c|c|c|c|c|}
\hline
 \multicolumn{7}{|c|}{11.A.}\\
\hline
\hline
Arr. & $y=0,z,bz$ & $x=0,z,az$ & $y=Ax$ &	$y=Bx+z$ &	$y=C(x-z)$ & $y=D(x-z)+z$ \\
\hline
\hline
i. & $L_3,L_2,L_1$	& $L_5,L_6,L_4$	& $L_{10}$:  $1$ & {}	&	$L_7$: $\frac{a}{ac-a-c}$ &	{} \\
\hline
 & \multicolumn{6}{|c|}{with $e_{10}:(c,0)$ and $L_8: y=\frac{1}{a-c}(x-cz)$, $L_9: y=\frac{-a}{c}x+az$}\\
 & \multicolumn{6}{|c|}{and satisfying $a-b=0$ and $2c(a-1)-(a^2+a-1)=0$}\\
\hline
\hline
ii. & $L_1,L_2,L_3$	& $L_4,L_5,L_6$	& $L_9$:  $\frac{b}{a+ab-1}$ & $L_{10}$: $-1$ &	{} &	{} \\
\hline
 & \multicolumn{6}{|c|}{with $L_7: y=\frac{b-1}{1-a}(x-az)+z$, $L_8: y=\frac{b(b-1)}{1-a}(x-az)+bz$ and satisfying $a(1+b)-(2-b)=0$}\\
\hline
\hline
iii. & $L_1,L_2,L_3$	& $L_4,L_5,L_6$	& $L_9$:  $\frac{b}{a-1}$ & $L_7$: $b-1$ &	$L_{10}$: $\frac{1}{a-1}$ &	{} \\
\hline
 & \multicolumn{6}{|c|}{with $L_8: y=\frac{b(1-b)}{a-ab-1}(x-az)+bz$ and satisfying $a(1-b)-(2-b)=0$} \\
\hline
\hline
iv. & $L_1,L_2,L_3$	& $L_4,L_5,L_6$	& $L_9$:  $\frac{b}{1-a}$ & $L_7$: $\frac{b-1}{a}$ &	$L_{10}$: $\frac{1}{a-1}$ &	{} \\
\hline
 & \multicolumn{6}{|c|}{with $L_8: y=\frac{b(1-b)}{1-a-b}(x-z)+bz$ and satisfying $a(1+b)-b=0$} \\
\hline
\hline
v. & $L_1,L_2,L_3$	& $L_4,L_5,L_6$	& $L_{10}$:  $\frac{1}{a}$ & $L_7$: $b-1$ &	$L_9$: $\frac{b}{a-b-1}$ &	{} \\
\hline
 & \multicolumn{6}{|c|}{with $L_8: y=\frac{b(1-b)}{a-ab-1}(x-az)+bz$ and satisfying $a(1-b)-b=0$} \\
\hline
\end{tabular}
\caption{Equations for (geometric) arrangements 11.A.}
\label{tab:11.A.EQNS1}
\end{table}

\begin{table}[htbp]\renewcommand{\arraystretch}{2.25}
\begin{tabular}{|c||c|c|c|c|c|c|}
\hline
 \multicolumn{7}{|c|}{11.A. continued}\\
\hline
\hline
Arr. & $y=0,z,bz$ & $x=0,z,az$ & $y=Ax$ &	$y=Bx+z$ &	$y=C(x-z)$ & $y=D(x-z)+z$ \\
\hline
\hline
vi. & $L_1,L_2,L_3$	& $L_4,L_5,L_6$	& $L_{10}$:  $\frac{1}{a}$ & $L_7$: $\frac{b-1}{a}$ &	$L_9$: $\frac{b}{1-a-b}$ &	{} \\
\hline
 & \multicolumn{6}{|c|}{with $L_8: y=\frac{b(1-b)}{1-a-b}(x-z)+bz$ and satisfying $  a(b+1) - (2-b)=0$} \\ 
\hline
\hline
vii. & $L_1,L_2,L_3$	& $L_4,L_5,L_6$	& $L_9$:  $\frac{b}{a}$ & $L_8$: $b-1$ &	$L_{10}$: $\frac{1}{a-1}$ &	{} \\
\hline
 & \multicolumn{6}{|c|}{with $L_7: y=\frac{b(1-b)}{a-b-ab}(x-\frac{1}{1-b}z)$ and satisfying $a(b-1)-b=0$} \\
\hline
\hline
viii. & $L_1,L_2,L_3$	& $L_4,L_5,L_6$	& $L_9$:  $b$ & $L_8$: $\frac{b-1}{a}$ &	$L_{10}$: $\frac{1}{a-1}$ &	{} \\
\hline
 & \multicolumn{6}{|c|}{with $L_7: y=\frac{b(1-b)}{1-b-ab}(x-\frac{a}{1-b}z)$ and satisfying $a(b-1)-b=0$} \\
\hline
\hline
ix. & $L_1,L_2,L_3$	& $L_4,L_5,L_6$	& $L_9$:  $\frac{b}{a}$ & $L_{10}$: $-1$ &	{} &	{} \\
\hline
 & \multicolumn{6}{|c|}{with $L_7: y=\frac{b(b-1)}{b-b^2-a}(x-\frac{1-ab}{1-b}z)$, $L_8: y=\frac{1-b}{a-1}(x-z)+bz$ and satisfying $a(b-1)-b=0$} \\
\hline
\hline
x. & $L_1,L_2,L_3$	& $L_4,L_5,L_6$	& $L_{10}$:  $\frac{1}{a}$ & $L_8$: $b-1$ &	$L_9$: $\frac{b}{a-1}$ &	{} \\
\hline
 & \multicolumn{6}{|c|}{with $L_7: y=\frac{b(1-b)}{ab-ab^2-1}(x-\frac{1}{1-b}z)$ and satisfying $a(b-1)-b=0$} \\
\hline
\end{tabular}
\caption{Equations for (geometric) arrangements 11.A.}
\label{tab:11.A.EQNS2}
\end{table}

\subsection{Generic subarrangement}
\label{subsec:11generic}

Now suppose the three lines $L_1,L_2,L_3$ that each contain four triples are not central but form a generic subarrangement giving three doubles.  We proceed according to how many of these three doubles are in fact triples in $\mathcal{A}$, the original arrangement.

Suppose none of the three doubles are triples in $\mathcal{A}$; then there are four distinct triples on each of the three lines, giving a total of twelve triples, a contradiction.

Suppose one of the three doubles is a triple in $\mathcal{A}$; then these three lines contain an additional ten triples, giving a total of eleven triples.  However, the line $L_4$ also passing through this first triple must pass through two more but can only pass through one, a contradiction.

Suppose two of the three doubles are triples in $\mathcal{A}$, say $L_1\cap L_2=e_1$ and $L_1\cap L_3=e_2$; then these three lines contain an additional eight triples.  We may assume that the triples $e_3,e_4$ are on $L_1$, the triples $e_5,e_6,e_7$ are on $L_2$, and the triples $e_8,e_9,e_{10}$ are on $L_3$.  This leaves one additional triple $e_{11}$ not on these lines.  Let $L_4$ and $L_5$ be the last lines passing through the triples $e_1$ and $e_2$, respectively.  Then they must intersect at $e_{11}$ or else they could not contain three triples, and we may assume up to symmetry of the triples above that the triple $e_5$ is on the line $L_5$ with triple $e_2$ and that the triple $e_8$ is on the line $L_4$ with triple $e_1$.  We may then assume that the triples $e_5,e_6,e_6,e_7,e_7$ lie on the lines $L_6,L_7,L_8,L_9,L_{10}$, respectively.

\begin{lemma}[11.B.2.]
\label{lem:11generic2}
For the generic subarrangement of $L_1,L_2,L_3$ for eleven triples where two of the three doubles are triples, there are four combinatorial configurations, all of which are geometric and have irreducible or complex conjugate moduli spaces.  Thus there are no potential Zariski pairs.  
  Configuration tables are given in Table \ref{tab:11generic2} and equations are given in Table \ref{tab:11.B.2.EQNS1}.
\end{lemma}

\begin{proof}
Suppose first that the triples $e_5$ and $e_8$ are colinear on the line $L_6$.  Then without loss of generality we may assume that the triples  $e_9,e_{10},e_9,e_{10}$ are on the lines $L_7,L_8,L_9,L_{10}$, respectively.  Since the triple $e_{11}$ is already colinear with triples $e_5$ and $e_8$ it cannot be on the line $L_6$, and up to symmetry of triples $e_3$ and $e_4$ we may assume that the triple $e_3$ is on the line $L_6$.  Then up to symmetry of triples $e_9$ and $e_{10}$ we may assume that the triple $e_4$ is on both lines $L_7$ and $L_{10}$.  This gives just one arrangement, the arrangements { 11.B.2.i.}. in Table \ref{tab:11generic2}.

\begin{table}[htbp]
{ \begin{tabular}{cc}
\begin{tabular}{|ccc|cc|c|cccc|}
\hline
$L_1$ &	$L_2$ &	$L_3$ & $L_4$	& $L_5$ &	$L_6$ &	$L_7$ & $L_8$ &	$L_9$ & $L_{10}$\\
\hline
$e_1$ & $e_1$ & $e_2$ & $e_1$			&	$e_2$ 	 & $e_5$ & $e_6$ & $e_6$		&	$e_7$ & $e_7$\\
$e_2$ & $e_5$ & $e_8$ & $e_8$			&	$e_5$ 	 & $e_8$ & $e_9$ & $e_{10}$	&	$e_9$ & $e_{10}$\\
$e_3$ & $e_6$ & $e_9$ & $e_{11}$	&	$e_{11}$ & $e_3$ & $e_4$ & $e_{11}$	&	$e_3$ & $e_4$\\
$e_4$ & $e_7$ & $e_{10}$ & 	&	 & 	&  & 	&	 & \\
\hline
\end{tabular}

 &

\begin{tabular}{|ccc|cc|ccccc|}
\hline
$L_1$ &	$L_2$ &	$L_3$ & $L_4$	& $L_5$ &	$L_6$ &	$L_7$ & $L_8$ &	$L_9$ & $L_{10}$\\
\hline
$e_1$ & $e_1$ & $e_2$ & $e_1$			&	$e_2$ 	 & $e_5$ 		& $e_6$ 	 & $e_6$		&	$e_7$ & $e_7$\\
$e_2$ & $e_5$ & $e_8$ & $e_8$			&	$e_5$ 	 & $e_{10}$ & $e_{10}$ & $e_9$		&	$e_9$ & $e_8$\\
$e_3$ & $e_6$ & $e_9$ & $e_{11}$	&	$e_{11}$ & $e_3$ 		& $e_{11}$ & $e_4$		&	$e_3$ & $e_4$\\
$e_4$ & $e_7$ & $e_{10}$ & 	&	 & 	&  & 	&	 & \\
\hline
\end{tabular}

\\

 &  \\

The arrangement 11.B.2.i. & 
The {($*$)} arrangement 11.B.2.ii. 
\\

 &  \\

\begin{tabular}{|ccc|cc|cc|c|cc|}
\hline
$e_1$ & $e_1$ & $e_2$ & $e_1$			&	$e_2$ 	 & $e_5$ 		& $e_6$ 	 & $e_6$		&	$e_7$ & $e_7$\\
$e_2$ & $e_5$ & $e_8$ & $e_8$			&	$e_5$ 	 & $e_{10}$ & $e_{10}$ & $e_9$		&	$e_9$ & $e_8$\\
$e_3$ & $e_6$ & $e_9$ & $e_{11}$	&	$e_{11}$ & $e_3$ 		& $e_4$ 	 & $e_{11}$	&	$e_3$ & $e_4$\\
$e_4$ & $e_7$ & $e_{10}$ & 	&	 & 	&  & 	&	 & \\
\hline
\end{tabular}

&
\begin{tabular}{|ccc|cc|cc|c|cc|}
\hline
$e_1$ & $e_1$ & $e_2$ & $e_1$			&	$e_2$ 	 & $e_5$ 		& $e_6$ 	 & $e_6$		&	$e_7$ & $e_7$\\
$e_2$ & $e_5$ & $e_8$ & $e_8$			&	$e_5$ 	 & $e_{10}$ & $e_{10}$ & $e_9$		&	$e_9$ & $e_8$\\
$e_3$ & $e_6$ & $e_9$ & $e_{11}$	&	$e_{11}$ & $e_3$ 		& $e_4$ 	 & $e_{11}$	&	$e_4$ & $e_3$\\
$e_4$ & $e_7$ & $e_{10}$ & 	&	 & 	&  & 	&	 & \\
\hline
\end{tabular}
\\

 &  \\

The non-geometric arrangement 11.B.2.iii. & The {($\mathbb{C}$)} arrangement 11.B.2.iv.\\

 &  \\

\end{tabular}
}
\caption{Arrangements with eleven triples whose distinguished three lines form a generic subarrangement where two of the three doubles are triples.}
\label{tab:11generic2}
\end{table}

\begin{table}[htbp]\renewcommand{\arraystretch}{2.25}
\begin{tabular}{|c||c|c|c|c|c|c|}
\hline
 \multicolumn{7}{|c|}{11.B.2.} \\
\hline
\hline
Arr. & $y=0,z,bz$ & $x=0,z,az$ & $y=Ax$ &	$y=Bx+z$ &	$y=C(x-z)$ & $y=D(x-z)+z$ \\
\hline
\hline
i. & $L_1,L_2,L_4$	& $L_7,L_9,L_3$	& $L_{10}$:  $1$ & $L_8$: $\frac{a-1}{a}$ &	$L_6$: $\frac{b}{a-1}$ &	{} \\
\hline
 & \multicolumn{6}{|c|}{with $L_{10}: y=\frac{b}{a+b-ab-1}(x-az)$ and satisfying $b(a^2-a+1)-(2a^2-2a+1)=0$} \\
\hline
\hline
ii. & $L_1,L_2,L_4$	& $L_8,L_9,L_3$	& $L_{10}$:  $1$ & $L_7$: $\frac{a-1}{a(c-a+1)}$ &	$L_{6}$:  $\frac{1}{c-1)}$ &	{} \\
\hline
 & \multicolumn{6}{|c|}{with $e_5=(c,1)$ and $L_7: y=\frac{1}{c-a}(x-az)$}\\
 & \multicolumn{6}{|c|}{and satisfying $a-b=0$ and $a^2-(c+2)a+c^2+1=0$, irreducible by hand}\\
\hline
\hline
iii. & $L_1,L_2,L_4$	& $L_6,L_7,L_3$	& $L_{10}$:  $\frac{b}{a}$ & $L_5$:  $-\frac{1}{a}$	&	$L_9$: $\frac{b}{a+b-1}$ &	$L_8$:  $\frac{1-b}{ab-a+1}$ \\
\hline
 & \multicolumn{6}{|c|}{satisfying $a+1=0$}\\
\hline
\hline
iv. & $L_1,L_2,L_4$	& $L_6,L_7,L_3$	& $L_{10}$:  $\frac{b}{a}$ & $L_5$:  $-\frac{1}{a}$	&	$L_9$: $\frac{b}{a-b}$ &	$L_8$:  $\frac{1-b}{ab-a+1}$ \\
\hline
 & \multicolumn{6}{|c|}{satisfying $a^2-a+1=0^\mathbb{C}$}\\
\hline
\end{tabular}
\caption{Equations for arrangements 11.B.2.}
\label{tab:11.B.2.EQNS1}
\end{table}

Next suppose that the triples $e_5$ and $e_8$ are not colinear with each other; and further suppose that they are not colinear with the same triple on the line $L_1$: then we may assume the triple $e_3$ is with triple $e_5$ on $L_6$ and the triple $e_4$ with trople $e_8$ on $L_{10}$.  Up to symmetry of triples ($e_6,e_7$) and ($e_9,e_{10}$) we may assume that triples $e_{10},e_{10},e_9,e_9,e_8$ lie on lines $L_6,L_7,L_8,L_9,L_{10}$, respectively.  Then the line $L_7$ cannot contain the triple $e_3$ and the line $L_9$ cannot contain the triple $e_4$.  Up to symmetry of triples ($e_3,e_4$) with ($e_5,e_8$)($e_6,e_{10}$)($e_7,e_9$), this gives just two arrangements, as in Table \ref{tab:11generic2}:  the arrangements 11.B.2.ii. and 11.B.2.iii., as determined by whether $e_{11}$ is on the line $L_7$ (which is the same as when it is on the line $L_9$) or on the line $L_8$, respectively.

Finally suppose that the triples $e_5$ and $e_8$ are not colinear with each other but that they are each colinear with the same triple, say $e_3$, on the line $L_1$.  As before we may assume that triples $e_{10},e_{10},e_9,e_9,e_8$ lie on lines $L_6,L_7,L_8,L_9,L_{10}$, respectively.  Then the triple $e_4$ must be on lines $L_7$ and $L_9$, leaving the triple $e_{11}$ on line $L_8$ as in the arrangement 11.B.2.iv. in Table \ref{tab:11generic2}.
\end{proof}

Suppose finally that all three doubles are triples in $\mathcal{A}$, say the triple $e_1$ is the intersection $L_1\cap L_2$, the triple $e_2$ is the intersection $L_1\cap L_3$, and the triple $e_3$ is the intersection $L_2\cap L_3$; then these three lines contain an additional six triples.  We may assume that the triples $e_4,e_5$ are on $L_1$, the triples $e_6,e_7$ are on $L_2$, and the triples $e_8,e_9$ are on $L_3$.  This leaves two additional triples $e_{10},e_{11}$ not on these lines.  Let $L_4,L_5,L_6$ be the last lines passing through the triples $e_1,e_2,e_3$, respectively.

The proof concludes with the consideration of a further subarrangement of the next three lines $L_4,L_5,L_6$:  first forming a central subarrangement and then a generic one.  In the latter instance, this leads to another case-by-case consideration of how many of the three doubles formed by this subarrangement are indeed triples in the original arrangement.


\subsubsection{Central subarrangement}

First suppose these three lines $L_4,L_5,L_6$ form a central subarrangement, and let $e_{10}$ be the name of this triple.

If the last triple $e_{11}$ is not on these three lines, then up to the symmetry of triples ($e_4,e_5$)($e_6,e_7$)($e_8,e_9$), we may assume that the triple $e_8$ must be on the line $L_4$, the triple $e_6$ must be on the line $L_5$, and the triple $e_4$ must be on the line $L_6$, with the triples $e_5,e_7,e_9$ on the line $L_{10}$.  This gives just one arrangement:  the  arrangement 11.B.3.a.i. in Table \ref{tab:11generic3central}. 

\begin{table}[htbp]
{ \begin{tabular}{cc}
\begin{tabular}{|ccc|ccc|ccc|c|}
\hline
$L_1$ &	$L_2$ &	$L_3$ & $L_4$ & $L_5$ &	$L_6$ &	$L_7$ & $L_8$ &	$L_9$ & $L_{10}$\\
\hline
$e_1$ & $e_1$ & $e_2$ & $e_1$			&	$e_2$ 	 & $e_3$ 		& $e_5$  	 & $e_7$		&	$e_9$ 	 & $e_5$\\
$e_2$ & $e_3$ & $e_3$ & $e_8$			&	$e_6$ 	 & $e_4$ 		& $e_6$ 	 & $e_8$		&	$e_4$ 	 & $e_7$\\
$e_4$ & $e_6$ & $e_8$ & $e_{10}$	&	$e_{10}$ & $e_{10}$ & $e_{11}$ & $e_{11}$	&	$e_{11}$ & $e_9$\\
$e_5$ & $e_7$ & $e_9$ & 	&	 & 	&  & 	&	 & \\
\hline
\end{tabular}

 &

\begin{tabular}{|ccc|ccc|cccc|}
\hline
$L_1$ &	$L_2$ &	$L_3$ & $L_4$ & $L_5$ &	$L_6$ &	$L_7$ & $L_8$ &	$L_9$ & $L_{10}$\\
\hline
$e_1$ & $e_1$ & $e_2$ & $e_1$			&	$e_2$ 	 & $e_3$ 		& $e_5$  	 & $e_7$		&	$e_5$ 	 & $e_4$\\
$e_2$ & $e_3$ & $e_3$ & $e_{10}$	&	$e_6$ 	 & $e_4$ 		& $e_6$ 	 & $e_8$		&	$e_9$ 	 & $e_7$\\
$e_4$ & $e_6$ & $e_8$ & $e_{11}$	&	$e_{10}$ & $e_{10}$ & $e_8$ 	 & $e_{11}$	&	$e_{11}$ & $e_9$\\
$e_5$ & $e_7$ & $e_9$ & 	&	 & 	&  & 	&	 & \\
\hline
\end{tabular}
\\

 &  \\

The  arrangement 11.B.3.a.i. & The  arrangement 11.B.3.a.ii.\\

 &  \\

\begin{tabular}{|ccc|ccc|cccc|}
\hline
$e_1$ & $e_1$ & $e_2$ & $e_1$			&	$e_2$ 	 & $e_3$ 		& $e_5$  	 & $e_7$		&	$e_4$ 	 & $e_5$\\
$e_2$ & $e_3$ & $e_3$ & $e_{10}$	&	$e_6$ 	 & $e_4$ 		& $e_6$ 	 & $e_8$		&	$e_9$ 	 & $e_7$\\
$e_4$ & $e_6$ & $e_8$ & $e_{11}$	&	$e_{10}$ & $e_{10}$ & $e_8$ 	 & $e_{11}$	&	$e_{11}$ 	 & $e_9$\\
$e_5$ & $e_7$ & $e_9$ & 	&	 & 	&  & 	&	 & \\
\hline
\end{tabular}

 & \\

 &  \\

The  arrangement 11.B.3.a.iii.& \\

 &  

\end{tabular}
}	
\caption{Arrangements with eleven triples whose distinguished three lines form a generic subarrangement where two of the three doubles are triples.}
\label{tab:11generic3central}
\end{table}

\begin{table}[htbp]\renewcommand{\arraystretch}{2.25}
\begin{tabular}{|c||c|c|c|c|c|c|}
\hline
Arr. & $y=0,z,bz$ & $x=0,z,az$ & $y=Ax$ &	$y=Bx+z$ &	$y=C(x-z)$ & $y=D(x-z)+z$ \\
\hline
\hline
 \multicolumn{7}{|c|}{11.B.3.a.} \\
\hline
\hline
i. & $L_1,L_2,L_4$	& $L_3,L_{10},L_9$	& $L_6$:  $-\frac{1}{a}$ & $L_5$: $\frac{b}{a(1-b)}$ &	$L_7$: $\frac{b}{a-b-ab}$ &	{} \\
\hline
 & \multicolumn{6}{|c|}{with $L_8: y=(1-b)x+bz$ }\\
 & \multicolumn{6}{|c|}{and satisfying $a^2(b^2-2b+1)-ab-(b^2-b)=0$ (see Example \ref{11.B.3.a.-irr})} \\
\hline
\hline
ii. & $L_1,L_2,L_4$	& $L_3,L_{10},L_9$	& $L_5$:  $\frac{b}{1-b}$ & $L_6$:  $-1$	&	{} &	$L_8$:  $\frac{b-1}{a-1}$ \\
\hline
 & \multicolumn{6}{|c|}{with $L_7: y=\frac{b}{1-b-ab}(x-az)$ and satisfying $a(b-1)+b=0$}\\
\hline
\hline
iii. & $L_1,L_2,L_4$	& $L_3,L_{10},L_9$	& $L_7$:  -$\frac{1}{a}$ & $L_4$: $\frac{b}{a(1-b)}$ &	$L_7$: $\frac{b}{a-b-ab}$ &	$L_8$:  $\frac{b-1}{a-1)}$ \\
\hline
 & \multicolumn{6}{|c|}{satisfying $b(a-1) - a^2=0$ }\\ 
\hline
\end{tabular}
\caption{Equations for (geometric) arrangements 11.B.3.a..}
\label{tab:11.B.3.a.EQNS1}
\end{table}

If the last triple $e_{11}$ is on one of these three lines, say $L_4$, then it cannot be on the lines $L_5$ or $L_6$ because of the triple $e_{10}$, and so we may assume the triple $e_6$ is on the line $L_5$ and that the triple $e_4$ is on the line $L_6$ by the symmetry of triples ($e_4,e_5$)($e_6,e_7$).  Then up to symmetry the last four lines $L_7,L_8,L_9,L_{10}$ contain the triples $e_8,e_8,e_9,e_9$, respectively.

If the triples $e_5$ and $e_7$ are not colinear, this gives just one arrangement up to symmetry:  the arrangement  11.B.3.a.ii. in Table \ref{tab:11generic3central}.  If they are colinear, this gives just one arrangement up to symmetry, as well:  the arrangement 11.B.3.a.iii. in Table \ref{tab:11generic3central}.

%


\subsubsection{Generic subarrangement}

Now suppose the three lines $L_4,L_5,L_6$ form a generic subarrangement, forming three doubles.  At most two of these can indeed be triples, as the first three lines already contain nine triples.  At least one of these must be a triple because a subarrangement of the last four lines cannot contain two triples by Fact \ref{fact:2trip5lines}.

\begin{lemma}[11.B.3.b.2.]
\label{lem:11generic3generic2}
For the generic subarrangement of $L_1,L_2,L_3$ for eleven triples where all three of the doubles are triples, with the generic subarrangement of $L_4,L_5,L_6$ where two of these doubles are triples, there are seven combinatorial configurations, all of which are geometric and have irreducible moduli spaces.  Thus there are no potential Zariski pairs.  
  Configuration tables are given in Table \ref{tab:11generic3generic2} and equations are given in Table \ref{tab:11.B.3.b.2.EQNS1}.
\end{lemma}

\begin{proof}
We assume that two of the three doubles formed by the generic subarrangement of $L_4,L_5,L_6$ are indeed triples, say the triple $e_{10}$ is the intersection $L_4\cap L_5$ and the triple $e_{11}$ is the triple $L_4\cap L_6$.  By the symmetry of the triples ($e_4,e_5$)($e_6,e_7$), we may assume that the triple $e_6$ is on the line $L_5$ and that the triple $e_4$ is on the line $L_6$.  Then up to the symmetry of triples ($e_8,e_9$) we may assume that the last four lines $L_7,L_8,L_9,L_{10}$ contain the triples $e_8,e_8,e_9,e_9$, respectively, as well as the triples $e_4,e_5,e_5,-$, respectively.

If the triple $e_6$ is on the line $L_7$, there are two arrangements:  the arrangements 11.B.3.b.2.i. and 11.B.3.b.2.ii. in Table \ref{tab:11generic3generic2}.  If the triple $e_6$ is on the line $L_8$, there are two arrangements:  the arrangements 11.B.3.b.2.iii. and 11.B.3.b.2.iv. in Table \ref{tab:11generic3generic2}.  If the triple $e_6$ is on the line $L_9$, there are two arrangements:  the arrangements 11.B.3.b.2.v. and 11.B.3.b.2.vi. in Table \ref{tab:11generic3generic2}.  If the triple $e_6$ is on the line $L_{10}$, there is just one arrangement:  the arrangement 11.B.3.b.2.vii. in Table \ref{tab:11generic3generic2}.


The equations for these arrangements are given in Table \ref{tab:11.B.3.b.2.EQNS1}.
\end{proof}

\begin{table}[htbp]
{ \begin{tabular}{cc}

\begin{tabular}{|ccc|ccc|cccc|}
\hline
$L_1$ &	$L_2$ &	$L_3$ & $L_4$ & $L_5$ &	$L_6$ &	$L_7$ & $L_8$ &	$L_9$ & $L_{10}$\\
\hline
$e_1$ & $e_1$ & $e_2$ & $e_1$			&	$e_2$ 	 & $e_3$ 		& $e_4$  	 & $e_5$		&	$e_5$ 	 & $e_7$\\
$e_2$ & $e_3$ & $e_3$ & $e_{10}$	&	$e_6$ 	 & $e_4$ 		& $e_6$ 	 & $e_7$		&	$e_9$ 	 & $e_9$\\
$e_4$ & $e_6$ & $e_8$ & $e_{11}$	&	$e_{10}$ & $e_{11}$ & $e_8$ 	 & $e_8$		&	$e_{10}$ & $e_{11}$\\
$e_5$ & $e_7$ & $e_9$ & 	&	 & 	&  & 	&	 & \\
\hline
\end{tabular}

 &

\begin{tabular}{|ccc|ccc|cccc|}
\hline
$L_1$ &	$L_2$ &	$L_3$ & $L_4$ & $L_5$ &	$L_6$ &	$L_7$ & $L_8$ &	$L_9$ & $L_{10}$\\
\hline
$e_1$ & $e_1$ & $e_2$ & $e_1$			&	$e_2$ 	 & $e_3$ 		& $e_4$  	 & $e_5$		&	$e_5$ 	 & $e_7$\\
$e_2$ & $e_3$ & $e_3$ & $e_{10}$	&	$e_6$ 	 & $e_4$ 		& $e_6$ 	 & $e_7$		&	$e_9$ 	 & $e_9$\\
$e_4$ & $e_6$ & $e_8$ & $e_{11}$	&	$e_{10}$ & $e_{11}$ & $e_8$ 	 & $e_8$		&	$e_{11}$ & $e_{10}$\\
$e_5$ & $e_7$ & $e_9$ & 	&	 & 	&  & 	&	 & \\
\hline
\end{tabular}
\\

 &  \\

The arrangement 11.B.3.b.2.i. & The  arrangement 11.B.3.b.2.ii.\\

 &  \\

\begin{tabular}{|ccc|ccc|cccc|}
\hline
$e_1$ & $e_1$ & $e_2$ & $e_1$			&	$e_2$ 	 & $e_3$ 		& $e_4$  	 & $e_5$		&	$e_5$ 	 & $e_7$\\
$e_2$ & $e_3$ & $e_3$ & $e_{10}$	&	$e_6$ 	 & $e_4$ 		& $e_7$ 	 & $e_6$		&	$e_9$ 	 & $e_9$\\
$e_4$ & $e_6$ & $e_8$ & $e_{11}$	&	$e_{10}$ & $e_{11}$ & $e_8$ 	 & $e_8$		&	$e_{10}$ & $e_{11}$\\
$e_5$ & $e_7$ & $e_9$ & 	&	 & 	&  & 	&	 & \\
\hline
\end{tabular}

 &

\begin{tabular}{|ccc|ccc|cccc|}
\hline
$e_1$ & $e_1$ & $e_2$ & $e_1$			&	$e_2$ 	 & $e_3$ 		& $e_4$  	 & $e_5$		&	$e_5$ 	 & $e_7$\\
$e_2$ & $e_3$ & $e_3$ & $e_{10}$	&	$e_6$ 	 & $e_4$ 		& $e_7$ 	 & $e_6$		&	$e_9$ 	 & $e_9$\\
$e_4$ & $e_6$ & $e_8$ & $e_{11}$	&	$e_{10}$ & $e_{11}$ & $e_8$ 	 & $e_8$		&	$e_{11}$ & $e_{10}$\\
$e_5$ & $e_7$ & $e_9$ & 	&	 & 	&  & 	&	 & \\
\hline
\end{tabular}
\\

 &  \\

The {($\mathbb{C}$)} arrangement 11.B.3.b.2.iii. & The {($\mathbb{C}$)} arrangement 11.B.3.b.2.iv.\\

 &  \\

\begin{tabular}{|ccc|ccc|cccc|}
\hline
$e_1$ & $e_1$ & $e_2$ & $e_1$			&	$e_2$ 	 & $e_3$ 		& $e_4$  	 & $e_5$		&	$e_5$ 	 & $e_7$\\
$e_2$ & $e_3$ & $e_3$ & $e_{10}$	&	$e_6$ 	 & $e_4$ 		& $e_7$ 	 & $e_8$		&	$e_6$ 	 & $e_9$\\
$e_4$ & $e_6$ & $e_8$ & $e_{11}$	&	$e_{10}$ & $e_{11}$ & $e_8$ 	 & $e_{10}$	&	$e_9$ 	 & $e_{11}$\\
$e_5$ & $e_7$ & $e_9$ & 	&	 & 	&  & 	&	 & \\
\hline
\end{tabular}

 &

\begin{tabular}{|ccc|ccc|cccc|}
\hline
$e_1$ & $e_1$ & $e_2$ & $e_1$			&	$e_2$ 	 & $e_3$ 		& $e_4$  	 & $e_5$		&	$e_5$ 	 & $e_7$\\
$e_2$ & $e_3$ & $e_3$ & $e_{10}$	&	$e_6$ 	 & $e_4$ 		& $e_7$ 	 & $e_8$		&	$e_6$ 	 & $e_9$\\
$e_4$ & $e_6$ & $e_8$ & $e_{11}$	&	$e_{10}$ & $e_{11}$ & $e_8$ 	 & $e_{11}$	&	$e_9$ 	 & $e_{10}$\\
$e_5$ & $e_7$ & $e_9$ & 	&	 & 	&  & 	&	 & \\
\hline
\end{tabular}
\\

 &  \\

The {($\mathcal{Z}$)} arrangement 11.B.3.b.2.v. & The  arrangement 11.B.3.b.2.vi.\\

 &  \\

\begin{tabular}{|ccc|ccc|cccc|}
\hline
$e_1$ & $e_1$ & $e_2$ & $e_1$			&	$e_2$ 	 & $e_3$ 		& $e_4$  	 & $e_5$		&	$e_5$ 	 & $e_6$\\
$e_2$ & $e_3$ & $e_3$ & $e_{10}$	&	$e_6$ 	 & $e_4$ 		& $e_7$ 	 & $e_8$		&	$e_7$ 	 & $e_9$\\
$e_4$ & $e_6$ & $e_8$ & $e_{11}$	&	$e_{10}$ & $e_{11}$ & $e_8$ 	 & $e_{10}$	&	$e_9$ 	 & $e_{11}$\\
$e_5$ & $e_7$ & $e_9$ & 	&	 & 	&  & 	&	 & \\
\hline
\end{tabular}

 & \\

 &  \\

The  arrangement 11.B.3.b.2.vii. & \\

 &  

\end{tabular}
}	
\caption{Arrangements with eleven triples whose distinguished three lines form a generic subarrangement where all three of the doubles are triples and whose next three lines form a generic subarrangement where two of the three doubles are triples.}
\label{tab:11generic3generic2}
\end{table}

\begin{table}[htbp]\renewcommand{\arraystretch}{2.25}
\begin{tabular}{|c||c|c|c|c|c|c|}
\hline
 \multicolumn{7}{|c|}{11.B.3.b.2. } \\
\hline
\hline
Arr. & $y=0,z,bz$ & $x=0,z,az$ & $y=Ax$ &	$y=Bx+z$ &	$y=C(x-z)$ & $y=D(x-z)+z$ \\
\hline
\hline
i. & $L_1,L_2,L_4$	& $L_3,L_7,L_8$	& $L_5$:  $1$ & $L_6$:  $-1$	&	{} &	{} \\
\hline
 & \multicolumn{6}{|c|}{with $L_9: y=\frac{b}{b-a}(x-az)$, $L_{10}: y=\frac{1-b}{a+b-1}(x-az)+z$ and satisfying $a(2b^2-2b-1)-(b-b^2)=0$}\\
\hline
\hline
ii. & $L_1,L_2,L_4$	& $L_3,L_7,L_8$	& $L_5$:  $1$ & $L_6$:  $-1$	&	{} &	{} \\
\hline
 & \multicolumn{6}{|c|}{with $L_9: y=\frac{b}{1-a-b}(x-az)$, $L_{10}: y=\frac{1-b}{a-b}(x-az)+z$ and satisfying $2a-1=0$}\\
\hline
\hline
iii. & $L_1,L_2,L_4$	& $L_3,L_7,L_8$	& $L_5$:  $\frac{1}{a}$ & $L_6$:  $-1$	&	{} &	$L_{10}$: $\frac{1-b}{b}$ \\
\hline
 & \multicolumn{6}{|c|}{with $L_9: y=\frac{b}{ab-a}(x-az)$ and satisfying $3b^2-3b+1=0^\mathbb{C}$}\\
\hline
\hline
iv. & $L_1,L_2,L_4$	& $L_3,L_7,L_8$	& $L_5$:  $\frac{1}{a}$ & $L_6$:  $-1$	&	{} &	$L_{10}$: $\frac{b-1}{ab-1}$ \\
\hline
 & \multicolumn{6}{|c|}{with $L_9: y=\frac{b}{1-a-b}(x-az)$ and satisfying $a^2-a+1=0^\mathbb{C}$}\\
\hline
\hline
v. & $L_1,L_2,L_4$	& $L_3,L_7,L_8$	& $L_5$:  $\frac{b}{a}$ & $L_6$:  $-1$	&	{} &	$L_{10}$: $\frac{1-b}{b}$ \\
\hline
 & \multicolumn{6}{|c|}{with $L_9: y=\frac{b}{a(1-b)}(x-az)$ and satisfying $b^2-3b+1=0^\mathcal{Z}$}\\
\hline
\hline
vi. & $L_1,L_2,L_4$	& $L_7,L_3,L_8$	& $L_6$:  $1$ & $L_{10}$:  $\frac{a-1}{c}$	&	$L_5$: $\frac{a}{c-1}$ &	{} \\
\hline
 & \multicolumn{6}{|c|}{with $e_{10}:(c,b)$ and $L_9: y=\frac{a(1-a)}{c(c-1)}(x-az)$ and }\\
 & \multicolumn{6}{|c|}{satisfying $a-b=0$ and $c^2+(a-2)c-(a-1)(a^2-a+1)=0^*$}\\
\hline
\hline
vii. & $L_1,L_2,L_4$	& $L_3,L_7,L_8$	& $L_5$:  $\frac{b}{a}$ & $L_6$:  $-1$	&	{} &	$L_9$: $\frac{1}{1-a}$ \\
\hline
 & \multicolumn{6}{|c|}{with $L_{10}: y=\frac{a+b-ab}{(a-1)(b-1)}x+\frac{a}{a-1}z$ and }\\
 & \multicolumn{6}{|c|}{satisfying $b^2-b(a^2-a+1)+a^2=0$, irreducible by hand} \\
\hline
\end{tabular}
\caption{Equations for (geometric) arrangements 11.B.3.b.2..}
\label{tab:11.B.3.b.2.EQNS1}
\end{table}

Lastly we consider when just one of the three double points of the generic subarrangement of $L_4,L_5,L_6$ is a triple.

\begin{lemma}[11.B.3.b.1.]
\label{lem:11generic3generic1}
For the generic subarrangement of $L_1,L_2,L_3$ for eleven triples where all three of the doubles are triples, with the generic subarrangement of $L_4,L_5,L_6$ where just one of these doubles is a triple, there are thirteen combinatorial configurations, twelve of which are geometric and have irreducible moduli spaces.  Thus there are no potential Zariski pairs.
  Configuration tables are given in Tables \ref{tab:11generic3generic1a} and \ref{tab:11generic3generic1b} and equations are given in Tables \ref{tab:11.B.3.b.1.EQNS1} and \ref{tab:11.B.3.b.1.EQNS2}.
\end{lemma}


\begin{proof}
We assume that only one of the three doubles formed by the generic subarrangement of $L_4,L_5,L_6$ is indeed a triple, say the triple $e_{10}$ is the intersection of the lines $L_5\cap L_6$.  The line $L_4$ must contain three triples:  since it already contains the triple $e_1$, it cannot contain any of the triples $e_2,\ldots,e_7$; it cannot contain more than one of the triples $e_8,e_9$ since these are colinear on the line $L_3$; it cannot contain the triple $e_{10}$ since we have assumed that the lines $L_4,L_5,L_6$ are generic.

Thus the line $L_4$ must contain the last triple $e_{11}$ not already accounted for.  Then assuming that $L_4\cap L_5$ and $L_4\cap L_6$ are not triples and up to symmetries ($e_4,e_5$), ($e_6,e_7$), and ($e_8,e_9$), this forces the triple $e_8$ to be on the line $L_4$, the triple $e_6$ to be on the line $L_5$, and the triple $e_4$ to be on the line $L_6$.

Because the triple $e_8$ is already colinear with each of the triples $e_9$ and $e_{11}$, and because there are only four remaining lines, the triples $e_9$ and $e_{11}$ must be colinear, say on the line $L_7$.

Suppose that the triple $e_5$ (which up to the symmetry of triples ($e_2,e_3$)($e_4,e_6$) is equivalent to the triple $e_7$) is also on this line $L_7$.  Then the line that also contains the triple $e_5$, say $L_8$, must contain the triple $e_8$, as well.  If the last triple on this line $L_8$ is the triple $e_6$, then the last two lines $L_9$ and $L_{10}$ each have $e_7$ as well as one of $e_9,e_{11}$ and one of $e_4,e_{10}$.  This gives two arrangements up to symmetry:  the arrangements 11.B.3.b.1.i. and 11.B.3.b.1.ii. in Table \ref{tab:11generic3generic1a}.  If the last triple on this line $L_8$ is the triple $e_7$, then since the triple $e_{10}$ cannot be again contained on the same line as either triples $e_4$ or $e_6$, it must be with the triple $e_7$.  This gives two arrangements up to symmetry:  the arrangements 11.B.3.b.1.iii. and 11.B.3.b.1.iv. in Table \ref{tab:11generic3generic1a}.

Finally suppose that the line $L_7$ does not contain the triples $e_5$ or $e_7$.  Then up to symmetry the final three lines $L_8,L_9,L_{10}$ must contain the triples $e_5$ and $e_7$ together, $e_5$, and $e_7$, respectively.  If the last triple on the line $L_7$ is the triple $e_4$ (which up to the symmetry above is equivalent to $e_6$), then the triple $e_6$ must be on $L_9$ and the triple $e_{10}$ must be on $L_{10}$.  This gives six arrangements, following the six permutations of the remaining triples $e_8,e_9,e_{11}$ on the three lines $L_8,L_9,L_{10}$:  the arrangements 11.B.3.b.1.v. through 11.B.3.b.1.x. in Tables \ref{tab:11generic3generic1a} and \ref{tab:11generic3generic1b}.

If the last triple on the line $L_7$ is the triple $e_{10}$, then the triple $e_6$ must be on the line $L_9$ with $e_5$ and the triple $e_4$ must be on the line $L_{10}$ with $e_7$.  Notice, however, that there still is symmetry exchanging these two last lines.  Thus this gives only three arrangements, following the choice of one of the remaining triples $e_8,e_9,e_{11}$ to be on the line $L_8$:  the arrangements 11.B.3.b.1.xi. through 11.B.3.b.1.xiii. in Table \ref{tab:11generic3generic1b}.

The equations for these geometric arrangements are given in Tables \ref{tab:11.B.3.b.1.EQNS1} and \ref{tab:11.B.3.b.1.EQNS2}.
\end{proof}

\begin{table}[htbp]
{ \begin{tabular}{cc}

\begin{tabular}{|ccc|ccc|cccc|}
\hline
$L_1$ &	$L_2$ &	$L_3$ & $L_4$ & $L_5$ &	$L_6$ &	$L_7$ & $L_8$ &	$L_9$ & $L_{10}$\\
\hline
$e_1$ & $e_1$ & $e_2$ & $e_1$			&	$e_2$ 	 & $e_3$ 		& $e_5$  	 & $e_5$		&	$e_4$ 	 & $e_7$\\
$e_2$ & $e_3$ & $e_3$ & $e_8$			&	$e_6$ 	 & $e_4$ 		& $e_9$ 	 & $e_6$		&	$e_7$ 	 & $e_{10}$\\
$e_4$ & $e_6$ & $e_8$ & $e_{11}$	&	$e_{10}$ & $e_{10}$ & $e_{11}$ & $e_8$		&	$e_9$ 	 & $e_{11}$\\
$e_5$ & $e_7$ & $e_9$ & 	&	 & 	&  & 	&	 & \\
\hline
\end{tabular}

 &

\begin{tabular}{|ccc|ccc|cccc|}
\hline
$L_1$ &	$L_2$ &	$L_3$ & $L_4$ & $L_5$ &	$L_6$ &	$L_7$ & $L_8$ &	$L_9$ & $L_{10}$\\
\hline
$e_1$ & $e_1$ & $e_2$ & $e_1$			&	$e_2$ 	 & $e_3$ 		& $e_5$  	 & $e_5$		&	$e_7$ 	 & $e_4$\\
$e_2$ & $e_3$ & $e_3$ & $e_8$			&	$e_6$ 	 & $e_4$ 		& $e_9$ 	 & $e_6$		&	$e_9$ 	 & $e_7$\\
$e_4$ & $e_6$ & $e_8$ & $e_{11}$	&	$e_{10}$ & $e_{10}$ & $e_{11}$ & $e_8$		&	$e_{10}$ & $e_{11}$\\
$e_5$ & $e_7$ & $e_9$ & 	&	 & 	&  & 	&	 & \\
\hline
\end{tabular}
\\

 &  \\

The  arrangement 11.B.3.b.1.i. & The arrangement 11.B.3.b.1.ii.\\

 &  \\

\begin{tabular}{|ccc|ccc|ccc|c|}
\hline
$e_1$ & $e_1$ & $e_2$ & $e_1$			&	$e_2$ 	 & $e_3$ 		& $e_5$  	 & $e_5$		&	$e_7$ 	 & $e_4$\\
$e_2$ & $e_3$ & $e_3$ & $e_8$			&	$e_6$ 	 & $e_4$ 		& $e_9$ 	 & $e_7$		&	$e_9$ 	 & $e_6$\\
$e_4$ & $e_6$ & $e_8$ & $e_{11}$	&	$e_{10}$ & $e_{10}$ & $e_{11}$ & $e_8$		&	$e_{10}$ & $e_{11}$\\
$e_5$ & $e_7$ & $e_9$ & 	&	 & 	&  & 	&	 & \\
\hline
\end{tabular}

 &

\begin{tabular}{|ccc|ccc|ccc|c|}
\hline
$e_1$ & $e_1$ & $e_2$ & $e_1$			&	$e_2$ 	 & $e_3$ 		& $e_5$  	 & $e_5$		&	$e_7$ 	 & $e_4$\\
$e_2$ & $e_3$ & $e_3$ & $e_8$			&	$e_6$ 	 & $e_4$ 		& $e_9$ 	 & $e_7$		&	$e_{10}$ & $e_6$\\
$e_4$ & $e_6$ & $e_8$ & $e_{11}$	&	$e_{10}$ & $e_{10}$ & $e_{11}$ & $e_8$		&	$e_{11}$ & $e_9$\\
$e_5$ & $e_7$ & $e_9$ & 	&	 & 	&  & 	&	 & \\
\hline
\end{tabular}
\\

 &  \\
The {($*$)} arrangement 11.B.3.b.1.iii. & The  arrangement 11.B.3.b.1.iv.\\
 &  \\

\begin{tabular}{|ccc|ccc|c|ccc|}
\hline
$e_1$ & $e_1$ & $e_2$ & $e_1$			&	$e_2$ 	 & $e_3$ 		& $e_4$  	 & $e_5$		&	$e_5$ 	 & $e_7$\\
$e_2$ & $e_3$ & $e_3$ & $e_8$			&	$e_6$ 	 & $e_4$ 		& $e_9$ 	 & $e_7$		&	$e_6$ 	 & $e_{10}$\\
$e_4$ & $e_6$ & $e_8$ & $e_{11}$	&	$e_{10}$ & $e_{10}$ & $e_{11}$ & $e_{8}$	&	$e_{9}$  & $e_{11}$\\
$e_5$ & $e_7$ & $e_9$ & 	&	 & 	&  & 	&	 & \\
\hline
\end{tabular}

 &

\begin{tabular}{|ccc|ccc|c|ccc|}
\hline
$e_1$ & $e_1$ & $e_2$ & $e_1$			&	$e_2$ 	 & $e_3$ 		& $e_4$  	 & $e_5$		&	$e_5$ 	 & $e_7$\\
$e_2$ & $e_3$ & $e_3$ & $e_8$			&	$e_6$ 	 & $e_4$ 		& $e_9$ 	 & $e_7$		&	$e_6$ 	 & $e_{10}$\\
$e_4$ & $e_6$ & $e_8$ & $e_{11}$	&	$e_{10}$ & $e_{10}$ & $e_{11}$ & $e_{8}$	&	$e_{11}$  & $e_{9}$\\
$e_5$ & $e_7$ & $e_9$ & 	&	 & 	&  & 	&	 & \\
\hline
\end{tabular}
\\

 &  \\
The  arrangement 11.B.3.b.1.v. & The arrangement 11.B.3.b.1.vi.\\
 &  

\end{tabular}
}	
\caption{Arrangements with eleven triples whose distinguished three lines form a generic subarrangement where all three of the doubles are triples and whose next three lines form a generic subarrangement where 
 one of the three doubles is a triple.}
\label{tab:11generic3generic1a}
\end{table}

\begin{table}[htbp]
{ \begin{tabular}{cc}

\begin{tabular}{|ccc|ccc|c|ccc|}
\hline
$L_1$ &	$L_2$ &	$L_3$ & $L_4$ & $L_5$ &	$L_6$ &	$L_7$ & $L_8$ &	$L_9$ & $L_{10}$\\
\hline
$e_1$ & $e_1$ & $e_2$ & $e_1$			&	$e_2$ 	 & $e_3$ 		& $e_4$  	 & $e_5$		&	$e_5$ 	 & $e_7$\\
$e_2$ & $e_3$ & $e_3$ & $e_8$			&	$e_6$ 	 & $e_4$ 		& $e_9$ 	 & $e_7$		&	$e_6$ 	 & $e_{10}$\\
$e_4$ & $e_6$ & $e_8$ & $e_{11}$	&	$e_{10}$ & $e_{10}$ & $e_{11}$ & $e_{9}$	&	$e_{8}$  & $e_{11}$\\
$e_5$ & $e_7$ & $e_9$ & 	&	 & 	&  & 	&	 & \\
\hline
\end{tabular}

 &

\begin{tabular}{|ccc|ccc|c|ccc|}
\hline
$L_1$ &	$L_2$ &	$L_3$ & $L_4$ & $L_5$ &	$L_6$ &	$L_7$ & $L_8$ &	$L_9$ & $L_{10}$\\
\hline
$e_1$ & $e_1$ & $e_2$ & $e_1$			&	$e_2$ 	 & $e_3$ 		& $e_4$  	 & $e_5$		&	$e_5$ 	 & $e_7$\\
$e_2$ & $e_3$ & $e_3$ & $e_8$			&	$e_6$ 	 & $e_4$ 		& $e_9$ 	 & $e_7$		&	$e_6$ 	 & $e_{10}$\\
$e_4$ & $e_6$ & $e_8$ & $e_{11}$	&	$e_{10}$ & $e_{10}$ & $e_{11}$ & $e_{9}$	&	$e_{11}$  & $e_{8}$\\
$e_5$ & $e_7$ & $e_9$ & 	&	 & 	&  & 	&	 & \\
\hline
\end{tabular}
\\

 &  \\
The {non-geometric} arrangement 11.B.3.b.1.vii. & The  arrangement 11.B.3.b.1.viii.\\
 &  \\

\begin{tabular}{|ccc|ccc|c|ccc|}
\hline
$e_1$ & $e_1$ & $e_2$ & $e_1$			&	$e_2$ 	 & $e_3$ 		& $e_4$  	 & $e_5$		&	$e_5$ 	 & $e_7$\\
$e_2$ & $e_3$ & $e_3$ & $e_8$			&	$e_6$ 	 & $e_4$ 		& $e_9$ 	 & $e_7$		&	$e_6$ 	 & $e_{10}$\\
$e_4$ & $e_6$ & $e_8$ & $e_{11}$	&	$e_{10}$ & $e_{10}$ & $e_{11}$ & $e_{11}$	&	$e_{8}$  & $e_{9}$\\
$e_5$ & $e_7$ & $e_9$ & 	&	 & 	&  & 	&	 & \\
\hline
\end{tabular}

 &

\begin{tabular}{|ccc|ccc|c|ccc|}
\hline
$e_1$ & $e_1$ & $e_2$ & $e_1$			&	$e_2$ 	 & $e_3$ 		& $e_4$  	 & $e_5$		&	$e_5$ 	 & $e_7$\\
$e_2$ & $e_3$ & $e_3$ & $e_8$			&	$e_6$ 	 & $e_4$ 		& $e_9$ 	 & $e_7$		&	$e_6$ 	 & $e_{10}$\\
$e_4$ & $e_6$ & $e_8$ & $e_{11}$	&	$e_{10}$ & $e_{10}$ & $e_{11}$ & $e_{11}$	&	$e_{9}$  & $e_{8}$\\
$e_5$ & $e_7$ & $e_9$ & 	&	 & 	&  & 	&	 & \\
\hline
\end{tabular}
\\

 &  \\
The  arrangement 11.B.3.b.1.ix. & The {($*$)} arrangement 11.B.3.b.1.x. \\
 &  \\

\begin{tabular}{|ccc|ccc|cc|cc|}
\hline
$e_1$ & $e_1$ & $e_2$ & $e_1$			&	$e_2$ 	 & $e_3$ 		& $e_9$  	 & $e_5$		&	$e_5$ 	 & $e_7$\\
$e_2$ & $e_3$ & $e_3$ & $e_8$			&	$e_6$ 	 & $e_4$ 		& $e_{10}$ & $e_7$		&	$e_6$ 	 & $e_4$\\
$e_4$ & $e_6$ & $e_8$ & $e_{11}$	&	$e_{10}$ & $e_{10}$ & $e_{11}$ & $e_{8}$	&	$e_{9}$ & $e_{11}$\\
$e_5$ & $e_7$ & $e_9$ & 	&	 & 	&  & 	&	 & \\
\hline
\end{tabular}

 &

\begin{tabular}{|ccc|ccc|cc|cc|}
\hline
$e_1$ & $e_1$ & $e_2$ & $e_1$			&	$e_2$ 	 & $e_3$ 		& $e_9$  	 & $e_5$		&	$e_5$ 	 & $e_7$\\
$e_2$ & $e_3$ & $e_3$ & $e_8$			&	$e_6$ 	 & $e_4$ 		& $e_{10}$ & $e_7$		&	$e_6$ 	 & $e_4$\\
$e_4$ & $e_6$ & $e_8$ & $e_{11}$	&	$e_{10}$ & $e_{10}$ & $e_{11}$ & $e_{9}$	&	$e_{8}$ & $e_{11}$\\
$e_5$ & $e_7$ & $e_9$ & 	&	 & 	&  & 	&	 & \\
\hline
\end{tabular}
\\

 &  \\
The arrangement 11.B.3.b.1.xi. & The  arrangement 11.B.3.b.1.xii.\\
 &  \\

\begin{tabular}{|ccc|ccc|cc|cc|}
\hline
$e_1$ & $e_1$ & $e_2$ & $e_1$			&	$e_2$ 	 & $e_3$ 		& $e_9$  	 & $e_5$		&	$e_5$ 	 & $e_7$\\
$e_2$ & $e_3$ & $e_3$ & $e_8$			&	$e_6$ 	 & $e_4$ 		& $e_{10}$ & $e_7$		&	$e_6$ 	 & $e_4$\\
$e_4$ & $e_6$ & $e_8$ & $e_{11}$	&	$e_{10}$ & $e_{10}$ & $e_{11}$ & $e_{11}$	&	$e_{8}$ & $e_{9}$\\
$e_5$ & $e_7$ & $e_9$ & 	&	 & 	&  & 	&	 & \\
\hline
\end{tabular}

 & \\
 &  \\
The  arrangement 11.B.3.b.1.xiii. & \\
 &

\end{tabular}
}	
\caption{Arrangements with eleven triples whose distinguished three lines form a generic subarrangement where all three of the doubles are triples and whose next three lines form a generic subarrangement where 
 one of the three doubles is a triple (continued).}
\label{tab:11generic3generic1b}
\end{table}

\begin{table}[htbp]\renewcommand{\arraystretch}{2.25}
\begin{tabular}{|c||c|c|c|c|c|c|}
\hline
 \multicolumn{7}{|c|}{11.B.3.b.1.}\\
\hline
\hline
Arr. & $y=0,z,bz$ & $x=0,z,az$ & $y=Ax$ &	$y=Bx+z$ &	$y=C(x-z)$ & $y=D(x-z)+z$ \\
\hline
\hline
i. & $L_2,L_1,L_4$	& $L_9,L_3,L_7$	& $L_{10}$:  $\frac{b}{a}$ & $L_6$:  $-1$	&	{} &	$L_5$:  $\frac{b-1}{b(1-a)}$ \\
\hline
 & \multicolumn{6}{|c|}{with $L_8: y=\frac{b-1}{1-a}(x-z)+bz$ and satisfying $b-(1+a-a^2)=0$}\\
\hline
\hline
ii. & $L_2,L_1,L_4$	& $L_9,L_3,L_7$	& $L_{10}$:  $\frac{b}{a}$ & {} &	$L_6$:  $\frac{b}{a-b}$ &	$L_5$:  $\frac{a}{a-b)}$ \\
\hline
 & \multicolumn{6}{|c|}{with $L_8: y=\frac{b-1}{1-a}(x-az)+z$ and satisfying $a^2b-a-b^2+b=0^*$ }\\
\hline
\hline
iii. & $L_1,L_3,L_5$	& $L_8,L_2,L_9$	& $L_7$:  $\frac{1}{a}$ & $L_4$: $-1$ &	{} &	$L_6$:  $\frac{b-1}{a-1)}$ \\
\hline
 & \multicolumn{6}{|c|}{with $L_{10}: y=\frac{b(b-1)}{a-1}(x-z)+bz$ and satisfying $a-a^2b+b^2-1=0^*$}\\
\hline
\hline
iv. & $L_2,L_1,L_4$	& $L_{10},L_3,L_7$	& $L_5$:  $1$ & $L_6$: $-1$ &	{} &	{} \\
\hline
 & \multicolumn{6}{|c|}{with $L_8: y=\frac{b-1}{1-a}(x-az)+z$ and $L_9:  y=\frac{b(b-1)}{1-a}(x-az)+bz$}\\
 & \multicolumn{6}{|c|}{and satisfying $a(2b^2-1)-(b^2+b-1)=0$} \\
\hline
\hline
v. & $L_2,L_1,L_4$	& $L_9,L_3,L_7$	& $L_5$:  $1$ & $L_8$: $b-1$ &	$L_6$: $\frac{1}{a-1}$ &	{} \\
\hline
 & \multicolumn{6}{|c|}{with $L_{10}: y=\frac{b(1-b)}{a-ab-1}(x-\frac{1}{1-b}z)$ and satisfying $a(2b-1)-(b^2+b-1)=0$} \\
\hline
\hline
vi. & $L_4,L_1,L_2$	& $L_7,L_3,L_{10}$	& $L_9$:  $\frac{b}{a+b-1}$ & $L_6$: $b-1$ &	$L_8$: $\frac{b}{a-1}$ &	$L_5$: $\frac{a(b-1)}{a-1}$ \\
\hline
 & \multicolumn{6}{|c|}{satisfying $b(a)+(a^2-3a+1)=0$} \\
\hline
\hline
vii. & $L_2,L_1,L_4$	& $L_6,L_5,L_{10}$	& $L_3$:  $1$ & $L_7$:  $\frac{b-1}{a}$ &	$L_9$: $\frac{b}{b-1}$ &	{} \\
\hline
 & \multicolumn{6}{|c|}{with $L_8: y=\frac{1}{b-a}(x-az)$ and satisfying $b-1=0$, a contradiction } \\
\hline
\hline
viii. & $L_2,L_1,L_4$	& $L_6,L_5,L_{10}$	& $L_3$:  $1$ & {} &	$L_9$: $\frac{c}{a(c-1)}$ &	{} \\
\hline
 & \multicolumn{6}{|c|}{with $e_9:(c,c)$, and $L_7: y=\frac{c-1}{c}x+1$, $L_8: y=\frac{c}{c-a}(x-az)$} \\
 & \multicolumn{6}{|c|}{and satisfying $a-b=0$ and $a^2c^2-2a^2c+a^2-ac^2+2c^2-c=0^*$ } \\
\hline
\end{tabular}
\caption{Equations for arrangements 11.B.3.b.1..}
\label{tab:11.B.3.b.1.EQNS1}
\end{table}

\begin{table}[htbp]\renewcommand{\arraystretch}{2.25}
\begin{tabular}{|c||c|c|c|c|c|c|}
\hline
 \multicolumn{7}{|c|}{11.B.3.b.1. continued }\\
\hline
\hline
Arr. & $y=0,z,bz$ & $x=0,z,az$ & $y=Ax$ &	$y=Bx+z$ &	$y=C(x-z)$ & $y=D(x-z)+z$ \\
\hline
\hline
ix. & $L_1,L_4,L_2$	& $L_3,L_7,L_{10}$	& $L_5$:  $\frac{b(1-a)}{a}$ & $L_9$: $\frac{b-1}{a-b}$ &	$L_6$: $-b$ &	$L_8$: $\frac{b-1}{a-1}$ \\
\hline
 & \multicolumn{6}{|c|}{satisfying $b(a-1)-(a^2)=0$ } \\
\hline
\hline
x. & $L_2,L_1,L_4$	& $L_6,L_5,L_{10}$	& $L_3$:  $1$ & $L_7$: $\frac{c-1}{c}$ &	$L_9$: $\frac{c}{c-1}$ &	{} \\
\hline
 & \multicolumn{6}{|c|}{with $e_9:(c,c)$, and $L_8: y=\frac{c}{2c-1-ac}(x-az)$} \\
 & \multicolumn{6}{|c|}{and satisfying $a-b=0$ and $c^2+a^2c+2ac^2-a^2c^2-4ac+a=0^*$ } \\
\hline
\hline
xi. & $L_2,L_1,L_4$	& $L_9,L_3,L_7$	& $L_5$:  $1$ & $L_8$: $b-1$ &	$L_6$: $\frac{a}{a-1}$ &	{} \\
\hline
 & \multicolumn{6}{|c|}{with $L_{10}: y=\frac{b(1-b)}{a-ab-1}(x-az)+bz$ and satisfying $b(1-2a)+(a^2-a)=0$} \\
\hline
\hline
xii. & $L_2,L_1,L_4$	& $L_8,L_3,L_7$	& $L_{10}$:  $\frac{b}{a}$ & $L_9$: $b-1$ &	$L_6$: $\frac{b}{a-b}$ &	$L_5$: $\frac{b-1}{b}$ \\
\hline
 & \multicolumn{6}{|c|}{satisfying $b(1-2a)+(a^2-a)=0$} \\
\hline
\hline
xiii. & $L_2,L_1,L_4$	& $L_{10},L_3,L_7$	& $L_8$:  $\frac{b}{a}$ & $L_6$: $-1$ &	{} &	$L_5$: $\frac{a}{1-a}$ \\
\hline
 & \multicolumn{6}{|c|}{with $L_9: y=\frac{b(1-b)}{a-b}(x-z)+bz$ and satisfying $b(1-2a)+(a^2+a-1)=0$} \\
\hline
\end{tabular}
\caption{Equations for arrangements 11.B.3.b.1..}
\label{tab:11.B.3.b.1.EQNS2}
\end{table}


This completes the proof for ten lines with eleven triples.
\end{proof}

\bibliographystyle{amsalpha}

\end{document}